\documentclass{amsart}
\usepackage[foot]{amsaddr}
\usepackage{float, graphicx}
\usepackage[]{epsfig}
\usepackage{amsmath, amsthm, amssymb,xcolor}
\usepackage{epsfig}
\usepackage{verbatim}
\usepackage{url}
\usepackage{latexsym}
\usepackage{mathrsfs}
\usepackage{hyperref}
\usepackage{graphicx}
\usepackage{enumerate}
\usepackage[normalem]{ulem}
\usepackage{bm}
\usepackage{stmaryrd}
\usepackage{enumitem}
\usepackage{mathtools}
\usepackage{color}
\usepackage[noadjust]{cite}
\usepackage{cleveref}
\usepackage[margin=1.1in]{geometry}

\numberwithin{equation}{section}

\DeclareMathOperator{\supp}{supp}
\DeclareMathOperator{\dist}{dist}

\DeclareMathOperator{\dom}{\mathcal{D}}

\DeclareMathOperator{\OO}{\mathcal{O}}
\DeclareMathOperator{\oo}{o}

\DeclareMathOperator{\argmin}{argmin}

\DeclareMathOperator{\del}{\partial}

\DeclareMathOperator{\Real}{Re}		
\DeclareMathOperator{\Imaginary}{Im}

\DeclareMathOperator{\fin}{fin}
\DeclareMathOperator{\PV}{PV}
\DeclareMathOperator{\semci}{sc}
\DeclareMathOperator{\deq}{:=}

\newtheorem{thm}{Theorem}[section]
\newtheorem{prop}[thm]{Proposition}
\newtheorem{lem}[thm]{Lemma}
\newtheorem{cor}[thm]{Corollary}

\theoremstyle{remark}
\newtheorem{rem}[thm]{Remark}
\newtheorem{remark}[thm]{Remark}
\theoremstyle{definition}
\newtheorem{definition}[thm]{Definition}
\newtheorem{assumption}[thm]{Assumption}

\newtheorem{example}[thm]{Example}

\newcommand\cA{{\mathcal A}}

\newcommand{\cD}{{\mathcal D}}
\newcommand{\cE}{{\mathcal E}}

\newcommand{\cL}{{\mathcal L}}

\newcommand{ \sfb }{{\mathsf b}}

\newcommand{\sfT}{{\mathsf T}}

\newcommand{\fa}{{\mathfrak a}}

\newcommand{\fc}{{\mathfrak c}}
\newcommand{\fC}{{\mathfrak C}}

\newcommand{\fe}{{\mathfrak e}}

\newcommand{\ft}{{\mathfrak t}}

\newcommand{\fM}{{\mathfrak M}}

\newcommand{\fB}{{\mathfrak B}}

\newcommand{\bmu}{{\bm{u}}}

\newcommand{\bmx}{{\bm{x}}}

\newcommand{\bmla}{{\bm \la}}

\newcommand{\rd}{{\rm d}}

\newcommand{\ri}{\mathrm{i}}

\newcommand{\bC}{{\mathbb C}}

\newcommand{\bE}{\mathbb{E}}
\newcommand{\bH}{\mathbb{H}}

\newcommand{\bP}{\mathbb{P}}

\newcommand{\bR}{{\mathbb R}}

\newcommand{\bZ}{\mathbb{Z}}
\newcommand{\bW}{\mathbb{W}}

\newcommand{\la}{\lambda}
\newcommand{\wt}{\widetilde}

\newcommand{\dL}{{\rm L}}

\renewcommand{\Im}{{\rm Im}}
\renewcommand{\Re}{{\rm Re}}

\title{Edge Rigidity of Dyson Brownian Motion with General Initial Data}
\author{Amol Aggarwal$^{1, 2}$}
\author{Jiaoyang Huang$^3$}
\address[1]{Department of Mathematics, Columbia University}
\address[2]{Clay Mathematics Institute}
\address[3]{Department of Statistics and Data Science, University of Pennsylvania}
\begin{document}

\maketitle

\begin{abstract}

In this paper, we study the edge behavior of Dyson Brownian motion with general $\beta$. 
Specifically, we consider the scenario where the averaged initial density near the edge, on the scale $\eta_*$, is lower bounded by a square root profile. Under this assumption, we establish that the fluctuations of extreme particles are bounded by $(\log n)^{\OO(1)}n^{-2/3}$ after time $C\sqrt{\eta_*}$. Our result improves previous edge rigidity results from \cite{landon2017edge, adhikari2020dyson} which require both lower and upper bounds of the averaged initial density. 
Additionally, combining with \cite{landon2017edge}, our rigidity estimates are used to prove that the distribution of extreme particles converges to the Tracy-Widom $\beta$ distribution in short time.
\end{abstract}

\setcounter{tocdepth}{1}
\tableofcontents

\section{Introduction}\label{s:intro}

\subsection{Preface}

In 1962 \cite{MR0148397}, Dyson interpreted the Gaussian orthogonal and unitary ensembles as dynamical limit of the matrix valued Brownian motion, which is given by
\begin{align}\label{e:MatrixDBM}
   \rd H(t)=\rd B(t),
\end{align}
where $B(t)$ is the Brownian motion on real symmetric or complex Hermitian matrices. It turns out the eigenvalues of the above matrix valued Brownian motion satisfy a system of stochastic differential equations. These equations have been later generalized to stochastic differential equations, called the $\beta$-Dyson Brownian motion,
	\begin{flalign}
		\label{e:DBM}
		\rd\lambda_i (t) = \Big( \displaystyle\frac{2}{\beta n} \Big)^{1/2} \rd B_i (t)+ \frac{1}{n}\displaystyle\sum_{\substack{1 \le j \le n \\ j \ne i}}	\displaystyle\frac{\rd t}{\lambda_i (t) - \lambda_j (t)},\quad 1\leq i\leq n,
	\end{flalign}
where  $\bmla^{(n)}(t)=\bmla(t)=(\lambda^{(n)}_1(t), \la^{(n)}_2(t),\cdots, \la_n^{(n)}(t))=(\lambda_1(t), \la_2(t), \cdots, \la_n(t))\in \overline{\mathbb{W}}_n$ lies in the closure of the Weyl chamber $\mathbb{W}_n = \{ \bm{y} \in \mathbb{R}^n : y_1 > y_2 > \cdots >  y_n \}$ and  $\{B_i(t)\}_{1\leq i\leq n}$ are independent Brownian motions.
The real symmetric and complex Hermitian matrix valued Brownian motion corresponds to \eqref{e:DBM} with $\beta=1$ and $\beta=2$ respectively.

Dyson Brownian motion is a prototype model for a one-dimension interacting particle system with singular interactions, which capture the statistical properties of energy levels in complex systems. The understanding  of Dyson Brownian motion plays an important role in the proof of Wigner's original universality conjecture by Erd{\H o}s, Schlein and Yau \cite{erdHos2011universality}. For a review about the general framework regarding the proof of the universality conjecture, one can read the book \cite{erdHos2017dynamical}.

In this paper, we investigate the edge behavior of Dyson Brownian motion (DBM) described by equation \eqref{e:DBM}, considering various initial data. If the initial data is a delta mass at zero,  the law of Dyson Brownian motion at time one corresponds to the Gaussian-$\beta$-ensemble distribution. Notably, the largest particle fluctuates exhibits fluctuations on the scale of $n^{-2/3}$, commonly referred to as optimal edge rigidity. Moreover, its fluctuations converge to the Tracy-Widom $\beta$ distribution \cite{MR2813333,MR3253704,krishnapur2016universality}. A natural question arises: \emph{when starting from general initial data, how long does it take to observe optimal edge rigidity, characterized by $n^{-2/3}$ fluctuations?}

It has been proven in \cite{landon2017edge} ($\beta=1,2$) and \cite{adhikari2020dyson} (general $\beta$), if the initial data has a spectral edge with  square root
behavior down to a scale $\eta_*$, then after time $C\sqrt{\eta_*}$, the extreme particle concentrates on the scale $n^{-2/3+\oo(1)}$. 
However, when the initial data contains an outlier, the extreme eigenvalue initially experiences fluctuations on the scale of $n^{-1/2}$, and after time $\OO(1)$, it merges with the bulk of the spectral distribution, resulting in $n^{-2/3+\oo(1)}$ fluctuations. This phenomenon is commonly referred to as the Baik--Ben Arous--P\'{e}ch\'{e} (BBP) phase transition \cite{baik2005phase, bloemendal2013limits, bloemendal2016limits}.

In this paper, we propose a general criterion that encompasses both aforementioned cases and is believed to be necessary  (up to constants). Specifically, we demonstrate that if  on the scale $\eta_*$, the averaged density near the edge is lower bounded by a square root profile (without imposing the analogous upper bound), then the fluctuations of extreme particles are bounded by $(\log n)^{\OO(1)}n^{-2/3}$ after time $C\sqrt{\eta_*}$.  Combining with \cite{landon2017edge}, our optimal edge rigidity estimates are used to give a proof of the local ergodicity of Dyson Brownian motion for more general initial data, i.e., the distribution of extreme particles converges to Tracy-Widom $\beta$ distribution in short time. 

The limiting empirical particle density of Dyson Brownian motion can be expressed as the free convolution of the initial density with the rescaled semicircle distribution, which corresponds to the solution of a complex Burgers equation \cite{anderson2010introduction}. To establish the edge rigidity, we employ a refined comparison between the empirical particle density of Dyson Brownian motion and a deterministic measure-valued process represented by the solution of the complex Burgers equation. This solution can be obtained using the method of characteristics. The Stieltjes transform of the empirical particle density satisfies the same complex Burgers equation, with the addition of a stochastic term and some small error terms. By plugging the characteristic flow of the complex Burgers equation, we can analyze the differences between the two Stieltjes transforms using estimates of Gr\"{o}nwall type. The method of characteristics has previously been employed in different contexts related to Dyson Brownian motion. In \cite{huang2019rigidity}, it was utilized to investigate the bulk behavior of Dyson Brownian motion. In \cite{adhikari2020dyson}, it was applied to study the edge behavior of Dyson Brownian motion when starting from a profile exhibiting square root behavior. Furthermore, the method of characteristics was utilized in \cite{bourgade2021extreme}
to study the extreme gaps between eigenvalues of Wigner matrices, and in  \cite{adhikari2022local},  to analyze particle rigidity in unitary Dyson Brownian motion.

Our case poses additional challenges as we only assume a lower bound for the averaged initial density near the edge, without assuming square root behavior. As a crucial initial step, we establish that if the averaged initial density near the edge satisfies this lower bound, it continues to hold within a finite time frame. To prove this, we develop a coupling technique for Dyson Brownian motions (see \Cref{l:coupling}) with different initial data. This coupling ensures that if one Dyson Brownian motion exhibits larger gaps between particles at the beginning, this property persists for all subsequent times. Notably, this coupling technique may have independent significance beyond our specific study. 

In previous work \cite{adhikari2020dyson}, a key insight relied on the observation that the characteristic flows near the edge progress at a slower rate compared to the movement of the spectral edge, which is easily obtained when the density has square root behavior. In our case, we undertake a more meticulous analysis of the equations governing the relative movement of our characteristics with respect to the spectral edge. As a result, we establish two quantitative versions (see \Cref{c:kappa} and \Cref{c:domainic}) of these estimates when the averaged density near the edge is lower bounded by a square root profile. These two estimates, combined with a stopping time argument, enable us to prove that the Stieltjes transform of the empirical particle density remains small outside the right edge of the free-convolution density. Consequently, there are no outliers, and optimal edge rigidity is guaranteed.

\subsection*{Notation.} Let $\mathbb{H} = \{ z \in \mathbb{C} : \Imaginary z > 0\}$ denote the upper half-plane and $\overline{\mathbb{H}}$ denote its closure. For any real numbers $a, b \in \mathbb{R}$ with $a < b$, we set $\llbracket a, b \rrbracket = [a, b] \cap \mathbb{Z}$. We will write $X=\OO(Y)$ to mean that $|X|\leq CY$ for some universal constant $C>0$. We also write $a\wedge b=\min \{a,b \}$ and  $a\vee b=\max \{a,b \}$.
Let $\mathscr{P}_{\fin} = \mathscr{P}_{\fin} (\mathbb{R})$ denote the set of nonnegative measures $\mu$ on $\mathbb{R}$ with finite total mass, $\mu (\mathbb{R}) < \infty$. Further let $\mathscr{P} = \mathscr{P} (\mathbb{R}) \subset \mathscr{P}_{\fin}$ denote the set of probability measures on $\mathbb{R}$; the support of any measure $\nu \in \mathscr{P}$ is denoted by $\supp \nu$. We say that a probability measure $\mu \in \mathscr{P}$ has density $\varrho$ (with respect to Lebesgue measure) if $\varrho : \mathbb{R} \rightarrow \mathbb{R}$ is a measurable function satisfying $\mu (dx) = \varrho(x) dx$. For any real number $x \in \mathbb{R}$, we let $\delta_x \in \mathscr{P}$ denote the delta function at $x$. Given two measures $\mu, \nu \in \mathscr{P}_{\fin}$ of finite total mass, the L\'{e}vy distance between them is 
	\begin{flalign}
		\label{e:defLM}
		d_{\dL} (\mu, \nu) = \inf \Bigg\{ a > 0 : \displaystyle\int_{-\infty}^{y-a} \mu (dx) - a \le \displaystyle\int_{-\infty}^y \nu (dx) \le \displaystyle\int_{-\infty}^{y+a} \mu (dx) + a, \quad \text{for all $y \in \mathbb{R}$} \Bigg\}.
	\end{flalign}
For two sets $A,B\in \bR$ we denote $\dist(A,B)=\inf_{x\in A, y\in B}|x-y|$. When $A=\{x\}$ is a single point, we simply write $\dist(\{x\},B)=\dist(x,B)$.
We denote $\overline{\bR}=\bR\cup \{\pm \infty\}$ and for any $\lambda \in \mathbb{R}$ make the convention 
\begin{align}\label{e:infinf}
+\infty> \lambda>-\infty,\quad \infty-\lambda, \lambda-(-\infty)=+\infty,\quad 
(-\infty)-\lambda, \lambda-\infty=-\infty.
\end{align}

\subsection*{Organization.} We  present the main results in the rest of \Cref{s:intro}. 
In \Cref{s:preliminary}, we collect some preliminary results on free convolution with semicircle distributions, previous results on Dyson Brownian motion and couplings of Dyson Brownian motions with different initial data. 
In \Cref{s:rigidity}, we prove our main results on the edge rigidity and universality of Dyson Brownian motion from general initial data. Finally in \Cref{s:example}, we discuss some examples and applications of our optimal rigidity result.

\subsection{Main Results}
We fix $\beta\geq 1$ and consider Dyson Brownian motion \eqref{e:DBM} with general initial data. 
We denote the empirical particle density of the $n$-particle Dyson Brownian motion $\bmla(t)$ by
\begin{align}\label{e:empiricald}
\mu_t^{(n)}=\frac{1}{n}\sum_{i=1}^n \delta_{\lambda_i(t)},\quad t\geq 0.
\end{align} 
\begin{remark}\label{r:changep}
In some cases (we will state when explicitly), we allow some particles to be at $\pm \infty$. More precisely, we allow for some $1\leq \ell \le m\leq n$ to have $\lambda_1(t)=\lambda_2(t)=\cdots\lambda_{\ell-1}(t)=+\infty$ and $\lambda_{m+1}(t)=\lambda_{m+2}(t)=\cdots\lambda_{n}(t)=-\infty$. In this case,  recalling our convention \eqref{e:infinf}, \eqref{e:DBM} is a Dyson Brownian motion with $m-\ell+1$ particles, indexed by $\llbracket \ell, m\rrbracket$. In this case, the empirical density is a measure valued process, with total mass $\mu_t^{(n)}(\bR)=(m-\ell+1)/n$. 
\end{remark}

We study $n$-particle Dyson Brownian motion, with initial data 
\begin{align}\label{e:initial}
\mu_0^{(n)}=\frac{1}{n}\sum_{i=1}^n \delta_{\lambda_i(0)}.
\end{align} 
\begin{assumption}\label{a:densitylowbound}
Fix constants $\mathsf T\geq 100$ and $\mathsf b\in (0,1)$. Take a small parameter $\eta_*=\eta_*(n) \in (0, 1/4)$ (possibly depending on $n$). Assume  the initial data $\mu_0^{(n)}$ (as in \eqref{e:initial}) is supported on $[-\infty, 0]$, satisfying $\mu_0^{(n)}([-x, 0])\geq \sfb x^{3/2}$ for all $ \eta_*\leq x\leq \sfT^2$. 
\end{assumption}

We denote by $\mu_t=\mu_0^{(n)}\boxplus \mu_{\semci}^{(t)}$ the free convolution of $\mu_0^{(n)}$ and the rescaled semicircle distribution $\mu_{\semci}^{(t)}$ (see \eqref{rho1}); see \Cref{TransformConvolution} for a more detailed definition. 
Adopt \Cref{a:densitylowbound}. We will prove in \Cref{l:densityt}, that for $(30/\sfb)\sqrt{\eta_*}\leq t\leq \sfT$, $\mu_t(\rd x)=\varrho_t(x)\rd x$ has square root behavior close to the right edge $E_t:=\sup (\supp \mu_t)$ of the support of $\mu_t$, that is,
\begin{align}\label{ctt}
\varrho_t(x)\rd x=(1+\OO(|E_t-x|)) \frac{\sqrt{\cA(t)(E_t-x)}}{\pi},\quad x\rightarrow E_t,
\end{align}
where $\cA(t)=\cA^{\mu_0}(t)$ is a scalar depending on $\mu_0$ and $t$. We refer to \Cref{l:densityt} for the explicit formula.

\begin{thm}
\label{t:main}
Adopt \Cref{a:densitylowbound}, and let $\fB= 2^{60}\sfb^{-6}>0$. Then there exists a large $C=C(\sfb, \sfT)>1$ such that the following holds with probability $1-Ce^{-(\log n)^2}$. For all $t \in [\fB\sqrt {\eta_*}, \sfT]$, we have
\begin{align}\label{e:edgerig}
\lambda_1(t)\leq E_t+ \frac{(\log n)^{15}}{n^{2/3}}.
\end{align}
\end{thm}

\begin{thm}\label{t:universality}
Adopt \Cref{a:densitylowbound}; fix real numbers $0 < \fa < 1 \le  \Theta$ and a smooth, compactly supported test function $F : \mathbb{R} \rightarrow \mathbb{R}$; and let $\fB= 2^{60}\sfb^{-6}>0$ and $\ft=\max\{2\fB\sqrt{\eta_*}, n^{-1/3+\fa}\}$. There exists a constant $\delta = \delta (\mathsf{b}, \mathsf{T}, \mathfrak{a}, \Theta) > 0$ such that the following holds for any sufficiently large integer $n$. If for each $\ft/2\leq  t\leq \sfT$ we have $\cA(t)\in [\Theta^{-1}, \Theta]$ (recall \eqref{ctt}) and  $\Theta^{-1}\sqrt{x}\leq \varrho_t(E_t-x)\leq \Theta \sqrt{x}$ for $0\leq x\leq \Theta^{-1}$, 
 then 
\begin{align}\begin{split} \label{eqn:htedgeb1}
\bE[ F (n^{2/3}\cA(t)^{1/3} ( E_t-\lambda_1))] = \bE_{\emph{Airy}_\beta}[ F (\Lambda_1) ]+\OO\left(n^{-\delta}\right),\quad \text{for any}\quad \ft\leq t\leq \sfT-\Theta,
\end{split}\end{align}

\noindent where the second expectation  is with respect to the Tracy--Widom $\beta$ distribution $TW_{\beta}$ of \cite[Theorem 1.3]{MR2813333}.
\end{thm}

The proofs of \Cref{t:main} and \Cref{t:universality} will be given in \Cref{s:rigidity}.

\subsection*{Acknowledgements} 

Amol Aggarwal was partially supported by a Packard Fellowship for Science and Engineering, a Clay Research Fellowship, by the NSF grant DMS-1926686, and by the IAS School of Mathematics. The research of Jiaoyang Huang is supported by NSF grant DMS-2054835 and DMS-2331096. 
Amol Aggarwal and Jiaoyang Huang also wish to acknowledge the NSF grant DMS-1928930, which supported their participation in the Fall 2021 semester program at MSRI in Berkeley, California titled, ``Universality and Integrability in Random Matrix Theory and Interacting Particle Systems."

\section{Preliminaries}\label{s:preliminary}
\subsection{Free Convolution With Semicircle Distributions}

	\label{TransformConvolution}
	
	In this section we recall various results concerning Stieltjes transforms and free convolutions with the semicircle distribution. To that end, fix a finite measure $\mu \in \mathscr{P}_{\fin}$. We define the \emph{Stieltjes transform} of $\mu$ to be the function $m = m^{\mu} : \mathbb{H} \rightarrow \mathbb{H}$ for any complex number $z \in \mathbb{H}$ setting
	\begin{flalign}
		\label{mz0} 
		m(z) = \displaystyle\int_{-\infty}^{\infty} \displaystyle\frac{\mu (\rd x)}{x-z}.
	\end{flalign}

We have the following estimates on the Stieltjes transform and its derivatives.
\begin{lem}\label{l:STproperty}
Let $m(z)=m_\mu(z)$ be the Stieltjes transform of a finite measure $\mu\in \mathscr P_{\rm fin}$ with $A=\mu(\bR)$. For any integer $p\geq 1$, we denote its $p$-th derivative by $m^{(p)}(z)$. Then, 
\begin{align}\label{e:stbb}
|m(z)|\leq \frac{A}{\dist(z,\supp(\mu))},\quad
\quad |m'(z)|\leq \frac{\Im[m(z)]}{\Im[z]}, 
\quad |m^{(p)}(z)|\leq \frac{p!A}{\dist(z,\supp(\mu))^{p+1}}
\end{align}
\end{lem}
\begin{proof}
The Stieltjes transform and its derivatives are given by 
\begin{align*}
|m(z)|=\left|\int_\bR \frac{\rd \mu(x)}{x-z}\right|\leq \int_\bR \frac{\rd \mu(x)}{|z-x|}\leq \frac{A}{\dist(z,\supp(\mu))},
\end{align*}
and 
\begin{align*}
|m^{(p)}(z)|=\left|\int_\bR \frac{p!\rd \mu(x)}{(x-z)^{p+1}}\right|\leq \int_\bR \frac{p!\rd \mu(x)}{|z-x|^{p+1}}\leq \frac{p!A}{\dist(z,\supp(\mu))^{p+1}}.
\end{align*}
Finally for the middle statement of \eqref{e:stbb}, we have
\begin{align*}
|m'(z)|= \left|\int_\bR \frac{\rd \mu(x)}{(x-z)^{2}}\right|\leq \int_\bR \frac{\rd \mu(x)}{|z-x|^{2}}=\frac{1}{\Im[z]}\int_\bR \frac{\Im[z]\rd \mu(x)}{|z-x|^{2}}=\frac{\Im[m(z)]}{\Im[z]}.
\end{align*} 
\end{proof}
	
	\noindent If $\mu$ has a density with respect to Lebesgue measure, that is, $\mu(\rd x) = \varrho (x) \rd x$ for some $\varrho \in L^1 (\mathbb{R})$, then $\varrho$ can be recovered from its Stieltjes transform by the identity \cite[Equation (8.14)]{FPRM}, 
	\begin{flalign}
		\label{mrho} 
		\pi^{-1} \displaystyle\lim_{y \rightarrow 0} \Imaginary m(x + \mathrm{i} y) = \varrho (x); \qquad \displaystyle\pi^{-1} \lim_{y \rightarrow 0} \Real m(x + \mathrm{i} y) = H \varrho (x),
	\end{flalign}
	
	\noindent for any $x \in \mathbb{R}$. In the latter, $Hf$ denotes the Hilbert transform of any function $f \in L^1 (\mathbb{R})$, given by 
	\begin{flalign}
		\label{transform2}
		Hf (x) = \pi^{-1} \cdot \PV \displaystyle\int_{-\infty}^{\infty} \displaystyle\frac{f(w) dw}{w-x},
	\end{flalign}
	
	\noindent where $\PV$ denotes the Cauchy principal value (assuming the integral exists as a principal value).

	The \emph{semicircle distribution} is a measure $\mu_{\semci} \in \mathscr{P} (\mathbb{R})$ whose density $\varrho_{\semci} : \mathbb{R} \rightarrow \mathbb{R}_{\ge 0}$ with respect to the Lebesgue measure is given by 
	\begin{flalign}
		\label{rho1} 
		\varrho_{\semci} (x) = \displaystyle\frac{(4-x^2)^{1/2}}{2\pi} \cdot \textbf{1}_{x \in [-2, 2]}, \qquad \text{for all $x \in \mathbb{R}$}. 
	\end{flalign}
	
	\noindent For any real number $t > 0$, we  denote the rescaled semicircle density $\varrho_{\semci}^{(t)}$ and distribution $\mu_{\semci}^{(t)} \in \mathscr{P}$ by 
	\begin{flalign}
		\label{rhosct}
		\varrho_{\semci}^{(t)} (x) = t^{-1/2} \varrho_{\semci} (t^{-1/2} x); \qquad \mu_{\semci}^{(t)} = \varrho_{\semci}^{(t)} (x) d x.
	\end{flalign}

	 We next discuss the free convolution of a probability measure $\mu \in \mathscr{P}$ with the (rescaled) semicircle distribution $\mu_{\semci}^{(t)}$. For any $t > 0$, denote the function $M = M^{\mu} = M^{t; \mu}: \mathbb{H} \rightarrow \mathbb{C}$ and the set $\Lambda_t = \Lambda_{t; \mu} \subseteq \mathbb{H}$ by (see \Cref{f:lambdat})
	\begin{flalign}
		\label{mtlambdat} 
		M(z) = z - t m(z); \quad \Lambda_t = \Big\{ z \in \mathbb{H} : \Imaginary \big( z - tm (z) \big) > 0 \Big\} = \Bigg\{ z \in \mathbb{H} : \displaystyle\int_{-\infty}^{\infty} \displaystyle\frac{\mu(d x)}{|z-x|^2} < \displaystyle\frac{1}{t} \Bigg\}.
	\end{flalign}

\begin{figure}
\centering
\includegraphics[width=0.7\textwidth]{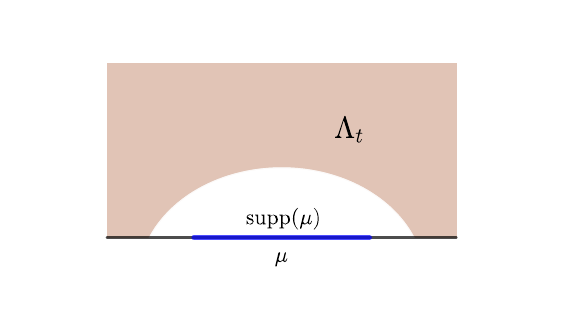}
\caption{$\Lambda_t$ as defined in \eqref{mtlambdat} is an open subset of the upper half plane $\mathbb H$. $M(z)$ is a holomorphic map from $\Lambda_t$ to $\mathbb H$. }
\label{f:lambdat}
\end{figure}

	\begin{lem}[{\cite[Lemma 4]{FCSD}}] 
		
		\label{mz} 
		
		 The function $M$ is a homeomorphism from $\overline{\Lambda}_t$ to $\overline{\mathbb{H}}$. Moreover, it is a holomorphic map from $\Lambda_t$ to $\mathbb{H}$ and a bijection from $\partial \Lambda_t$ to $\mathbb{R}$. 
	
	\end{lem} 

	For any real number $t \ge 0$, define $m_t = m_t^{\mu} : \mathbb{H} \rightarrow \mathbb{H}$ as follows. First set $m_0 (z) = m(z)$; for any real number $t > 0$, define $m_t$ so that	
	\begin{flalign}
		\label{mt} 
		m_t \big( z - t m_0 (z) \big) = m_0 (z), \qquad \text{for any $z \in \Lambda_t$}.
	\end{flalign}
	
	\noindent Since by \Cref{mz} the function $M(z) = z - tm_0 (z)$ is a bijection from $\Lambda_t$ to $\mathbb{H}$, \eqref{mt} defines $m_t$ on $\mathbb{H}$. By \cite[Proposition 2]{FCSD}, $m_t$ is the Stieltjes transform of a measure $\mu_t \in \mathscr{P} (\mathbb{R})$. This measure is called the \emph{free convolution} between $\mu$ and $\mu_{\semci}^{(t)}$, and we often write $\mu_t = \mu \boxplus \mu_{\semci}^{(t)}$. By \cite[Corollary 2]{FCSD}, $\mu_t$ has a density $\varrho_t = \varrho_t^{\mu} : \mathbb{R} \rightarrow \mathbb{R}_{\ge 0}$ with respect to Lebesgue measure for $t > 0$.

In the following lemma, we collect some properties of free convolution with rescaled semicircle distributions; it mainly follows from \cite{FCSD}.

\begin{lem}
	
	\label{yconvolution}
	
	The following statements hold, for any real number $t > 0$.
	
	\begin{enumerate} 
		
		\item \label{i:boundary} Define the function $v_t : \mathbb{R} \rightarrow \mathbb{R}_{\ge 0}$ by for each $u \in \mathbb{R}$ setting 
		\begin{flalign}
			\label{vte} 
			v_t (u) = \displaystyle\inf \bigg\{ v \ge 0 :  \displaystyle\int_{-\infty}^{\infty} \displaystyle\frac{\mu(dx)}{(u-x)^2 + v^2} \le t^{-1} \bigg\}.
		\end{flalign} 
		
		\noindent Then $v_t$ is continuous on $\mathbb{R}$. Moreover, the boundary of $\Lambda_t$ is parameterized by $\partial\Lambda_t=\{E+\ri v_t(E): E\in \bR\}$, and the set $\big\{ E\in \bR: v_t(E)>0 \big\}$ consists of countably many open intervals $\bigcup_{i\geq 1}(a_i,b_i)$. 
		\begin{enumerate} 
			\item  \label{vesum} For each $E\in \bigcup_{i\geq 1}(a_i,b_i)$, we have  $\int_{-\infty}^\infty \big|x-E-\ri v_t(E) \big|^{-2} \mu (dx) = t^{-1}$.
			\item For each $E\in\bR\setminus \bigcup_{i\geq 1}(a_i,b_i)$, we have $\int_{-\infty}^\infty \big| x-E-\ri v_t(E) \big|^{-2} \mu(dx) \leq t^{-1}$.
		\end{enumerate}
		\item \label{i:mu} We have $\supp \mu \subseteq \bigcup_{i\geq 1}[a_i,b_i]$, and $\mu \big( \bigcup_{i\geq 1}\{a_i,b_i\} \big)=0$.
		
		\item	\label{i:bijection} The function $M \big(E+\ri v_t(E) \big)$ is (strictly) increasing in $E \in \bR$. Moreover, $v_t (E)$ and $M \big( E + \mathrm{i} v_t (E) \big)$ are smooth for $E$ in the interior of $\mathbb{R} \setminus \bigcup_{i\ge 1} \{ a_i, b_i \}$. 
		
		\item
		\label{i:rhoy}A real number $y \in \mathbb{R}$ satisfies $\varrho_t (y) > 0$ if and only if $y =M(w)= w - t m_0 (w)$ for some $w = w(y) \in \partial \Lambda_t \cap \mathbb{H}$. Moreover, the function $w(y)$ is smooth  in $y \in \big\{ y' \in \mathbb{R} : \varrho_t (y') > 0 \big\}$. 
		\begin{enumerate} 
			\item \label{rhoyw} We have $\varrho_t (y) = \pi^{-1} \Imaginary m_0(w)= (\pi t)^{-1} \Imaginary w$.
			\item The Hilbert transform of $\varrho_t$ is given by $H\varrho_t (y)= \pi^{-1} \Real m_0(w) = (\pi t)^{-1} \Real (w-y)$. 
		\end{enumerate}
		
		\item \label{i:edge}  Denote $\mathfrak{e}_+ = \mathfrak{e}_+ (t) = \max (\supp \mu_t)$ and $\mathfrak{e}_- = \mathfrak{e}_- (t) = \min \supp \mu_t$, and let
		\begin{flalign*} 
			w_+ = \sup \Bigg\{ w \in \mathbb{R} : \displaystyle\int_{-\infty}^{\infty} \displaystyle\frac{\mu (d x)}{(x-w)^2} > t^{-1} \Bigg\}; \qquad w_- = \inf \Bigg\{ w \in \mathbb{R} : \displaystyle\int_{-\infty}^{\infty} \displaystyle\frac{\mu (d x)}{(x-w)^2} > t^{-1} \Bigg\}.
		\end{flalign*}
		Then $\mathfrak{e}_+=w_+-tm_0(w_+)$ and $\mathfrak{e}_- = w_- - tm_0 (w_-)$.
		
		\item \label{i:rhoevolution} For any real number $t > 0$, both $\varrho_t(y)$ and its Hilbert transform $H \varrho_t (y)$ are smooth in $y \in \big\{ y' \in \mathbb{R} : \varrho_t (y') > 0 \big\}$. Moreover, we have 
	 \begin{flalign}
	 	\label{trhoty}
	 	\partial_t \varrho_t(y)= \pi \cdot \partial_y \big(\varrho_t(y) H \varrho_t (y) \big), \qquad \text{for $y \in \big\{ y' \in \mathbb{R} : \varrho_t (y') > 0 \big\}$}.
	\end{flalign}

	\end{enumerate}

\end{lem}

\begin{proof}
	\Cref{i:boundary} of the lemma follows from \cite[Lemma 2]{FCSD} and the content in \cite{FCSD} directly following it (the fact that $\big\{ E \in \mathbb{R}: v_t (E) > 0 \big\}$ is open, implying that it consists of countably many open intervals, follows from the definition \eqref{vte} of $v_t$). To establish \Cref{i:mu}, assume to the contrary that $\mu \big( \mathbb{R} \setminus \bigcup_{i \ge 1} (a_i, b_i) \big) > 0$. Then, there exists some constant $c > 0$ and subset $I \subseteq \mathbb{R} \setminus \bigcup_{i \ge 1} (a_i, b_i)$ such that $\mu (I) \ge c$ and $I \subseteq [-c^{-1}, c^{-1}]$. Denoting for any integers $n \ge 1$ and $j \in \llbracket -n, n-1 \rrbracket$ the subsets $I_{j; n} = I \cap \big[ j / cn, (j+1) / cn \big]$, it follows for each $n \ge 1$ that there exists some $j_n \in \llbracket -n, n-1 \rrbracket$ such that $\mu (I_{j_n; n}) \ge c / 2n$. Hence, for any real number $E_n \in I_{j_n; n} \subset \mathbb{R} \setminus \bigcup_{i \ge 1} (a_i, b_i)$, we have 
	\begin{flalign*} 
		t^{-1} \ge \displaystyle\int_{-\infty}^{\infty} |x-E_n|^{-2} \mu (dx) \ge \displaystyle\int_{I_{j_n; n}} |x-E_n|^{-2} \mu (dx) \ge c^2 n^2 \mu (I_{j_n; n}) \ge \displaystyle\frac{c^3 n}{2},
	\end{flalign*}
	
	\noindent where in the first inequality we used \Cref{i:boundary} of the lemma; in the second we restricted the integral to $I_{j_n; n}$; in the third we used the bound $|x-E| \le (cn)^{-1}$ for each $x, E \in I_{j_n; n}$; and in the fourth we used the fact that $\mu (I_{j_n; n}) > c / 2n$. Taking $n > 2(tc^3)^{-1}$ gives a contradiction, implying \Cref{i:mu} of the lemma.
	
	To verify \Cref{i:bijection} of the lemma, recall from \Cref{mz} that $M$ is a continuous bijection from $\partial \Lambda_t$ to $\mathbb{R}$. Together with \Cref{i:boundary} of the lemma, this implies that $M \big( E + \mathrm{i} v_t (E) \big)$ is either (strictly) increasing or (strictly) decreasing. By \cite[Lemma 5]{FCSD}, $M \big( E + \mathrm{i} v_t (E) \big)$ is (strictly) increasing on $\bigcup_{i \ge 1} (a_i, b_i)$ (which is nonempty by \Cref{i:mu} of the lemma), meaning that $M \big( E + \mathrm{i} v_t (E) \big)$ is (strictly) increasing on $\mathbb{R}$; this verifies the first statement in \Cref{i:bijection} of the lemma. The fact that $v_t$ is smooth on $\bigcup_{i \ge 1} (a_i, b_i)$ follows from \cite[Lemma 2]{FCSD}; since $v_t = 0$ in the interior of $\mathbb{R} \setminus \bigcup_{i \ge 1} (a_i, b_i)$, it follows that $v_t$ is smooth in the interior of $\mathbb{R} \setminus \bigcup_{i \ge 1} \{ a_i, b_i \}$. This also implies that $M \big( E + \mathrm{i} v_t (E) \big)$ is smooth in the interior of $\mathbb{R} \setminus \bigcup_{i \ge 1} (a_i, b_i)$ (since, by \eqref{mz0}, $m \big( E + \mathrm{i} v_t (E) \big)$ is smooth in $v_t (E)$, as long as $v_t (E) > 0$ or $E \notin \supp \mu$), verifying \Cref{i:bijection} of the lemma.
	
	All statements in \Cref{i:rhoy} of the lemma, except for the smoothness of $w$ in $y$, follow from \cite[Corollary 2, Corollary 3]{FCSD}. To show this smoothness, we use \Cref{i:boundary} of the lemma to express $w=E+\ri v_t(E)$ for some real number $E \in \mathbb{R}$. Thus, $y(E)=M(w)=M \big(E+\ri v_t(E) \big)$, which is smooth in $E$ by \Cref{i:bijection} of the lemma, which also indicates that $\partial_E M \big(E+\ri v_t(E) \big)>0$ on the domain where $v_t (E) > 0$. So, the implicit function theorem implies that $w = E+\ri v_t(E)$ is smooth in $y$ when $v_t (E) > 0$ (that is, when $w \in \mathbb{H}$, or equivalently when $\varrho_t (y) > 0$). Then,  \Cref{i:edge} of the lemma follows from \Cref{i:rhoy} and the definition \eqref{vte} of $v_t$.    
	
	The first part of \Cref{i:rhoevolution} follows from \cite[Corollary 3]{FCSD} (together with \cite[Lemma 2]{FCSD}). For the second part of  \Cref{i:rhoevolution}, by taking the derivative on both sides of \eqref{mt} with respect to $t$, we get for any $z \in \Lambda_t$ that
	\begin{flalign*} 
		0 = \partial_t m_0 (z) = \partial_t \Big( m_t \big( z - tm_0 (z) \big) \Big) & = \partial_t m_t \big( z - t m_0 (z) \big) - \partial_z m_t \big( z - t m_0 (z) \big) \cdot m_0 (z) \\ 
		& = \partial_t m_t \big( z - tm_0 (z) \big) - \partial_z m_t \big( z - tm_0 (z) \big) \cdot m_t \big( z - tm_0 (z) \big).
	\end{flalign*} 

	\noindent Setting $w = z - tm_0 (z)$ and applying \Cref{mz}, it follows for any $w = y + \mathrm{i} \eta \in \mathbb{H}$ that 
	\begin{flalign*}
		\partial_t m_t (y + \mathrm{i} \eta) = \displaystyle\frac{1}{2} \cdot \partial_z \big( m_t (y + \mathrm{i} \eta) \big)^2.
	\end{flalign*}  
	
	\noindent Sending $\eta$ to zero and applying \eqref{mrho} (with the first part of the lemma), we find
	\begin{align}\label{e:limitzero}
		\pi \cdot \partial_t \big(H \varrho_t (y)+\mathrm{i} \varrho_t(y) \big)=\frac{\pi^2}{2}\partial_y \Big( \big(H \varrho_t (y)+\mathrm{i} \varrho_t(y) \big)^2 \Big).
	\end{align}

	\noindent Then \eqref{trhoty} follows from comparing the imaginary parts on both sides of \eqref{e:limitzero}.
\end{proof}

\begin{rem} 
		
		\label{r:convpmeasure}
		
		While free convolutions are typically defined between probability measures, the relation \eqref{mt} also defines the free convolution of any measure $\mu\in \mathscr{P}_{\fin}$, satisfying $A = \mu(\mathbb{R}) < \infty$, with the rescaled semicircle distribution $\mu_{\semci}^{(t)}$. Indeed, define the probability measure $\widetilde{\mu} \in \mathscr{P}$ from $\mu$ by setting $\widetilde{\mu} (I)= A^{-1} \cdot \mu (A^{1/2} I)$, for any interval $I\subseteq \mathbb{R}$. Furthermore, for any real number $s \ge 0$, define the probability measure $\widetilde{\mu}_s = \widetilde{\mu} \boxplus \mu_{\semci}^{(s)}$, and denote its Stieltjes transform by $\widetilde{m}_s$. Then, define the free convolution $\mu_t = \mu \boxplus \mu_{\semci}^{(t)}$ and its Stieltjes transform $m_t = m_{t; \mu}$ by setting
		\begin{flalign*} 
			\mu_t (I) = A \cdot \widetilde{\mu}_t (A^{-1/2} I), \qquad \text{for any interval $I \subseteq \mathbb{R}$, so that} \qquad	m_t (z) = A^{1/2} \cdot \widetilde m_t ( A^{-1/2} z),
		\end{flalign*} 
		
		\noindent where the second equality follows from the first by \eqref{mz0}. Then, 
		\begin{align*}
			m_{t} \big( z-t m_{0}(z) \big) = A^{1/2} \cdot \widetilde{m}_t \Big( A^{-1/2} \big(z- t m_{0}(z) \big) \Big) & = A^{1/2} \cdot \widetilde{m}_t \big( A^{-1/2} z- t \widetilde{m}_{0}(A^{-1/2} z) \big) \\
			& = A^{1/2} \cdot \widetilde{m}_{0}(A^{-1/2} z)= m_{0}(z),
		\end{align*}
		
		\noindent so that \eqref{mt} continues to hold for $m_t$. In particular, \Cref{mz} and \Cref{yconvolution} also hold for $\mu$. 
		
	\end{rem}

\begin{example}\label{r:Etbound}
If we start from the delta mass $\mu_0=A\delta_0$, then the free convolution $\mu_t=\mu_0\boxplus \mu_{\semci}^{(t)}$ is supported on $[-2\sqrt {At}, 2\sqrt {At}]$, and it is given explicitly by 
\begin{align}
\mu_t(\rd x)=\frac{\sqrt{4At-x^2}}{2\pi t}\rd x, \quad t\geq 0. 
\end{align}
\end{example}

\subsection{Empirical measure of Dyson Brownian motion}
It is known that for $n$-particle Dyson Brownian motion, if the initial data \eqref{e:initial} weakly converges to a probability measure $\mu$, then the empirical particle density \eqref{e:empiricald} (as a measure valued process) converges to the free convolution of $\mu$ with the rescaled semicircle distribution. We recall the precise statement from \cite{anderson2010introduction}. 

	\begin{prop}[{\cite[Proposition 4.3.10, Corollary 4.3.11]{anderson2010introduction}}]
		
		\label{rhot} 
		
		Fix a compactly supported probability measure $\mu \in \mathscr{P}$ and real number $\sfT>1$. Denote the free convolution of $\mu$ with the rescaled semicircle distribution $\mu_{\semci}^{(t)} \in \mathscr{P}$ by $\mu_t$. 
		For each integer $n\geq 1$, let $\bmu^{(n)}=(u^{(n)}_1,u^{(n)}_2,\cdots,u^{(n)}_n)\in \overline\bW_n$ denote sequences such that the empirical measures $(1/n)\sum_{i=1}^{(n)}\delta_{u^{(n)}_i}$ converge weakly to $\mu$. 
		 Consider $n$ particle Dyson Brownian motion $\bm\la^{(n)}(t) = (\la_1^{(n)}(t), \la_2^{(n)}(t), \ldots , \la_n^{(n)}(t))$ starting from $\bmu^{(n)}$. We denote the empirical particle density as $\mu_t^{n}=(1/n)\sum_{i=1}^{(n)} \delta_{\la_i^{(n)}(t)}$. Then, for any real number $\varepsilon > 0$, we have $\lim_{n \rightarrow \infty} \mathbb{P} \big[ \max_{0\leq t\leq \sfT}d_{\dL} (\mu_t^{(n)}, \mu_t) > \varepsilon \big] = 0$, where $d_{\text L}$ is the L\'{e}vy metric from \eqref{e:defLM}.
		
	\end{prop}

\begin{remark}
In \Cref{rhot}, we can allow some particles to be at $\pm \infty$, as in \Cref{r:changep}. Assume the initial data
$(1/n)\sum_{i=1}^n\delta_{u^n_i}$ converges weakly to a positive measure $\mu$ with $\mu(\bR)\leq 1$. Let $\mu_t$  be the free convolution of $\mu$ with the rescaled semicircle distribution $\mu_{\semci}^{(t)} \in \mathscr{P}$ as in \Cref{r:convpmeasure}. Then \Cref{rhot} still holds.
\end{remark}

A  stronger version of convergence has been proven in \cite{huang2019rigidity}, which proves that the optimal bulk rigidity holds for Dyson Brownian motion with general potentials. To state it, we need some further notation. 
\begin{definition}
		
		\label{gammarho} 
		
		Let $\mu \in \mathscr{P}_{\fin}$ denote a finite measure with $\mu(\bR)=A$.  We denote the associated inverted cumulative density $\gamma: [0, A]\mapsto \bR$, as
		\begin{flalign}
			\label{gammay} 
			\gamma(y) = \displaystyle\sup \Bigg\{ \gamma \in \mathbb{R} : \displaystyle\int_{\gamma}^\infty \rd\mu(x) \ge y \Bigg\}. 
		\end{flalign}
		We also set $\gamma(y)=\infty$ if $y<0$, and $\gamma(y)=-\infty$ if $y>A$. Then $\gamma(y)$ is non-increasing and left continuous, i.e., $\gamma(y^-)=\gamma(y)$. 
		 For any integers $n \ge 1$ and $j \in \mathbb{Z}$, we define the \emph{classical location} (also called \emph{$n^{-1}$-quantiles}) with respect to $\mu$, 
		 \begin{align}\label{e:classical_loc}
		 \gamma_j = \gamma_j^{\mu} = \gamma_{j; n} = \gamma_{j; n}^{\mu} \in \mathbb{R}=\gamma((j-1/2)/n),
		 \end{align}
	 for $j\in \llbracket1,An\rrbracket$. We also set $\gamma_j = \infty$ if $j < 1$ and $\gamma_j = -\infty$ if $j > An$.
		
	\end{definition}

\begin{prop}[{\cite[Corollary 3.2]{huang2019rigidity}}]
\label{t:bulkrigidity}
	For any real number $D > 1$, there exists a constant $C = C (D) > 1$ such that the following holds.
Consider $n$ particle Dyson Brownian motion $\bm\la(t) = (\la_1(t), \la_2(t), \ldots , \la_n(t))$ starting from $\mu_0=\mu_0^{(n)}$ with $\supp \mu_0\subset [-n^{D}, n^D]$.  We denote by $\mu_t=\mu_0\boxplus \mu_{\semci}^{(t)}$ the free convolution of $\mu_0$ with rescaled semicircle distribution, and its classical locations by $\{\gamma_i(t)\}_{1\leq i\leq n}$ (as in \eqref{e:classical_loc}). Then, with probability $1-Ce^{-(\log n)^2}$, we have for all $0\leq t\leq n^D$ and $1\leq i\leq n$ that
\begin{align}
	\label{lambdatprobability}
\gamma_{i+\lfloor(\log n)^5\rfloor}(t)-n^{-D}\leq \lambda_i(t)\leq \gamma_{i-\lfloor(\log n)^5\rfloor}(t)+n^{-D}. 
\end{align}
\end{prop}

We mention that, In \cite{huang2019rigidity}, the probability on the right side of \eqref{lambdatprobability} was written to be $1 - C n^{-D}$ for any $D > 1$, but it can be seen from the proof (see that of \cite[Proposition 3.8]{MCTMGP}, where $\delta$ there is $\frac{5}{4}$ here) that it can be taken to be $1 - C e^{ - (\log n)^2}$ instead.

\subsection{Dyson Brownian Motion Comparison}

In this section, we allow some particles to be at $\pm\infty$ as in \Cref{r:changep}. The following lemma gives $i$-th particle location and gap comparisons between Dyson Brownian motions with different initial data. The proof utilizes a coupling of Dyson Brownian motions and a maximum principle, which has been used previously in the random matrix literature \cite{bourgade2017eigenvector, bourgade2021extreme}.

\begin{figure}
\centering
\includegraphics[width=1\textwidth]{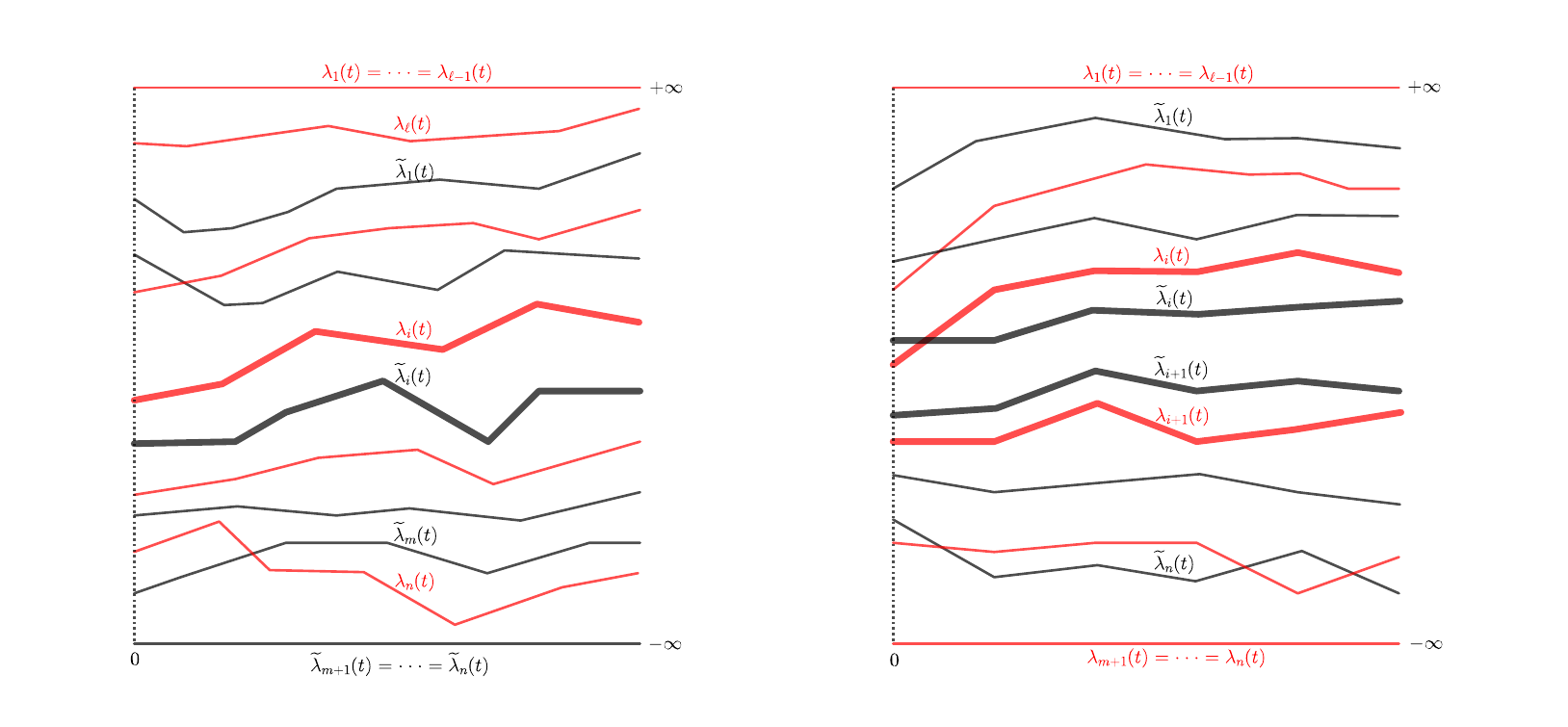}
\caption{Left Panel: if the initial data satisfies $\lambda_i(0)\geq \wt\lambda_i(0)$, then we can couple the two Dyson Brownian motions $\lambda_i(t)\geq \wt\lambda_i(t)$ (with possibly some particles at $\pm \infty$).
Right Panel: if the initial data satisfies $\lambda_i(0)-\lambda_j(0)\geq \wt\lambda_i(0)-\wt\lambda_j(0)$, then we can couple the two Dyson Brownian motions $\lambda_i(t)-\lambda_j(t)\geq \wt\lambda_i(t)-\wt\lambda_j(t)$.
 }
\label{f:DBM_coupling}
\end{figure}

\begin{lem}\label{l:coupling}
Take any integers $1\leq \ell\leq m\leq n$. 
Consider two coupled Dyson Brownian motions (see \Cref{f:DBM_coupling})
\begin{flalign}
\label{e:DBM1}		&\rd\lambda_i (t) = \displaystyle\sum_{\substack{1 \le j \le n \\ j \ne i}}	\displaystyle\frac{\rd t}{\lambda_i (t) - \lambda_j (t)} + \Big( \displaystyle\frac{2}{\beta n} \Big)^{1/2} \rd B_i (t),\\
\label{e:DBM2}		&\rd\wt\lambda_i (t) = \displaystyle\sum_{\substack{1 \le j \le n \\ j \ne i}}	\displaystyle\frac{\rd t}{\wt\lambda_i (t) - \wt\lambda_j (t)} + \Big( \displaystyle\frac{2}{\beta n} \Big)^{1/2} \rd B_i (t),
	\end{flalign}
which share the same Brownian motions $\{B_i(t)\}_{1\leq i\leq n}$. 
\begin{enumerate}
\item Assume that $\lambda_i(t)=+\infty$ for $i\in\llbracket1, \ell-1\rrbracket$ and $\lambda_i(t)\in \bR$ for $i\in\llbracket\ell, n\rrbracket$, and that $\wt\lambda_i(t)\in \bR$ for $i\in \llbracket 1, m\rrbracket$ and $\wt \lambda_i(t)=-\infty $ for $i\in\llbracket m+1,n\rrbracket$. If $\lambda_i(0)\geq \wt\lambda_i(0)$ for all $\ell\leq i\leq m$, then $\lambda_i(t)\geq \wt\lambda_i(t)$ for all $\ell\leq i\leq m$ and $t\geq 0$. 
\item Assume that $\lambda_i(t)=+\infty$ for $i\in\llbracket1, \ell-1\rrbracket$; $\lambda_i(t)=-\infty$ for $i\in\llbracket m+1, n\rrbracket$; $\lambda_i(t)\in \bR$ for $i\in\llbracket\ell, m\rrbracket$; and $\wt\lambda_i(t)\in \bR$ for $i\in \llbracket 1, n\rrbracket$. If $\lambda_i(0)-\lambda_j(0)\geq \wt\lambda_i(0)-\wt\lambda_j(0)$ for all $\ell\leq i<j\leq m$, then $\lambda_i(t)-\lambda_j(t)\geq \wt\lambda_i(t)-\wt\lambda_j(t)$ for all $\ell\leq i<j\leq m$ and $t\geq 0$. 
\end{enumerate} 	
\end{lem}

\begin{proof}[Proof of \Cref{l:coupling}]
By taking difference of \eqref{e:DBM1} and \eqref{e:DBM2}, we have
\begin{align}\label{e:diffeq}
\rd(\lambda_i (t)-\wt\lambda_i(t)) = -\displaystyle\sum_{\substack{1 \le j \le n \\ j \ne i}}	\displaystyle\frac{(\lambda_i (t)-\wt\lambda_i(t))-(\lambda_j (t)-\wt\lambda_j(t))\rd t}{(\lambda_i (t) - \lambda_j (t))(\wt\lambda_i (t) - \wt\lambda_j (t))},\quad \ell\leq i\leq m.
\end{align}
Let 
$
i_*(t):=\argmin_{\ell\leq i\leq m}\{\lambda_i (t)-\wt\lambda_i(t)\}
$
(if $i_*(t)$ is not unique, we can simply take the smallest one).
By our assumption for $i\leq \ell-1$ or $i\geq m+1$, $\lambda_i(t)-\wt\lambda_{i}(t)=\infty$. Thus we also have 
$
i_*(t)=\argmin_{1\leq i\leq n}\{\lambda_i (t)-\wt\lambda_i(t)\}. 
$
Since $\lambda_i(t)-\widetilde \lambda_i(t)$ are continuous, $i_*(t)$ is piecewise constant. Either $i_*(t)$ is continuous at $t$ or $i_*(t)$ has a jump at $t$. If $i_*(t)$ is continuous at $t$, then $\rd i_*(t)=0$ and 
\begin{align}\label{e:positiveder}
\frac{\rd(\lambda_{i_*(t)} (t)-\wt\lambda_{i_*(t)}(t))}{\rd t} = -\displaystyle\sum_{\substack{1 \le j \le n \\ j \ne {i_*(t)}}}	\displaystyle\frac{(\lambda_{i_*(t)} (t)-\wt\lambda_{i_*(t)}(t))-(\lambda_j (t)-\wt\lambda_j(t))}{(\lambda_{i_*(t)} (t) - \lambda_j (t))(\wt\lambda_{i_*(t)} (t) - \wt\lambda_j (t))}\geq 0,
\end{align}
where in the last inequality we used the definition of $i_*(t)$ that $(\lambda_{i_*(t)} (t)-\wt\lambda_{i_*(t)}(t))-(\lambda_j (t)-\wt\lambda_j(t))\leq 0$, and $(\lambda_{i_*(t)} (t) - \lambda_j (t))(\wt\lambda_{i_*(t)} (t) - \wt\lambda_j (t))> 0$. 
If  $i_*(t)$ has a jump at $t$, then we have $\lambda_{i_*(t^-)} (t)-\wt\lambda_{i_*(t^-)}(t)=\lambda_{i_*(t^+)} (t)-\wt\lambda_{i_*(t^+)}(t)$, and as in \eqref{e:positiveder} we have
\begin{align}\label{e:positiveder2}
\frac{\rd(\lambda_{i_*(t^-)} (t^-)-\wt\lambda_{i_*(t^-)}(t^-))}{\rd t},  \frac{\rd(\lambda_{i_*(t^+)} (t^+)-\wt\lambda_{i_*(t^+)}(t^+))}{\rd t} \geq 0.
\end{align}
We conclude from \eqref{e:positiveder} and \eqref{e:positiveder2} that $\lambda_{i_*(t)} (t)-\wt\lambda_{i_*(t)}(t)$ is non-decreasing in $t$. Moreover by our assumption $\lambda_{i_*(0)} (0)-\wt\lambda_{i_*(0)}(0)\geq 0$.  We conclude that for all $\ell\leq i\leq m$,
$\lambda_i(t)- \wt\lambda_i(t)\geq \lambda_{i_*(t)} (t)-\wt\lambda_{i_*(t)}(t)\geq 0$ for any $t\geq 0$. This finishes the first claim.

To prove the second claim, we further take the difference of \eqref{e:diffeq} with $i$ replaced by $i+1$, 
\begin{align}\begin{split}\label{e:gapdiff}
&\frac{\rd((\lambda_i (t)-\lambda_{i+1}(t))-(\wt\lambda_i (t)-\wt\lambda_{i+1}(t)))}{\rd t}
=-2\displaystyle\frac{(\lambda_i (t)-\lambda_{i+1} (t))-(\wt\lambda_i(t)-\wt\lambda_{i+1}(t))}{(\lambda_i (t) - \lambda_{i+1} (t))(\wt\lambda_i (t) - \wt\lambda_{i+1} (t))}\\
&-\displaystyle\sum_{\substack{1 \le j \le n \\ j \ne i,i+1}}
\frac{\lambda_i(t)-\lambda_{i+1}(t)}{(\lambda_i(t)-\lambda_j(t))(\lambda_{i+1}(t)-\lambda_j(t))}
-\frac{\wt\lambda_i(t)-\wt\lambda_{i+1}(t)}{(\wt\lambda_i(t)-\wt\lambda_j(t))(\wt\lambda_{i+1}(t)-\wt\lambda_j(t))}.
\end{split}\end{align}
Let 
$
i_*(t)=\argmin_{\ell\leq i\leq m-1}\{(\lambda_i (t)-\lambda_{i+1}(t))-(\wt\lambda_i (t)-\wt\lambda_{i+1}(t))\}
$ (if $i_*(t)$ is not unique, we can simply take the smallest one).
By our assumption for $i\leq \ell-1$ or $i\geq m$, $(\lambda_i (t)-\lambda_{i+1}(t))-(\wt\lambda_i (t)-\wt\lambda_{i+1}(t))=\infty$. Thus we also have 
$
i_*(t)=\argmin_{1\leq i\leq n}\{(\lambda_i (t)-\lambda_{i+1}(t))-(\wt\lambda_i (t)-\wt\lambda_{i+1}(t))\}
. 
$
Since $(\lambda_i (t)-\lambda_{i+1}(t))-(\wt\lambda_i (t)-\wt\lambda_{i+1}(t))$ are continuous, $i_*(t)$ is piecewise constant.
Either $i_*(t)$ is continuous at $t$ or $i_*(t)$ has a jump at $t$.
If  $i_*(t)$ has a jump at $t$, then we have $(\lambda_{i_*(t^-)}(t)-\lambda_{{i_*(t^-)}+1}(t))-(\wt\lambda_{i_*(t^-)} (t)-\wt\lambda_{{i_*(t^-)}+1}(t))=(\lambda_{i_*(t^+)}(t)-\lambda_{{i_*(t^+)}+1}(t))-(\wt\lambda_{i_*(t^+)} (t)-\wt\lambda_{{i_*(t^+)}+1}(t))$. In both cases $(\lambda_{i_*(t)}(t)-\lambda_{{i_*(t)}+1}(t))-(\wt\lambda_{i_*(t)} (t)-\wt\lambda_{{i_*(t)}+1}(t))$ is continuous in $t$. 

By our assumption $(\lambda_{i_*(0)} (0)-\lambda_{i_*(0)+1}(0))-(\wt\lambda_{i_*(0)} (0)-\wt\lambda_{i_*(0)+1}(0))\geq 0$.
We denote the stopping time
\begin{align}\label{e:deftau}
\tau=\inf_{t\geq 0}\{t: (\lambda_{i_*(t)} (t)-\lambda_{i_*(t)+1}(t))-(\wt\lambda_{i_*(t)} (t)-\wt\lambda_{i_*(t)+1}(t))<0\}. 
\end{align}
Then either $\tau=\infty$, and our claim follows or $\tau<\infty$. Next we prove by contradiction that $\tau<\infty$ is impossible. Otherwise 
\begin{align}\begin{split}\label{e:diff0}
&(\lambda_{i_*(\tau)} (\tau)-\lambda_{i_*(\tau)+1}(\tau))-(\wt\lambda_{i_*(\tau)} (\tau)-\wt\lambda_{i_*(\tau)+1}(\tau))=0,\\
&
\lambda_{i}(\tau)-\lambda_{j}(\tau)
\geq \widetilde\lambda_{i}(\tau)-\widetilde \lambda_{j}(\tau),\quad \ell\leq i\leq j\leq m,
\end{split}\end{align}
 and there exists arbitrarily small $\varepsilon>0$ such that 
 \begin{align}\label{e:diffsmall}
 (\lambda_{i_*(\tau+\varepsilon)} (\tau+\varepsilon)-\lambda_{i_*(\tau+\varepsilon)+1}(\tau+\varepsilon))-(\wt\lambda_{i_*(\tau+\varepsilon)} (\tau+\varepsilon)-\wt\lambda_{i_*(\tau+\varepsilon)+1}(\tau+\varepsilon))<0.
 \end{align}
 
 By plugging $i=i_*(\tau)$ (if $i_*(t)$ has a jump at $\tau$, we interpret this as $i=i_*(\tau^+)$) and $t=\tau$ in \eqref{e:gapdiff}, and using the first statement in \eqref{e:diff0}, we get
 \begin{align}\begin{split}\label{e:derla}
\frac{\rd((\lambda_{i_*(\tau)} (\tau)-\lambda_{i_*(\tau)+1}(\tau))-(\wt\lambda_{i_*(\tau)} (\tau)-\wt\lambda_{i_*(\tau)+1}(\tau)))}{\rd t}=
-(\lambda_{i_*(\tau)}(\tau)-\lambda_{{i_*(\tau)}+1}(\tau))\times
\\
\sum_{\substack{1 \le j \le n \\ j \ne {i_*(\tau)},{i_*(\tau)}+1}}\frac{1}{(\lambda_{i_*(\tau)}(\tau)-\lambda_j(\tau))(\lambda_{{i_*(\tau)}+1}(\tau)-\lambda_j(\tau))}
-\frac{1}{(\wt\lambda_{i_*(\tau)}(\tau)-\wt\lambda_j(\tau))(\wt\lambda_{{i_*(\tau)}+1}(\tau)-\wt\lambda_j(\tau))}\geq 0,
\end{split}\end{align}
where we also used that $(\lambda_{i_*(\tau)}(\tau)-\lambda_{{i_*(\tau)}+1}(\tau))>0$ and $0<(\wt\lambda_{i_*(\tau)}(\tau)-\wt\lambda_j(\tau))(\wt\lambda_{{i_*(\tau)}+1}(\tau)-\wt\lambda_j(\tau))\leq(\lambda_{i_*(\tau)}(\tau)-\lambda_j(\tau))(\lambda_{{i_*(\tau)}+1}(\tau)-\lambda_j(\tau))$ from \eqref{e:diff0}. If  \eqref{e:derla} is postive, then this contradicts \eqref{e:diffsmall}. 

Otherwise \eqref{e:derla} equals zero, which is only possible if $(\wt\lambda_{i_*(\tau)}(\tau)-\wt\lambda_j(\tau))(\wt\lambda_{{i_*(\tau)}+1}(\tau)-\wt\lambda_j(\tau))-(\lambda_{i_*(\tau)}(\tau)-\lambda_j(\tau))(\lambda_{{i_*(\tau)}+1}(\tau)-\lambda_j(\tau))=0$ for all $j\neq i_*(\tau), i_*(\tau)+1$. Together with \eqref{e:diff0}, this implies that $\lambda_{i_*(\tau)}(\tau)-\lambda_j(\tau)=\wt\lambda_{i_*(\tau)}(\tau)-\wt\lambda_j(\tau)$ for all $1\leq j\leq n$. In particular, $\{\wt \lambda_j(\tau)\}_{1\leq j\leq n}=\{ \lambda_j(\tau)+c\}_{1\leq j\leq n}$ for  $c:=\wt \lambda_{i_*(\tau)}(\tau)-\lambda_{i_*(\tau)}(\tau)\in \bR$. Then the coupling \eqref{e:DBM1} (notice that DBM only depends on the gaps $\lambda_i(t)-\lambda_j(t)$) implies that for any $t\geq \tau$, we have  $\{\wt \lambda_j(t)\}_{1\leq j\leq n}=\{ \lambda_j(t)+c\}_{1\leq j\leq n}$, which again contradicts the definition of $\tau$ in \eqref{e:deftau}.  
\end{proof}

\Cref{rhot} and \Cref{l:coupling} together give a continuous version of comparison for the free convolution with a rescaled semicircle distribution. 
\begin{lem}
	
	\label{freeconvolutioncompare}
	 	
	 	Fix a real number $A\in (0,1]$, and finite measures $\mu, \widetilde{\mu} \in \mathscr{P}_{\rm fin}$ such that $\mu (\mathbb{R}) = A $ and $\widetilde{\mu} (\mathbb{R}) = 1$. Let $\mu_t = \mu \boxplus \mu_{\semci}^{(t)}$ and $\widetilde{\mu}_t = \mu \boxplus \mu_{\semci}^{(t)}$ denote the free convolutions of $\mu$ and $\widetilde \mu$ with the rescaled semicircle distribution respectively. Also denote the associated inverted cumulative density functions (recall \eqref{gammay}) by $\gamma_t:  [0, A] \rightarrow \mathbb{R}$ and $\widetilde{\gamma}_t : [0, 1] \rightarrow \mathbb{R}$  respectively.

		\begin{enumerate}
			\item \label{i:convolutionheight}Assume for each $y\in [0, A]$ that $\gamma_0(y)\leq \wt\gamma_0(y) $. Then, $\gamma_t(y) \leq \widetilde\gamma_t(y)$ holds for each $(t, y) \in [0, \infty] \times [0, A]$.

			\item \label{i:convolutiongap}Assume for each $y,y'\in [0, A]$ that $ |\gamma_0(y)-\gamma_0(y')|\geq |\widetilde{\gamma}_0(y) - \widetilde{\gamma}_0(y')|$. Then, $|\gamma_t(y)-\gamma_t(y')|\geq  |\widetilde{\gamma}_t(y) - \widetilde{\gamma}_t(y')|   $ holds for each $t \in [0,\infty]$ and $y,y'\in [0,A]$.

		\end{enumerate}
	\end{lem}

\begin{proof}[Proof of \Cref{freeconvolutioncompare}]
Using \Cref{rhot} and \Cref{l:coupling} as input, the proofs of the two statements are similar. So we will only prove the second statement. 
We take two sets of initial data, given by  $\{\la_i(0)=\gamma_0((i-1/2)/n)\}_{1\leq i\leq \lfloor An\rfloor}$, $\{\la_i(0)=-\infty\}_{\lfloor An\rfloor<i\leq n}$ and  $\{\wt \la_i(0)=\wt\gamma_0((i-1/2)/n)\}_{1\leq i\leq n}$. Then for 
any index $1\leq i\leq n$, by our assumption for $1\leq i\leq \lfloor An\rfloor$, 
\begin{align}\begin{split}\label{e:largegap1}
\la_{i}(0)-\la_{i+1}(0)
&=\gamma_0((i-1/2)/n)-\gamma_0((i+1/2)/n)\\
&\geq \wt\gamma_0((i-1/2)/n)-\wt\gamma_0((i+1/2)/n)=\wt\la_{i}(0)-\wt\la_{i+1}(0).
\end{split}\end{align}

Then \eqref{e:largegap1} verifies the assumptions in the second statement of \Cref{l:coupling}, so we can couple the two associated Dyson Brownion motions such that
\begin{align}
\la_{i}(t)-\la_{i+1}(t)\geq \wt\la_{i}(t)-\wt\la_{i+1}(t),\quad 1\leq i\leq \lfloor An\rfloor.
\end{align}

\noindent By \Cref{rhot} (recalling the definitions of L\'{e}vy distance from \eqref{e:defLM} and the inverted cumulative density from \eqref{gammay}), we have for any $\varepsilon>0$, $0\leq a\leq A$, and sufficiently large $n\geq 1$ that
\begin{align}
\gamma_t(a+\varepsilon)-\varepsilon\leq\la_{\lfloor an\rfloor}(t)\leq \gamma_t(a-\varepsilon)+\varepsilon,
\quad
\wt\gamma_t(a+\varepsilon)-\varepsilon\leq\wt\la_{\lfloor an\rfloor}(t)\leq \wt\gamma_t(a-\varepsilon)+\varepsilon,
\end{align}
In particular if $\gamma_t$ and $\wt\gamma_t$ are both continuous at $a$, then we have 
\begin{align}
\gamma_t(a)=\lim_{n\rightarrow \infty}\lambda_{\lfloor a n\rfloor}(t),\quad \wt\gamma_t(a)=\lim_{n\rightarrow \infty}\wt\lambda_{\lfloor a n\rfloor}(t),
\end{align}
and for any $0\leq a<b\leq A$ such that $\gamma_t$ and $\wt\gamma_t$ are both continuous at $a, b$ we have
\begin{align}\label{e:ctpoint}
\gamma_t(a)-\gamma_t(b)=\lim_{n\rightarrow\infty} \la_{\lfloor an\rfloor}(t)-\la_{\lfloor bn\rfloor}(t)\geq \lim_{n\rightarrow\infty} \wt \la_{\lfloor an\rfloor}(t)-\wt \la_{\lfloor bn\rfloor}(t)=\wt\gamma_t(a)-\wt\gamma_t(b).
\end{align}

Recall $\gamma_t$ and $\wt\gamma_t$ are non-increasing and left continuous, by from the discussion after  \eqref{gammay}. They are thus continuous at all but countably many points. So for any $0<a<b\leq A$, we can take sequences $\{a_k\}_{k\geq 1}, \{b_k\}_{k\geq 1}$ which converge to $a,b$ respectively from below, such that $\gamma_t$ and $\wt\gamma_t$ are both continuous at each $a_k, b_k$. 
Then it follows from \eqref{e:ctpoint} and the fact the $\gamma_t$ and $\wt\gamma_t$ are left continuous
\begin{align}\label{e:gapbb}
\gamma_t(a)-\gamma_t(b)=\lim_{k\rightarrow \infty}\gamma_t(a_k)-\gamma_t(b_k)\geq \lim_{k\rightarrow \infty}\wt\gamma_t(a_k)-\wt\gamma_t(b_k)=\wt\gamma_t(a)-\wt\gamma_t(b).
\end{align}
If $a=0$, then $\gamma_t(a)=\wt\gamma_t(a)=\infty$ and we have the trivial statement $\gamma_t(a)-\gamma_t(b)=\infty=\wt\gamma_t(a)-\wt\gamma_t(b)$.
This together with \eqref{e:gapbb} finishes the proof of the second statement in \eqref{freeconvolutioncompare}.
\end{proof}

\subsection{Comparisons with Free Convolutions}

In this section we discuss comparisons between Dyson Brownian motions with the classical locations of the free convolution, which will be used in \Cref{s:Strongbound}.

\begin{prop}\label{p:detercoupling}
Fix $\sfT>0$. 
Given a probability measure $\mu_0\in \mathscr{P}$, we denote by $\mu_t=\mu_0\boxplus \mu_{\semci}^{(t)}$ the free convolution of $\mu$ with the rescaled semicircle distribution $\mu_{\semci}^{(t)}$. We denote by $\gamma_i(t)=\gamma_{i;n}^{\mu_t}(t)$ the classical locations of $\mu_t(\rd x)=\varrho_t(x)\rd x$ (recall from \eqref{e:classical_loc}). We assume there exists a small constant $c>0$ such that the following two statements hold for any $0\leq t\leq \sfT$.
\begin{enumerate}
\item $\mu_t=\varrho_t(x)\rd x$  has a bounded density: $\sup_{x \in \mathbb{R}} \varrho_t(x)\leq (2c)^{-1}$.
\item The derivative of $\varrho_t(x)$ is bounded in the following way: when $\dist(x, \{\gamma_1(t),\gamma_2(t),\cdots, \gamma_n(t)\})\leq c/2n$, we have $|\varrho_t'(x)|\leq n/c$. 
\end{enumerate} 
Then there exists $C=C(c,\sfT) > 1$ such that the following holds for  any $0\leq t\leq \sfT$
\begin{align}\label{e:ceigbound}
\bP\left(\sup_{0\leq s\leq t, \atop 1\leq i\leq n}|\lambda_i(s)-\gamma_i(s)|\leq \max_{1\leq j\leq n}|\lambda_j(0)-\gamma_j(0)|+C\left(t+\frac{(\log n)\sqrt t}{\sqrt n}\right)\right)\geq 1-Ce^{-(\log n)^2} . 
\end{align}
\end{prop}

\begin{proof}[Proof of \Cref{p:detercoupling}]
We denote $\fe_+(t)= \max (\supp \mu_t)$ and $\fe_-(t)=\min \supp \mu_t$, as in \Cref{i:edge} of \Cref{yconvolution}.
By our assumption $\mu_t=\varrho_t(x)\rd x$ has a
bounded density, so the inverted cumulative density (recall from \eqref{gammay}) satisfies
\begin{align}\label{e:chgammat}
y=\int_{\gamma_t(y)}^\infty \varrho_t(x)\rd x. 
\end{align}
By taking the time derivative of \eqref{e:chgammat}, and using \Cref{i:rhoevolution} in \Cref{yconvolution}, we get
\begin{align}
\del_t \gamma_t(y)\varrho_t(\gamma_t(y))=\int_{\gamma_t(y)}^\infty \del_t\varrho_t(x)\rd x
=\int_{\gamma_t(y)}^\infty \pi \del_x (\varrho_t(x) H\varrho_t(x)) \rd x
=-\pi \varrho_t(\gamma_t(y)) H\varrho_t(\gamma_t(y)).
\end{align}
It follows that $\del_t \gamma_t(y)=-\pi H\varrho_t(\gamma_t(y))$.
In particular, 
the classical locations of $\varrho_t$ (recall from \eqref{e:classical_loc}) satisfy $\del_t \gamma_i(t)=-\pi H \varrho_t (\gamma_i(t))$. Next we show that the Hilbert transform on the right side of this equation can be approximated by 
\begin{align}\label{e:apDBMa}
\del_t \gamma_i(t)=-\pi H\varrho_t(\gamma_i(t))=\frac{1}{n}\sum_{j:j\neq i}\frac{1}{\gamma_i(t)-\gamma_j(t)}+\OO\left(1\right), \quad 1\leq i\leq n.
\end{align}
where $\OO(1)$ depends on $c$.

Since the density $\varrho_t$ is bounded by $1/(2c)$,  we have $|\gamma_i(t)-\gamma_{i\pm 1}(t)|\geq c/n$, and $|\gamma_i(t)-\fe_\pm(t)|\geq c/n$ (which follows from \eqref{e:classical_loc}, as it implies that $\int_{\fe_-(t)}^{\gamma_i(t)}\varrho_t(x)\rd x, \int_{\gamma_i(t)}^{\fe_+(t)}\varrho_t(x)\rd x\geq 1/(2n)$). We consider the integral
\begin{align}\begin{split}\label{e:diff1}
\PV & \int_{\gamma_{i-1}(t)}^{{\gamma_{i+1}}(t)}\frac{ \varrho_t(x)}{\gamma_i(t)-x}d x \\
&=\PV  \int_{\gamma_{i}(t)-c/2n}^{{\gamma_i(t)}+c/2n}\frac{ \varrho_t(x)}{\gamma_i(t)-x}d x+\int_{\gamma_{i}(t)+c/2n}^{{\gamma_{i-1}}(t)}\frac{ \varrho_t(x)}{\gamma_i(t)-x}d x+\int_{\gamma_{i+1}(t)}^{{\gamma_{i}}(t)-c/2n}\frac{ \varrho_t(x)}{\gamma_i(t)-x}d x\\
&=\PV  \int_{\gamma_{i}(t)-c/2n}^{{\gamma_i(t)}+c/2n}\frac{ \varrho_t(x)}{\gamma_i(t)-x}d x+\OO\left(1\right)\\
&=\PV  \int_{\gamma_{i}(t)-c/2n}^{{\gamma_i(t)}+c/2n}\frac{ \varrho_t(x)- \varrho_t(\gamma_i(t))}{\gamma_i(t)-x}d x+\OO\left(1\right) =\OO\left(\frac{\max_{x: |x-\gamma_i(t)|\leq c/2n}|\varrho_t'(x)|}{n}+1\right)=\OO\left(1\right),
\end{split}\end{align}
where the implicit error $\OO(1)$ depends only on $c$. We used 
in the third line  that on $[\gamma_i(t)+c/2n, \gamma_{i-1}(t)]$ and $[\gamma_{i+1}(t), \gamma_i(t)-c/2n]$ we have $|1/(\gamma_i(t)-x)|\leq 1/(c/2n))=2n/c$, so 
\begin{align*}
&\phantom{{}={}}\left|\int_{\gamma_{i}(t)+c/2n}^{{\gamma_{i-1}}(t)}\frac{ \varrho_t(x)}{\gamma_i(t)-x}d x\right|+\left|\int_{\gamma_{i+1}(t)}^{{\gamma_{i}}(t)-c/2n}\frac{ \varrho_t(x)}{\gamma_i(t)-x}d x\right|\\
&\leq \frac{2n}{c}\left|\int_{\gamma_{i}(t)+c/2n}^{{\gamma_{i-1}}(t)} \varrho_t(x) d x\right|+\frac{2n}{c}\left|\int_{\gamma_{i+1}(t)}^{{\gamma_{i}}(t)-c/2n} \varrho_t(x)d x\right|\leq \frac{2n}{c}\cdot \frac{1}{n}
+ \frac{2n}{c}\cdot \frac{1}{n}=\frac{4}{c}.
\end{align*}
where we also used the definition of classical locations (recall from \eqref{e:classical_loc});
 in the fourth  line we used that $\PV \int_{\gamma_{i}(t)-c/2n}^{{\gamma_i(t)}+c/2 n} (\gamma_{i}(t)-x)^{-1}d x =0$; in the last line, we used that for any $\dist(x, \{\gamma_1(t),\gamma_2(t),\cdots, \gamma_n(t)\})\leq c/2n$, it holds $|\varrho_t'(x)|\leq n/c$ from our assumption. 
For the integral outside the interval $[\gamma_{i-1}(t), \gamma_{i+1}(t)]$, we have the trivial bounds
\begin{align*}
\int_{-\infty}^{\gamma_{i+1}(t)}\frac{\varrho_t(x)d x}{\gamma_i(t)-x}
\leq \sum_{j=i+1}^{n}\int_{\gamma_{j+1}(t)}^{\gamma_{j}(t)}\frac{\varrho_t(x)d x}{\gamma_i(t)-\gamma_j(t)}\leq \sum_{j=i+1}^{n}\frac{1}{n(\gamma_i(t)-\gamma_j(t))},
\end{align*}
and 
\begin{align*}
\int_{-\infty}^{\gamma_{i+1}(t)}\frac{\varrho_t(x)d x}{\gamma_i(t)-x}
\geq \sum_{j=i+1}^{n}\int_{\gamma_{j+1}(t)}^{\gamma_{j}(t)}\frac{\varrho_t(x)d x}{\gamma_i(t)-\gamma_{j+1} (t)}\geq \sum_{j=i+2}^n\frac{1}{n(\gamma_i(t)-\gamma_j(t))}.
\end{align*}
Thus we conclude that
\begin{align}\label{e:diff2}
\left|\int_{-\infty}^{\gamma_{i+1}(t)}\frac{\varrho_t(x)d x}{\gamma_i(t)-x}-\sum_{j=i+1}^{n}\frac{1}{n(\gamma_i(t)-\gamma_j(t))}\right|
\leq \frac{1}{n(\gamma_i(t)-\gamma_{i+1}(t))}\leq\frac{1}{c}.
\end{align}
We have the same estimate for the integral from $\gamma_{i-1}(t)$ to $\infty$. The claim \eqref{e:apDBMa} follows from combining \eqref{e:diff1} and \eqref{e:diff2}.

Next we show \eqref{e:ceigbound} using \eqref{e:apDBMa} as input. By taking  the difference of \eqref{e:DBM} and \eqref{e:apDBMa}, we get
\begin{align}\label{e:dTt}\begin{split}
\rd ( \lambda_i(t)-\gamma_i(t))
&= \Big( \displaystyle\frac{2}{\beta n} \Big)^{1/2} \rd B_i (t)-\frac{1}{n}\sum_{j:j\neq i} \frac{( \lambda_i(t)-\gamma_i(t))-( \lambda_j(t)-\gamma_j(t))}{( \lambda_i(t)- \lambda_j(t))(\gamma_i(t)- \gamma_j(t))}\rd t+\OO\left(1\right),\quad 1\leq i\leq n.
\end{split}\end{align}
We take $i_*(t)=\arg\max_{1\leq i\leq n} ( \lambda_i(t)-\gamma_i(t))$ (if $i_*(t)$ is not unique, we can simply take the smallest one). Since the $\lambda_i(t)-\gamma_i(t)$ are continuous, $i_*(t)$ is piecewise constant. By plugging $i=i_*(t)$ (if $i_*$ has a jump at $t$, we interpret this as either $i_*(t^-)$ or $i_*(t^+)$) in \eqref{e:dTt},  noticing that  the second term on the righthand side of \eqref{e:dTt} is nonpositive, we deduce (through a very similar argument to the one used to obtain \eqref{e:positiveder} and \eqref{e:positiveder2}) that
\begin{align}\label{e:difflaga}
\rd ( \lambda_{i_*(t)}(t)-\gamma_{i_*(t)}(t))\leq  \Big( \displaystyle\frac{2}{\beta n} \Big)^{1/2} \rd B_{i_*(t)} (t)+\OO\left(1\right)\rd t.
\end{align}
The martingale term $\rd B_{i_*(t)}(t)$ has the same law as a standard Brownian motion. Thus by the reflection principle we have
\begin{align}
\bP\left(\sup_{0\leq s\leq t}\int_0^s \rd B_{i_*(\tau)}(\tau)\geq a\sqrt t\right)
=\bP\left(\int_0^t \rd B_{i_*(\tau)}(\tau)\geq a\sqrt t\right)
\leq 2e^{-a^2/2}.
\end{align}
By integrating both sides of \eqref{e:difflaga} (and taking $a=\log n$ in the above tail estimate), we deduce that there exists a large constant $C=C(c,\sfT)>0$ such that, the following holds with probability $1-Ce^{-(\log n)^2}$ for any $0\leq s\leq t$,
\begin{align*}
 \lambda_{i_*(s)}(s)-\gamma_{i_*(s)}(s)
&\leq  \lambda_{i_*(0)}(0)-\gamma_{i_*(0)}(0)+\Big( \displaystyle\frac{2}{\beta n} \Big)^{1/2} \int_0^s\rd B_{i_*(\tau)} (\tau)+\OO(s)\\
&\leq \max_i| \lambda_{i}(0)-\gamma_{i}(0)|+C\left(t+\frac{(\log n)\sqrt t}{ \sqrt n}\right),
\end{align*}
It follows that for any $0\leq t\leq \sfT$ and $n$ large enough, 
\begin{align}\label{e:ub}
\bP\left(\sup_{0\leq s\leq t,\atop 1\leq i\leq n} \lambda_{i}(s)-\gamma_{i}(s)\leq  \max_i| \lambda_{i}(0)-\gamma_{i}(0)|+C\left(t+\frac{(\log n)\sqrt t}{ \sqrt n}\right)\right)\geq 1-Ce^{-(\log n)^2}.
\end{align}
By the same argument, we have a similar lower bound by considering  $i_*(t)=\arg\min_{1\leq i\leq n} ( \lambda_i(t)-\gamma_i(t))$, that is, we have  for any $0\leq t\leq \sfT$ and $n$ large enough,
\begin{align}\label{e:lowb}
\bP\left(\sup_{0\leq s\leq t,\atop 1\leq i\leq n}\gamma_{i}(s)- \lambda_{i}(s)\leq\max_i| \lambda_{i}(0)-\gamma_{i}(0)|+C\left(t+\frac{(\log n)\sqrt t}{ \sqrt n}\right)\right)\geq 1-Ce^{-(\log n)^2},
\end{align}
 \Cref{p:detercoupling} follows from combining \eqref{e:ub} and \eqref{e:lowb}. 
\end{proof}

We also need the following special case of \Cref{p:detercoupling}, which compares Dyson Brownian motion starting from a delta mass with the quantiles of semicircle distribution.

\begin{prop}\label{p:detercoupling2}
Fix a real number $\sfT>0$ and integers $1\leq m\leq n$; let $A=m/n\in (0,1]$. Define the measure $\mu_0=A\delta_0\in \mathscr P$, and denote by $\mu_t=\mu_0\boxplus \mu_{\semci}^{(t)}$ the free convolution of $\mu$ with the rescaled semicircle distribution $\mu_{\semci}^{(t)}$. We denote by $\gamma_i(t)=\gamma_{i;n}^{\mu_t}(t)$ the classical locations (recall from \eqref{e:classical_loc}) of $\mu_t(\rd x)=\varrho_t(x)\rd x$ for $1\leq i\leq m$. We consider the $m$-particle Dyson Brownian motion \eqref{e:DBM} starting from $\lambda_1(0)=\lambda_2(0)=\cdots=\lambda_m(0)=0$ (and $\lambda_{m+1}(0)=\cdots=\lambda_n(0)=-\infty$). Then there exists a constant $C=C(\sfT)>0$ such that the following holds for  any $0\leq t\leq \sfT$
\begin{align}\label{e:extremeee}\begin{split}
&\bP\left(
\sup_{0\leq s\leq t,\atop 1\leq i\leq n}|\lambda_i(s)-\gamma_i(s)|\leq C\left(\sqrt{At}+\frac{(\log n)\sqrt t}{\sqrt n}\right)\right)\geq 1-Ce^{-(\log n)^2},
\\
&\bP\left(
\sup_{0\leq s\leq t,\atop 1\leq i\leq n}
|\lambda_i(s)|\leq C\left(\sqrt{At}+\frac{(\log n)\sqrt t}{\sqrt n}\right)
\right)\geq 1-Ce^{-(\log n)^2}.
\end{split}\end{align}

\end{prop}

\begin{proof}[Proof of \Cref{p:detercoupling2}]
The proof follows essentially the same argument as in \Cref{p:detercoupling}, so we will only sketch the proof. 
We recall from \Cref{r:Etbound} that the free convolution $\mu_t=\mu_0\boxplus \mu_{\semci}^{(t)}$ is supported on $[-2\sqrt {At}, 2\sqrt {At}]$, and given explicitly by the rescaled semicircle distribution
\begin{align}
\mu_t(\rd x)=\varrho_t(x)\rd x=\frac{\sqrt{4At-x^2}}{2\pi t}\rd x, \quad t\geq 0. 
\end{align}
One can directly check that $\varrho_t(x)$ satisfies for $0\leq x\leq 2\sqrt{At}$ the bounds 
\begin{align}\label{e:rhotdbound}
\varrho_t(x)\leq \varrho_t(0)=\frac{\sqrt{A/t}}{\pi}\leq \frac{\sqrt{A/t}}{2},\quad |\varrho_t'(x)|\leq (2\sqrt {At}-x)^{-1/2}\frac{A^{1/4}}{t^{3/4}}.
\end{align}

Using the above estimates as input, similarly to \eqref{e:apDBMa}, next we show 
\begin{align}\label{e:apDBMa2}
\del_t \gamma_i(t)=-\pi H(\varrho_t)(\gamma_i(t))=\frac{1}{n}\sum_{1\leq j\leq m\atop j\neq i}\frac{1}{\gamma_i(t)-\gamma_j(t)}+\OO\left(\sqrt\frac{A}{ t}\right), \quad 1\leq i\leq m.
\end{align}
where the implicit constant in $\OO(\cdot)$ depends only on $\sfT$. Since the density $\varrho_t$ is bounded by $\sqrt{A/t}/2$ (from \eqref{e:rhotdbound},  we have $|\gamma_i(t)-\gamma_{i\pm 1}(t)|\geq \sqrt{t/A}/n$, and $|\gamma_i(t)-(\pm 2\sqrt{At})|\geq \sqrt{t/A}/n$ (which follows from \eqref{e:classical_loc} that $\int_{-2\sqrt{At}}^{\gamma_i(t)}\varrho_t(x)\rd x, \int_{\gamma_i(t)}^{2\sqrt{At}}\varrho_t(x)\rd x\geq 1/(2n)$). We consider the integral 
\begin{align}\begin{split}\label{e:diffA}
&\phantom{{=}}\PV \int_{\gamma_{i-1}(t)}^{{\gamma_{i+1}}(t)}\frac{ \varrho_t(x)}{\gamma_i(t)-x}d x\\
&=\PV  \int_{\gamma_{i}(t)-\sqrt{t/A}/(2n)}^{{\gamma_i(t)}+\sqrt{t/A}/(2n)}\frac{ \varrho_t(x)}{\gamma_i(t)-x}d x+\int_{\gamma_{i}(t)+\sqrt{t/A}/(2n)}^{{\gamma_{i-1}}(t)}\frac{ \varrho_t(x)}{\gamma_i(t)-x}d x+\int_{\gamma_{i+1}(t)}^{{\gamma_{i}}(t)-\sqrt{t/A}/(2n)}\frac{ \varrho_t(x)}{\gamma_i(t)-x}d x\\
&=\PV  \int_{\gamma_{i}(t)-\sqrt{t/A}/(2n)}^{{\gamma_i(t)}+\sqrt{t/A}/(2n)}\frac{ \varrho_t(x)}{\gamma_i(t)-x}d x+\OO\left(\sqrt{\frac{A}{t}}\right)\\
&=\PV  \int_{\gamma_{i}(t)-\sqrt{t/A}/(2n)}^{{\gamma_i(t)}+\sqrt{t/A}/(2n)}\frac{ \varrho_t(x)- \varrho_t(\gamma_i(t))}{\gamma_i(t)-x}d x+\OO\left(\sqrt{\frac{A}{t}}\right)\\
&\leq \frac{\sqrt{t/A}}{n}\max_{x: |x-\gamma_i(t)|\leq \sqrt{t/A}/(2n)}|\varrho_t'(x)|+\OO\left(\sqrt{\frac{A}{t}}\right)\leq \frac{1}{\sqrt{tn}}+\OO\left(\sqrt{\frac{A}{t}}\right)=\OO\left(\sqrt{\frac{A}{t}}\right).
\end{split}\end{align}
We used in the third line that  on $[\gamma_i(t)+\sqrt{t/A}/(2n), \gamma_{i+1}(t)]$ and $[\gamma_{i-1}(t), \gamma_i(t)-\sqrt{t/A}/(2n)]$ 
we have $|1/(\gamma_i(t)-x)|\leq 1/(\sqrt{t/A}/(2n)))=2n\sqrt{A/t}$. So
\begin{align*}
&\phantom{{}={}}
\left|\int_{\gamma_{i}(t)+\sqrt{t/A}/(2n)}^{{\gamma_{i-1}}(t)}\frac{ \varrho_t(x)}{\gamma_i(t)-x}d x\right|+\left|\int_{\gamma_{i+1}(t)}^{{\gamma_{i}}(t)-\sqrt{t/A}/(2n)}\frac{ \varrho_t(x)}{\gamma_i(t)-x}d x\right|\\
&\leq 
2n\sqrt{\frac{A}{t}}\int_{\gamma_{i}(t)+\sqrt{t/A}/(2n)}^{{\gamma_{i-1}}(t)} \varrho_t(x)d x+2n\sqrt{\frac{A}{t}}\int_{\gamma_{i+1}(t)}^{{\gamma_{i}}(t)-\sqrt{t/A}/(2n)} \varrho_t(x)d x\leq 4\sqrt{\frac{A}{t}},
\end{align*}
where we also used the definition of classical locations from \eqref{e:classical_loc};
 in the  fourth  line we used that $\PV \int_{\gamma_{i}(t)-\sqrt{t/A}/(2n)}^{{\gamma_i(t)}+\sqrt{t/A}/(2 n)} (\gamma_{i}(t)-x)^{-1}d x =0$; in the last line, we used that for $x\in [\gamma_m(t)-\sqrt{t/A}/(2n), \gamma_1(t)+\sqrt{t/A}/(2n)]\subset [-2\sqrt{At}+\sqrt{t/A}/(2n), 2\sqrt{At}-\sqrt{t/A}/(2n)]$,  it holds that 
$|\varrho_t'(x)|\leq \sqrt{An}/t$ from \eqref{e:rhotdbound}; in the last equality we used that $A=m/n\geq 1/n$. 
For the integral outside the interval $[\gamma_{i-1}(t), \gamma_{i+1}(t)]$, we have the trivial bounds
\begin{align*}
\int_{-\infty}^{\gamma_{i+1}(t)}\frac{\varrho_t(x)d x}{\gamma_i(t)-x}
\leq \sum_{j=i+1}^{n}\int_{\gamma_{j+1}(t)}^{\gamma_{j}(t)}\frac{\varrho_t(x)d x}{\gamma_i(t)-\gamma_j(t)}\leq \sum_{j=i+1}^{m}\frac{1}{n(\gamma_i(t)-\gamma_j(t))},
\end{align*}
and 
\begin{align*}
\int_{-\infty}^{\gamma_{i+1}(t)}\frac{\varrho_t(x)d x}{\gamma_i(t)-x}
\geq \sum_{j=i+1}^{n}\int_{\gamma_{j+1}(t)}^{\gamma_{j}(t)}\frac{\varrho_t(x)d x}{\gamma_i(t)-\gamma_{j+1} (t)}\geq \sum_{j=i+2}^m\frac{1}{n(\gamma_i(t)-\gamma_j(t))}.
\end{align*}

\noindent Together with the fact that $\gamma_i(t)-\gamma_{i+1}(t)\geq \sqrt {t/A}/n$, this gives 
\begin{align}\label{e:diffA2}
\left|\int_{-\infty}^{\gamma_{i+1}(t)}\frac{\varrho_t(x)d x}{\gamma_i(t)-x}-\sum_{j=i+1}^{n}\frac{1}{n(\gamma_i(t)-\gamma_j(t))}\right|
\leq \frac{1}{n(\gamma_i(t)-\gamma_{i+1}(t))}\leq\sqrt{\frac{A}{t}}.
\end{align}
We have the same estimate for the integral from $\gamma_{i-1}(t)$ to $\infty$. The claim \eqref{e:apDBMa2} follows from combining \eqref{e:diffA} and \eqref{e:diffA2}.

Using \eqref{e:apDBMa2} as input, following the derivations of \eqref{e:ub} and \eqref{e:lowb}, we conclude there exists a large constant $C>0$ such that for any $0\leq t\leq \sfT$ and $n$ large enough, the following holds with probability $1-Ce^{-(\log n)^2}$, 
\begin{align}
\sup_{0\leq s\leq t, \atop 1\leq i\leq n}|\lambda_i(s)-\gamma_i(s)|\leq C\left(\int_0^t \frac{A\rd \tau}{\sqrt \tau}+\frac{(\log n)\sqrt t}{\sqrt n}\right)
= C\left(2\sqrt{tA}+\frac{(\log n)\sqrt t}{\sqrt n}\right).
\end{align}
This finishes the proof of the first statement in \eqref{e:extremeee}. The second statement follows from the first one and the bound $|\gamma_i(s)|\leq 2\sqrt{As}\leq 2\sqrt{At}$.
\end{proof}

\subsection{Free Convolution Density Estimates}
Adopt \Cref{a:densitylowbound}. Denote the free convolution of $\mu_0:=\mu_0^{(n)}$ with the rescaled semicircle distribution $\mu_{\semci}^{(t)}$ by $\mu_t=\mu_0\boxplus \mu_{\semci}^{(t)}$. Denote the Stieltjes transform of $\mu_t$ by $m_t(z)$ as in \eqref{mz0}, and its $p$-th derivative by $m_t^{(p)}(z)$. The free convolution $\mu_t=\varrho_t(x)\rd x$ has a density $\varrho_t (x)$. Thanks to \Cref{i:edge}, the right edge of $\mu_t$, denoted by $E_t = \max (\supp \mu_t)$, satisfies $E_t=\xi(t)-tm_0(\xi(t))$, where $\xi(t)$ is given by 
		\begin{flalign}\label{e:defxit}
			\xi(t) = \sup \Bigg\{ w \in \mathbb{R} : \displaystyle\int_{-\infty}^{\infty} \displaystyle\frac{\mu_0 (\rd  x)}{(x-w)^2} = \displaystyle\frac{1}{t} \Bigg\}.
		\end{flalign}

The following two lemmas collects some properties of $\mu_t$, i.e., square root edge behavior and averaged density estimates. They will be used in \Cref{s:rigidity} to prove our main result \Cref{t:main}.

\begin{lem}\label{l:densityt}
Adopt  \Cref{a:densitylowbound}. Denote the free convolution of $\mu_0:=\mu_0^{(n)}$ with the rescaled semicircle distribution $\mu_{\semci}^{(t)}$ by $\mu_t=\mu_0\boxplus \mu_{\semci}^{(t)}$, and its right edge by $E_t = \max (\supp \mu_t)$. We denote the Stieltjes transform of $\mu_t$ as $m_t(z)$, recall $\xi(t)$ from \eqref{e:defxit}, and introduce the following two parameters:
\begin{align}\label{e:defAs0}
\cA(t)=\frac{2}{-t^3m^{(2)}_0(\xi(t))},\quad \fc(t)=-2^{-9}t m^{(2)}_0(\xi(t))\xi(t)^2.
\end{align}
Then for any  $(30/\sfb)\sqrt \eta_*\leq t\leq \sfT$ we have
\begin{align}\label{e:defAs}
\sqrt t\geq  \xi(t) \geq \left(\frac{\sfb t}{20}\right)^2,\quad \cA(t)\geq 2^{-10}\sfb^2,\quad \fc(t)\geq 2^{-20}(\sfb t)^2,
\end{align}
and for all $x\in [E_t-\fc(t), E_t]$ we have
\begin{align}\label{e:rhotdensity}
 \varrho_t(x)= (1+\cE(E_t-x))\frac{\sqrt{\cA(t)(E_t-x)}}{\pi},\qquad \text{where $\mathcal{E}$ satisfies $|\cE(x)|\leq  \frac{|x|}{\fc(t)}$}.
\end{align}
\end{lem}

\begin{lem}\label{p:densitybound}
Adopt \Cref{a:densitylowbound}. Denote the free convolution of $\mu_0:=\mu_0^{(n)}$ with the rescaled semicircle distribution $\mu_{\semci}^{(t)}$ by $\mu_t=\mu_0\boxplus \mu_{\semci}^{(t)}$, and its right edge by $E_t = \max (\supp \mu_t)$.  We denote the Stieltjes transform of $\mu_t$ by $m_t$. Then there exists  constants $\fC= 2^{21}/ \sfb^2>0$ and $\fc  =\sfb \fC^{-3/2}>0$ such that the following hold. 
\begin{enumerate}
\item The right edge satisfies $E_t\leq 2\sqrt t$.
\item For any $t\in [0,\sfT]$ and $x\in [0, \sfT^2]$ satisfying either $t\geq \fC\sqrt \eta_*$ or $ x\geq \fC\eta_*$, we have $\mu_t([E_t-x, E_t])\geq \fc  x^{3/2}$.
\item For any $t\in [0,\sfT]$ and $\kappa,\eta \geq 0$ with $\kappa+\eta\leq \sfT^2$ satisfying either $t\geq \fC\sqrt \eta_*$ or $\kappa\geq \fC\eta_*$,  we have
\begin{align}\label{e:Immbound}
\Im[m_t(E_t+\kappa+\ri\eta)]\geq \frac{\fc  \eta}{8\sqrt{\kappa+\eta}}.
\end{align}
\end{enumerate}
\end{lem}

\begin{proof}[Proof of \Cref{l:densityt}]

For simplicity of notation, in the rest of the proof, we will write $\kappa=\xi(t)$ as defined in \eqref{e:defxit}. Let $F(x)=\mu_0 ( [-x,0] )$, which satisfies $F(x) \geq \sfb x^{3/2}$ for $x \in [\eta_*, \mathsf{T}^2]$. Then we have
\begin{align}\label{estk}
\frac{1}{t}=\int_{-\infty}^{0}\frac{\mu_0(\rd x)}{(\kappa-x)^2} =
-\int_{0}^{\infty}\frac{\rd F(x)}{(\kappa+x)^2} \geq \int_0^\infty\frac{2F(x)\rd x}{(\kappa+x)^3}
\geq\int_{\eta_*}^1\frac{2\sfb x^{3/2}\rd x}{(\kappa+x)^3},
\end{align}
where the first statement is from \eqref{e:defxit}; the second statement is from the definition $F(x)=\mu_0 ( [-x,0] )$; the third statement is by integration by parts; and the last statement is by restricting the integration domain to $[\eta_*,1]$, and using that $F(x)\geq \sfb x^{3/2}$ on $[\eta_*,\sfT]$ and $\sfT\geq 1$ from \Cref{a:densitylowbound}.

 If $\kappa\leq 2\eta_*$, then  \eqref{estk} gives
\begin{align}
\frac{1}{t}\geq \int_{\eta_*}^1\frac{2\sfb x^{3/2}\rd x}{27x^3}=\frac{4\sfb}{27}\left(\frac{1}{\sqrt \eta_*}-1\right)\geq \frac{\sfb}{27\sqrt \eta_*},
\end{align}
where the second statement is from performing the integral, and the last statement uses $\eta_*\leq 1/4$ from \Cref{a:densitylowbound}. This contradicts our assumption $t\geq (30/\sfb) \sqrt{\eta_*} > 27\sqrt\eta_*/\sfb$, so we have $\kappa\geq 2\eta_*$, and  \eqref{estk} gives
\begin{align}\label{estk2}
\frac{1}{t}\geq \int_{\eta_*}^{\kappa}\frac{2\sfb x^{3/2}\rd x}{8\kappa^3}=\frac{\sfb(\kappa^{5/2}-\eta_*^{5/2})}{10\kappa^3}\geq \frac{\sfb}{20\kappa^{1/2}}, \qquad \text{so} \qquad 
 \kappa \geq \left(\frac{\sfb t}{20}\right)^2,
\end{align}
where the first inequality follows from restricting the integral  \eqref{estk} to $[\eta_*, \kappa]$ and the fact that $\kappa+x\leq 2\kappa$ on this region; the second statement follows from performing the integral; and the last inequality follows from the bound $\kappa\geq 2\eta_*$. This gives the lower bound in the first statement in \eqref{e:defAs}. The upper bound follows from the following estimate
\begin{align*}
\frac{1}{t}=\int_{-\infty}^{0}\frac{\mu_0(\rd x)}{(\kappa-x)^2} \leq 
\int_{-\infty}^{0}\frac{\mu_0(\rd x)}{\kappa^2} =\frac{1}{\kappa^2}
\end{align*}
where the first statement is from \eqref{e:defxit}; the second statement uses $\kappa-x\geq \kappa$; and the last statement uses $\mu_0(\bR)=1$.

We next bound the derivatives of $m_0$, in order to Taylor expand it. In the following, we recall that the $p$-th derivative of $m_0$ is denoted by $m_0^{(p)}$. Since $m_0(w)$ is analytic outside $\supp \mu_0\subset(-\infty,0]$, it is in particular analytic on $\{w\in \bC: |w-\kappa|<\kappa/2\}$. To bound $m_0^{(2)}(\kappa)$ observe that, similarly to \eqref{estk} and \eqref{estk2}, we have 
\begin{align}\begin{split}\label{e:lowm0}
-m_0^{(2)}(\kappa)
& = \displaystyle\int_0^{\infty} \displaystyle\frac{6 F(x) \mathrm{d} x}{(\kappa + x)^4}  \geq\int_{\eta_*}^{\kappa} \frac{6\sfb x^{3/2}\rd x}{(\kappa+x)^4}
\ge \int^\kappa_{\eta_*} \frac{4\sfb x^{3/2}\rd x}{(2\kappa)^4}
= \frac{\sfb(\kappa^{5/2}-\eta_*^{5/2})}{10\kappa^4}
\geq \frac{\sfb}{20\kappa^{3/2}},
\end{split}\end{align}
where in the second statement we restricted the integral to $[\eta_*,\kappa]$, and used $\kappa\leq \sqrt t\leq \sqrt \sfT\leq \sfT^2$ from the first statement in \eqref{e:defAs}, and $F(x)=\mu_0 ( [-x,0] )\geq \sfb x^{3/2}$ from \Cref{a:densitylowbound}; in the third statement we used $\kappa+x\leq 2\kappa$; and in the last two statements used that $\kappa\geq 2\eta_*$. We also have the upper bounds of $-m_0^{(2)}$
\begin{align}\label{e:upm0}
-m_0^{(2)}(\kappa)=\int_\bR2!\frac{\rd\mu_0(x)}{(\kappa-x)^3}\leq \int_\bR2!\frac{\rd\mu_0(x)}{(\kappa-x)^2\kappa}=\frac{2}{t\kappa},
\end{align}
\noindent and 
\begin{align}\label{e:m03}
m_0^{(3)}(\kappa)=\int_\bR 3!\frac{\rd\mu_0(x)}{(\kappa-x)^4}\leq 6\int_\bR \frac{\rd\mu_0(x)}{\kappa(\kappa-x)^3}=\frac{-3 m_0^{(2)}(\kappa)}{\kappa}.
\end{align}
where we bounded $|\kappa-x|\geq \kappa$ for $x\in \supp \mu_0\subset (-\infty, 0]$ (to obtain both estimates) and used \eqref{e:defxit} (in the last statement of \eqref{e:upm0}). 
For any $w\in \bC$ such that $|w-\kappa|<\kappa/2$, we have
\begin{align}\label{e:m04}
|m_0^{(4)}(w)|\leq\int_\bR 4!\frac{\rd\mu_0(x)}{|x-w|^5}\leq 4!\int_\bR \frac{2^3\rd\mu_0(x)}{(\kappa/2)^2(x-\kappa)^3}=\frac{-3\times 2^7 m_0^{(2)}(\kappa)}{\kappa^2},
\end{align}
where we used that $|x-w|\geq \kappa/2$ and $2|x-w|\geq |x-\kappa|$ for $x \in \supp \mu_0\subseteq (-\infty, 0]$ for the second inequality.

We can thus Taylor expand the Stieltjes transform $m_0(w)$ of $\mu_0$ around $\kappa$ when $|w - \kappa| < \kappa / 2$:
\begin{align}\begin{split}\label{e:m0taylor}
m_0(w)
&=m_0(\kappa)+m_0'(\kappa)(w-\kappa)+m_0^{(2)}(\kappa)\frac{(w-\kappa)^2}{2}+m_0^{(3)}(\kappa) \frac{(w-\kappa)^3}{3!}+\widetilde \cE(w) \frac{(w-\kappa)^4}{4!}\\
&=m_0(\kappa)+\frac{(w-\kappa)}{t}+m_0^{(2)}(\kappa)\frac{(w-\kappa)^2}{2}+m_0^{(3)}(\kappa) \frac{(w-\kappa)^3}{3!}+\widetilde \cE(w) \frac{(w-\kappa)^4}{4!},
\end{split}\end{align} 

\noindent where in the second line we used that $\kappa=\xi(t)$ and $m_0'(\xi(t))=1/t$ from the definition of $\xi(t)$ \eqref{e:defxit}. Moreover, thanks to \eqref{e:m04}, the Taylor remainder is bounded by 
\begin{align}\label{e:tcEup}
|\widetilde \cE(z) |\leq \max_{w\in \bC:|w-\kappa|\leq \kappa/2} |m_0^{(4)}(w)|\leq \frac{-3\times 2^7 m_0^{(2)}(\kappa)}{\kappa^2}.
\end{align}

 By \Cref{i:rhoy} in \Cref{yconvolution}, any $x \in \mathbb{R}$ satisfies $\varrho_t (E_t - x) > 0$ if and only if there exists a $w\in \del \Lambda_t\cap \bH$ such that $E_t-x=w-tm_0(w)$. Recalling that $E_t=\kappa-tm_0(\kappa)$ (from $\kappa=\xi(t)$ and \Cref{i:edge} in \Cref{yconvolution}), rearranging \eqref{e:m0taylor} gives
\begin{align}\begin{split}\label{e:constructx}
\bR\ni x
&=(\kappa-tm_0(\kappa))-(w-tm_0(w))\\
&=tm_0^{(2)}(\kappa)\frac{(w-\kappa)^2}{2}+tm_0^{(3)}(\kappa) \frac{(w-\kappa)^3}{3!}+t\widetilde \cE(w) \frac{(w-\kappa)^4}{4!}\\
&=\left(1+\cE(w)\right)\frac{t m_0^{(2)}(\kappa)(w-\kappa)^2}{2},
\end{split}\end{align}
where
\begin{align}\label{e:defcE}
\cE(w):=\frac{m_0^{(3)}(\kappa)}{m_0^{(2)}(\kappa)}\frac{(w-\kappa)}{3}+\frac{\widetilde \cE(w)}{m_0^{(2)}(\kappa)}\frac{(w-\kappa)^2}{12}.
\end{align}

When $|w-\kappa|\leq 2^{-3}\kappa$, from \eqref{e:defcE} we have 
\begin{align}\label{e:cEww}
|\cE(w)|\leq \frac{|w-\kappa|}{\kappa}+\frac{2^5|w-\kappa|^2}{\kappa^2}
\leq \frac{5|w-\kappa|}{\kappa}
\leq \frac{5}{8}<1,
\end{align}
where we used \eqref{e:m03} and \eqref{e:tcEup} for the first inequality, and $|w-\kappa|\leq 2^{-3}\kappa$ for the second. Next we show that the map $w\mapsto \left(1+\cE(w)\right)t m_0^{(2)}(\kappa)(w-\kappa)^2/2$ is a surjection from $|w-\kappa|\leq 2^{-3}\kappa$ onto  $\{u\in \bC: |u|\leq -2^{-9}t m_0^{(2)}(\kappa)\kappa^2\}$.

Fix any $|u| \leq -2^{-9}t m_0^{(2)}(\kappa)\kappa^2$; we will use Rouch\'{e}'s theorem to compare the two holomorphic functions $t m_0^{(2)}(\kappa)(w-\kappa)^2/2-u$ and $\left(1+\cE(w)\right)t m_0^{(2)}(\kappa)(w-\kappa)^2/2-u$ on the domain $\{w\in \bC: |w-\kappa|\leq 2^{-3}\kappa\}$. First we notice that  $t m_0^{(2)}(\kappa)(w-\kappa)^2/2-u$ has a root with $|w-\kappa|\leq 2^{-3}\kappa$. Next, for any $w$ such that $|w-\kappa|=2^{-3}\kappa$, we can upper bound the difference
\begin{align}\label{e:boundarydiff}\begin{split}
&\phantom{{}={}}\left|\left(\left(1+\cE(w)\right)t m_0^{(2)}(\kappa)(w-\kappa)^2/2-u\right) -\left(t m_0^{(2)}(\kappa)(w-\kappa)^2/2-u\right)\right|\\
&=-t m_0^{(2)}(\kappa)|w-\kappa|^2|\cE(w)|/2\leq -(5/8)\cdot t m_0^{(2)}(\kappa)|w-\kappa|^2/2\leq \left|t m_0^{(2)}(\kappa)(w-\kappa)^2/2-u\right|,
\end{split}\end{align} 
where in the second inequality we used \eqref{e:cEww} indicating that $|\cE(w)|\leq 5/8$, and in the last inequality we used that, when $|w-\kappa|=2^{-3}\kappa$,
\begin{align}
-\frac{3}{8} \frac{t m_0^{(2)}(\kappa)|w-\kappa|^2}{2}
= -\frac{3}{8} \frac{t m_0^{(2)}(\kappa)2^{-6}\kappa^2}{2}
> -2^{-9} t m_0^{(2)}(\kappa)\kappa^2\geq |u|.
\end{align}
The relation \eqref{e:boundarydiff} verifies the conditions of Rouch\'{e}'s theorem, and hence $\left(1+\cE(w)\right)t m_0^{(2)}(\kappa)(w-\kappa)^2/2-u$ also has a root with $|w-\kappa|\leq 2^{-3}\kappa$. 
 
 In particular for any $0\leq x\leq -2^{-9}tm_0^{(2)}(\kappa)\kappa^2$, \eqref{e:constructx} has a solution $w=w(x)$ with $|w-\kappa|\leq 2^{-3}\kappa$:
\begin{align}\label{e:solve0}
w-\kappa=\ri\sqrt{\frac{2x}{-tm_0^{(2)}(\kappa)}}\left(1+\cE(w)\right)^{-1/2},\quad |\cE(w)|\leq 5/8.
\end{align}

To further simplify \eqref{e:solve0}, we will use the following elementary inequalities: for any $|\varepsilon|\leq 5/8$.
\begin{align}\label{e:basic}
\left|\frac{1}{\sqrt{1+\varepsilon}}-1\right|\leq 2\varepsilon,
\quad 
\left|\frac{1}{\sqrt{1+\varepsilon}}-1+\frac{\varepsilon}{2}\right|\leq \frac{8\varepsilon^2}{3},
\end{align}
Using them we can rewrite $\cE(w)$ from \eqref{e:defcE} as
\begin{align}\begin{split}\label{e:error}
&\phantom{{}={}}\left|\left(1+\cE(w)\right)^{-1/2}-1+\frac{m_0^{(3)}(\kappa)}{6m_0^{(2)}(\kappa)}\ri\sqrt{\frac{2x}{-tm_0^{(2)}(\kappa)}}\right|
\\
&\leq \left|\left(1+\cE(w)\right)^{-1/2}-1+\frac{\cE(w)}{2}\right|
+\left|\frac{\cE(w)}{2}-\frac{m_0^{(3)}(\kappa)}{m_0^{(2)}(\kappa)}\frac{(w-\kappa)}{6}\right|
+\frac{m_0^{(3)}(\kappa)}{6|m_0^{(2)}( \kappa)|}\left| w-\kappa -\ri\sqrt{\frac{2x}{-tm_0^{(2)}(\kappa)}}\right|
\\
&\leq \frac{8|\cE(w)|^2}{3}
+\frac{|\widetilde \cE(w)|}{m_0^{(2)}(\kappa)}\frac{|w-\kappa|^2}{24}
+\frac{3}{\kappa}\sqrt{\frac{2x}{-tm_0^{(2)}(\kappa)}}\frac{|\cE(w)|}{3}\\
&\leq \frac{|w-\kappa|^2}{\kappa^2}\frac{200}{3}+\frac{16|w-\kappa|^2}{\kappa^2}+\frac{5|w-\kappa|}{\kappa^2}\sqrt{\frac{2x}{-tm_0^{(2)}(\kappa)}}\\
&\leq \frac{2x}{-t\kappa^2 m_0^{(2)}(\kappa)} \left(\frac{248}{3|1+\cE(w)|}+\frac{5}{|1+\cE(w)|^{1/2}} \right)
\leq \frac{2^9 x}{-t \kappa^2 m_0^{(2)}(\kappa)},
\end{split}\end{align}
where in the second line we used the triangle inquality; in the third line we used \eqref{e:m03}, the definition of $\cE(w)$ from \eqref{e:defcE}, \eqref{e:solve0} and \eqref{e:basic}; in the fourth line we used  \eqref{e:cEww} for the first and third terms and \eqref{e:tcEup} for the second term; in the last line we used  \eqref{e:solve0} and $1/|1+\cE(w)|\leq 8/3$ from \eqref{e:cEww}. 

By plugging \eqref{e:error} into \eqref{e:solve0} and using \eqref{rhoyw} in \Cref{yconvolution}, we conclude that for any $0\leq x\leq -2^{-9}tm_0^{(2)}(\kappa)\kappa^2$ 
\begin{align*}
&\phantom{{}={}}\varrho_t(E_t-x)=\frac{\Im[w]}{t\pi}
=
\frac{1}{\pi}\sqrt{\frac{2x}{-t^3 m_0^{(2)}(\kappa)}}\Re\left[\left(1+ \cE(w)\right)^{-1/2}\right]\\
&=\frac{1}{\pi}\sqrt{\frac{2x}{-t^3 m_0^{(2)}(\kappa)}}\Re\left[1-\frac{m_0^{(3)}(\kappa)}{6m_0^{(2)}(\kappa)}\ri\sqrt{\frac{2x}{-tm_0^{(2)}(\kappa)}} +\widehat \cE(x)\right]=\frac{1}{\pi}\sqrt{\frac{2x}{-t^3 m_0^{(2)}(\kappa)}}(1+\Re[\widehat \cE(x)])
 \end{align*}
where for the last equality, we used that the middle term inside $\Re[\cdot]$ is purely imaginary, and introduced the notation
\begin{align}
\widehat \cE(x)=\left(1+\cE(w(x))\right)^{-1/2}-1+\frac{m_0^{(3)}(\kappa)}{6m_0^{(2)}(\kappa)}\ri\sqrt{\frac{2x}{-tm_0^{(2)}(\kappa)}},\quad 
|\Real \widehat \cE(x)|\leq |\widehat \cE(x)|\leq  \frac{2^9 |x|}{-t \kappa^2 m_0^{(2)}(\kappa)},
\end{align}
where the bound of $\widehat \cE(x)$ is from \eqref{e:error}.
This gives the statement \eqref{e:rhotdensity}. 
Using \eqref{e:upm0} and the lower bound $\kappa\geq (\sfb t/20)^2$ from \eqref{estk2}, get 
$\cA(t)=-2/(t^3 m_0^{(2)}(\kappa))\geq \kappa/t^2\geq \sfb^2/400\geq 2^{-10}\sfb^2$. Using the lower bound \eqref{e:lowm0} and \eqref{estk2}, we have $\fc(t)=-2^{-9}tm_0^{(2)}(\kappa)\kappa^2\geq 2^{-9}t\sfb\kappa^{1/2}/20\geq 2^{-9}t^2\sfb^2/400\geq 2^{-20}(\sfb t)^2$. These two estimates give the last two statements in \eqref{e:defAs}.
\end{proof}

\begin{proof}[Proof of \Cref{p:densitybound}]
We denote the inverted cumulative density (recall from \eqref{gammay})  associated with $\{\mu_t\}_{0\leq t\leq \sfT}$,  by $\gamma_t: [0,1]\mapsto \bR$ for $0\leq t\leq \sfT$.
Let $F(x)=\mu_0([-x,0])\geq \sfb x^{3/2}$ for $\eta_*\leq x\leq \sfT^2$.

The first statement follows from a comparison argument using \Cref{freeconvolutioncompare}. Denote $\breve\mu_0:=\delta_0\in \mathscr P$, and let $\breve\mu_t=\breve\mu_0\boxplus \mu_{\semci}^{(t)}$ be the free convolution of $\breve\mu_0$ with the rescaled semicircle distribution $\mu_{\semci}^{(t)}$. Then recall from \Cref{r:Etbound} (with $A=1$) that $\breve\mu_t$ is given explicitly by
\begin{align*}
\rd \breve\mu_t= \frac{\sqrt{4t-x^2}}{2\pi t}\rd x.
\end{align*}
For the measure valued process $\{\breve\mu_t\}_{0\leq t\leq \sfT}$, we denote the associated inverted cumulative density function by $\breve\gamma_t: [0,1]\mapsto \bR$. From construction, for $y\in [0, 1]$ we have $\breve\gamma_0(y)= 0\geq\gamma_0(y)$. Thus Item \ref{i:convolutionheight} in \Cref{freeconvolutioncompare} implies that for any $t\geq 0$, 
\begin{align*}
E_t= \gamma_t(0)\leq \breve\gamma_t(0)=2\sqrt{t}.
\end{align*}
This gives the first statement in \Cref{p:densitybound}.

For the second claim, we analyze the two scenarios; the first is if $t\leq \sqrt{\fC x}$, and the second is if $t\geq \sqrt{\fC x}$. 
If $t\leq \sqrt{\fC x}$, then we either have $\fC \sqrt \eta_*\leq t\leq \sqrt{ \fC x}$ or $x\geq \fC \eta_*$. In both cases $x\geq \fC \eta_*$. We will prove the claim $\mu_t(E_t-x, E_t)\geq \fc x^{3/2}$ by using the  comparison results from \Cref{freeconvolutioncompare}.

From the discussion above (and our assumption in \Cref{p:densitybound}), in the case that $t\leq \sqrt{\fC x}$, we have $\sfT^2 \geq x\geq \fC \eta_*$. By \Cref{a:densitylowbound} we have  $1\geq \mu_0([-x/\fC , 0])\geq \sfb(x/\fC )^{3/2}$ for any $x/\fC\leq \sfT^2$. We denote 
$A:=\sfb(x/\fC )^{3/2}\leq 1$. Then $\gamma_0(A)\geq -x/\fC $. Let $\widetilde \gamma_0$ be the restriction of $\gamma_0$ to $[0,A]$; then $\wt\gamma_0$ is the inverted height function of a finite measure $\wt\mu_0\in \mathscr P_{\rm fin}$ (by the definition \eqref{gammay}, $\wt\mu_0$ is the pushforward of the Lebesgue measure on $[0,A]$ by $\wt \gamma_0$), and $\wt\mu_0(\bR)=A$. We also denote by $\widetilde \mu_t=\widetilde\mu_0\boxplus \mu_{\semci}^{(t)}$ the free convolution of $\widetilde \mu_0$ with the rescaled semicircle distribution $\mu_{\semci}^{(t)}$. Then $\{\widetilde\mu_t\}_{t\geq 0}$ is a measure-valued process with constant total mass $A$. We denote the associated inverted cumulative density function (recall from \eqref{gammay}) of $\{\widetilde\mu_t\}_{0\leq t\leq \sfT}$ by $\widetilde\gamma_t : [0,A]\mapsto \bR$. Then from construction we have $\widetilde \gamma_0( y)=\gamma_0(y)$ for $y\in[0, A]$. Thus for $y,y'\in [0, A]$ we have $|\gamma_0(y)-\gamma_0(y')|\leq |\widetilde\gamma_0(y)-\widetilde \gamma_0(y')|$, and Item \ref{i:convolutiongap} in \Cref{freeconvolutioncompare} implies that, for any $t\geq 0$,  
\begin{align}\label{e:Gstarbound}
\gamma_t(0)-\gamma_t(A)\leq \widetilde \gamma_t(0)-\widetilde \gamma_t(A).
\end{align} 

Next we use \Cref{i:convolutionheight} in \Cref{freeconvolutioncompare}, by comparing to delta masses at $0$ and $-x/\fC $, respectively, to show that $\wt\gamma_t(0)\leq 2\sqrt{tA}$ and $\widetilde \gamma_t(A)\geq -x/\fC -2\sqrt{tA}$. The proofs of the two statements are very similar, so we will only prove the first one.  Denote $\widehat \mu_0:=A\delta_0$, and let $\widehat \mu_t=\widehat\mu_0\boxplus \mu_{\semci}^{(t)}$ be the free convolution of $\widehat \mu_0$ with the rescaled semicircle distribution $\mu_{\semci}^{(t)}$. Then recall from \Cref{r:Etbound} that $\widehat \mu_t$ is given explicitly by
\begin{align*}
\rd \widehat \mu_t= \frac{\sqrt{4At-x^2}}{2\pi t}\rd x.
\end{align*}
For the measure valued process $\{\widehat\mu_t\}_{0\leq t\leq \sfT}$, we denote the associated inverted cumulative density function by 
 $\widehat \gamma_t: [0,A]\mapsto \bR$. From construction, for $y\in [0, A]$ we have $\widetilde \gamma_0(y)\leq 0=\widehat \gamma_0(y)$. Thus Item \ref{i:convolutionheight} in \Cref{freeconvolutioncompare} implies that for any $t\geq 0$, 
\begin{align*}
\widetilde \gamma_t(0)\leq \widehat \gamma_t(0)=2\sqrt{At}.
\end{align*}
The same argument also gives that $\widetilde \gamma_t(A)\geq -x/\fC -2\sqrt{At}$. Together with \eqref{e:Gstarbound}, we conclude that for $t \le \sqrt{\mathfrak{C}x}$ we have 
\begin{align*}
\gamma_t(0)-\gamma_t(A)\leq \widetilde \gamma_t(0)-\widetilde \gamma_t(A)\leq x/\fC +4\sqrt{At}\leq x\left(\frac{1}{\fC }+4\sqrt{\frac{\sfb}{\fC }}\right)\leq x,
\end{align*}

\noindent where in the third inequality we used $A= \sfb (x/\fC )^{3/2}$ and $t\leq \sqrt{\fC x}$; in the last inequality we used $\fC =2^{21}/ \sfb^2\geq \max\{2, 2^6 \sfb\}$. Hence, for $t\leq \sqrt{\fC x}$, we have
\begin{align}\label{e:lowdd1}
\mu_t([E_t-x, E_t])\geq A=\sfb(x/\fC )^{3/2}.
\end{align}

For $t\geq \sqrt{\fC x}$, either $t\geq \fC \sqrt{\eta_*}$ or $x\geq \fC \eta_*$. In both cases we have $t\geq \fC \sqrt{\eta_*}\geq (30/\sfb)\sqrt{\eta_*}$. Recall $\xi(t)$, $\cA(t)$, $\mathcal{E}(t)$, and $\fc(t)$ from \Cref{l:densityt}; for $(30/\sfb)\sqrt{\eta_*}\leq t\leq \sfT$, we have that
\begin{align}\label{e:rhotylow}
 \varrho_t(y)= (1+\cE(E_t-y))\frac{\sqrt{\cA(t)(E_t-y)}}{\pi}\geq \frac{1}{2}\cdot \frac{\sfb}{2^{5}}\frac{\sqrt{E_t-y}}{\pi},
 \end{align}
provided that $|y|\leq 2^{-21}\sfb^2 t^2\leq \fc(t)/2$. 
 Moreover, for $t\geq \sqrt{\fC x}$, we have $x\leq t^2/\fC \leq 2^{-21}\sfb^2 t^2\leq \fc(t)/2$, where we used that $\fC \geq 2^{21}/\sfb^2$.  Then integrating \eqref{e:rhotylow} gives
\begin{align}\label{e:lowdd2}
\mu_t([E_t-x,E_t])
&=
\int_0^x \varrho_t(E_t-y)\rd y
\geq \int_0^x \frac{\sfb}{2^{6}\pi}\sqrt{y}
\rd y= \frac{\sfb}{96\pi}x^{3/2}.
\end{align}
The estimates \eqref{e:lowdd1} and \eqref{e:lowdd2} together give the first statement in \Cref{p:densitybound}.

Next we prove \eqref{e:Immbound}. Recall our assumption that we have either $t\geq \fC \sqrt \eta_*$, or $\kappa\geq \fC \eta_*$, implying that 
\begin{align*}
\Im[m_t(E_t+\kappa+\ri\eta)]
&=\int_{\bR} \frac{\eta\varrho_t(x)\rd x}{(E_t+\kappa-x)^2+\eta^2}
\geq \int_{E_t-(\kappa+\eta)}^{E_t}\frac{\eta\varrho_t(x)\rd x}{(E_t+\kappa-x)^2+\eta^2}\\
&\geq \int_{E_t-(\kappa+\eta)}^{E_t}\frac{\eta\varrho_t(x)\rd x}{(2\kappa+\eta)^2+\eta^2}
\geq \frac{\fc  \eta(\kappa+\eta)^{3/2}}{2(2\kappa+2\eta)^2}
=\frac{\fc  \eta}{8\sqrt{\kappa+\eta}},
\end{align*}
where in the second statement we restricted the integration on $[E_t-\kappa-\eta, E_t]$; in the third statement we used that for $x\in [E_t-\kappa-\eta, E_t]$, we have $E_t+\kappa-x\leq 2\kappa+\eta$; in the fourth statement inequality we used $\kappa+\eta\leq \sfT^2$ from our assumption, and $\mu_t([E_t-(\kappa+\eta), E_t])\geq \fc  (\kappa+\eta)^{3/2}$ from the second statement in \Cref{p:densitybound}. 
\end{proof}

\section{Edge Rigidity and Universality}\label{s:rigidity}

In this section we study Dyson Brownian motion, with initial data $\mu_0=\mu_0^{(n)}=(1/n)\sum_{i=1}^n \delta_{\lambda_i(0)}$ satisfying  \Cref{a:densitylowbound}, and prove \Cref{t:main} and \Cref{t:universality}. Recall $\mu_t=\mu_0\boxplus \mu_{\rm sc}^{(t)}$, $E_t=\max (\supp \mu_t)$, and the constants $\fc$ and $\fC $ from \Cref{p:densitybound}. 
We will first study the case that
\begin{flalign}
	\label{etan} 
	\fC ^2\eta_*\geq (\log n)^{15}/n^{2/3}.
\end{flalign}
Under the assumption \eqref{etan}, define the function 
\begin{align}\label{e:deff}
f(t)=\max\{ \fC \sqrt{\eta_*}-\fc   \max\{0, t-\fC \sqrt{\eta_*}\}/8,(\log n)^{15/2} n^{-1/3}\}^2.
\end{align}
Then for $t\leq \fC \sqrt{\eta_*}$, we have by \eqref{etan} that $f(t)=\fC^2 \eta_*$; for $t\geq \fC \sqrt \eta_*$, $f(t)$ decreases to $(\log n)^{15}/n^{2/3}$, and
\begin{align}\label{e:ftfsdiff}
(f(s))^{1/2}-(f(t))^{1/2}\leq \frac{\fc  (t-s)}{8},\quad \fC \sqrt{\eta_*}\leq s\leq t.
\end{align}

The following proposition states that $\lambda_1(t)\leq E_t+f(t)$ for $0\leq t\leq \sfT$ with overwhelming probability. \Cref{t:main} is an easy consequence of it. The proof of \Cref{p:edgerigidity} will be given in \Cref{s:prove31}. 
	
	\begin{prop}\label{p:edgerigidity}
	Consider Dyson Brownian motion $\bm\la(t)$ as in \eqref{e:DBM}. We assume its initial data $\mu_0=\mu_0^{(n)}=(1/n)\sum_{i=1}^n \delta_{\lambda_i(0)}$ satisfies \Cref{a:densitylowbound}. We further assume that \eqref{etan} holds and $\sfT\geq \fB\sqrt {\eta_*}=2^{60}\sfb^{-6}\sqrt{\eta_*}$. Recall $\mu_t=\mu_0\boxplus \mu_{\rm sc}^{(t)}$ and $E_t=\max (\supp \mu_t)$ from \Cref{p:densitybound}. There exists a large constant $C=C(\sfb,\sfT)>1$, such that, with probability $1-Ce^{-(\log n)^2}$, we have for any $0\leq t\leq \sfT$ that 
	\begin{align}\label{e:ubt}
	\lambda_1(t)\leq E_t+f(t),
	\end{align}
where $f(t)$ is as defined in \eqref{e:deff}.
	\end{prop}

\begin{proof}[Proof of \Cref{t:main}]
Recall $\fc$ and $\fC$ from \Cref{p:densitybound}. We first discuss the case when \eqref{etan} holds.
From the construction of $f(t)$ in \eqref{e:deff} and discussion after it, for $t\geq 2^{60}\sfb^{-6}\sqrt{\eta_*}\geq (1+8/\fc)\fC\sqrt{\eta_*}$,
 we have $f(t)=\max\{ 0,(\log n)^{15/2} n^{-1/3}\}^2=(\log n)^{15} n^{-2/3}$. The claim \eqref{e:edgerig} thus follows from \eqref{e:ubt}.  When \eqref{etan} does not hold, then \Cref{a:densitylowbound} still holds by replacing $\eta_*$ with $(\log n)^{15}/(\fC^2 n^{2/3})$. If $\eta_*=(\log n)^{15}/(\fC^2 n^{2/3})$, we can construct $f(t)$ as in \eqref{e:deff}, then $f(t)=(\log n)^{15} n^{-2/3}$ for any $t\geq 0$. Thus \Cref{p:edgerigidity} gives that  with probability $1-Ce^{-(\log n)^2}$, $\lambda_1(t)\leq E_t+(\log n)^{15}/ n^{2/3}$ for any $0\leq t\leq \sfT$, and the claim \eqref{e:edgerig}  follows.
 \end{proof}

\subsection{Characteristic flow}\label{s:Cf}

In the remainder of \Cref{s:rigidity}, we will use the notations $\sfT, \sfb, \eta_*$ and the measure $\mu_0=\mu_0^{(n)}$ from \Cref{a:densitylowbound}. 
We recall the functions $\cA(t), \fc(t)$, and constants $\fc$ and $\fC$ from \Cref{p:densitybound}. We also recall the functions $f(t)$ from \eqref{e:deff}. We denote the free convolution of $\mu_0$ with rescaled semicircle distribution by as $\mu_t$ (recall from \Cref{TransformConvolution}). We let $E_t=\max (\supp \mu_t)$ and denote the Stieltjes transform of $\mu_t$ by $m_t(z)$.  We recall the set $\Lambda_t$ from \eqref{mtlambdat}. 

We can reformulate the defining relation \eqref{mt} of the free convolution by introducing the characteristic flow:
\begin{align}\label{e:flow}
m_t(z_t(u))=m_0(u),\quad z_t(u)=u-tm_0(u)=:E_t+\kappa_t(u)+\ri\eta_t(u),\quad u\in \Lambda_t.
\end{align}
When the context is clear, we will simply write $z_t(u), \kappa_t(u), \eta_t(u)$ as $z_t, \kappa_t, \eta_t$. 
It follows from \eqref{e:flow} that $z_t=u-tm_t(z_t)\in \bH$, since $z_t : \Lambda_t \rightarrow \mathbb{H}$ is a bijection, by \Cref{mz}. The characteristic flow satisfies the fundamental property that 
\begin{flalign}\label{e:m0=ms}
	m_t (z_t) \qquad \text{is constant in $t \ge 0$},
\end{flalign} 

\noindent which we will often use without comment.

In the following we collect  some basic facts about the Stieltjes transform and the characteristic flows, which will be used repeatedly in the rest of this section. Take any point $u\in \bH$, and time $t\geq 0$, such that at time $t$ we have $z_t=z_t(u)=E_t+\kappa_t+\ri\eta_t\in \Lambda_t$. Then for any time $0\leq s\leq t$, using $z_s=u-sm_0(u)=u-sm_t(z_t)$ (from \eqref{e:flow}) we have 
\begin{align}\label{e:etas}
z_s=z_t+(t-s)m_t(z_t),\quad  \eta_s=\eta_t+(t-s) \Im[m_t(z_t)]\geq \eta_t,
\end{align}
where the last inequality follows from $\Im[m_t(z_t)]=\Im[m_0(u)]\geq 0$.

We collect some further properties of the characteristic flow in the following two lemmas.
\Cref{c:kappa} states that if $\kappa_t\geq 0$, then $\kappa_s\geq 0$ for all $0\leq s\leq t$; it also gives lower bounds on $\kappa_s-\kappa_t$. \Cref{c:domainic} states that if $\kappa_t\geq f(t)$ (recall from \eqref{e:deff}), then $\kappa_s\geq f(s)$ for all $0\leq s\leq t$; it also gives lower bounds on $\kappa_s-f(s)$.

\begin{lem}\label{c:kappa}
Adopt \Cref{a:densitylowbound}, and recall the constants $\fc$ and $\fC  $ from \Cref{p:densitybound}. Take $u\in \Lambda_t$, and fix real numbers $t > 0$ and $0\leq s\leq t$. Denote $z_s=z_s(u)=E_s+\kappa_s+\ri\eta_s$, and assume that $\kappa_s\leq \sfT^2$ (recall $\sfT$ from \Cref{a:densitylowbound}). If  $\kappa_t\geq 0$ then, for $0\leq s\leq t$, we have
\begin{align}\label{e:ksktrelation}
\kappa_s-\kappa_t\geq \left(\frac{\eta_s}{\eta_t}-1\right)\kappa_t= (t-s)\Im[m_t(z_t)]\frac{\kappa_t}{\eta_t}\geq 0. 
\end{align}

\noindent Moreover, if $\kappa_t\geq 0$ then, for $\fC \sqrt{\eta_*}\leq s\leq t$, we have
\begin{align}
	\label{e:ktksdiff}
\kappa_s^{1/2}-\kappa_t^{1/2}\geq \frac{\fc(t-s)}{4}.
\end{align}
\end{lem}

\begin{lem}\label{c:domainic}
Adopt \Cref{a:densitylowbound}, and recall the constants $\fc$ and $\fC$ from \Cref{p:densitybound}. Take $u\in \Lambda_t$, and fix real numbers $t > 0$ and $0\leq s\leq t$. Denote $z_s=z_s(u)=E_s+\kappa_s+\ri\eta_s$, and assume that $\kappa_s\leq \sfT^2$ (recall $\sfT$ from \Cref{a:densitylowbound}). If $\kappa_t\geq f(t)$ (recall from \eqref{e:deff}) then, for $\fC \sqrt{\eta_*}\leq s\leq t$, we have
\begin{align}\label{e:ksfsdiff1}
\kappa_s-f(s)\geq \frac{(t-s)}{2}\Im[m_t(z_t)]\frac{\kappa_{ t}}{\eta_{t}};
\end{align}
Moreover, if $\kappa_t\geq f(t)$ then, for $0\leq s\leq \widetilde t:=\min\{t,\fC \sqrt{\eta_*}\}$, we have
\begin{align}\label{e:ksfsdiff2}
\kappa_s-f(s)\geq 
(\widetilde t-s)\Im[m_t(z_t)]\frac{\kappa_{\widetilde t}}{\eta_{\widetilde t}}
\end{align}
\end{lem}

\begin{proof}[Proof of \Cref{c:kappa}]
We will first show that, if $\kappa_t \ge 0$, then $\kappa_s\geq 0$ for $0\leq s\leq t$, and we will also estimate its derivative. Notice 
\begin{align}\label{e:delzs}
\del_s \Re[z_s]=\del_s \Re[u-sm_0(u)]=-\Re[m_0(u)]=-\Re[m_s(z_s)],
\end{align}
where the first statement is from the definition  \eqref{e:flow} that $z_s=u-sm_0(u)$;
the second statement is from taking derivative with respect to $s$;
the last statement follows from \eqref{e:m0=ms}. Further recall the construction from Item \ref{i:edge} in \Cref{yconvolution} that there exists $w_s\in \bR$ such that $m_0'(w_s)=1/s$ and $E_s=w_s-sm_0(w_s)=z_s(w_s)$. By taking the time derivative on both sides, we get a similar equation as \eqref{e:delzs}, namely,
\begin{align}\label{e:delEs}\begin{split}
\del_s E_s&=\del_s w_s-sm_0'(w_s)\del_s w_s-m_0(w_s)\\
&=\del_s w_s(1-sm_0'(w_s))-m_0(w_s)=-m_0(w_s)=-m_s(E_s),
\end{split}\end{align}
where the first statement follows from taking the derivative with respect to $s$ on both sides of $E_s=w_s-sm_0(w_s)$;
the second and third statements use $m_0'(w_s)=1/s$; and the last statement uses \eqref{e:m0=ms} and $E_s=z_s(w_s)$.
Then by taking difference between \eqref{e:delzs} and \eqref{e:delEs}, we get
\begin{align}\begin{split}\label{e:dkappa}
\del_s \kappa_s
&=-(\Re[m_s(z_s)]-m_s(E_s))
=-((\Re[m_s(z_s)]-m_s(\Re[z_s])+(m_s(\Re[z_s])-m_s(E_s)))\\
&
= -\eta_s^2\int_\bR \frac{\varrho_s(x)\rd x}{((E_s+\kappa_s-x)^2+\eta_s^2)(E_s+\kappa_s-x)}
-\kappa_s \int_\bR \frac{\varrho_s(x)\rd x}{(E_s-x)(E_s+\kappa_s-x)},
\end{split}\end{align}
where the first statement follows from $\kappa_s=\Re[z_s]-E_s$; the second statement is from separating the difference into two; in the third statement we used $\Re[z_s]=E_s+\kappa_s$ and $\Re[m_s(z_s)]=\int_\bR (x-E_s-\kappa_s)\varrho_s(x)\rd x/((E_s+\kappa_s-x)^2+\eta_s^2)$.
By our assumption, we have $\kappa_t\geq 0$. Next we will show that $\kappa_s\geq 0$ for all $0\leq s\leq t$. Otherwise, there exists $\tau=\sup\{0\leq s\leq t: \kappa_s< 0\}$. Since $\kappa_s$ is continuous, we have that $\tau\in [0,t)$,  $\kappa_s\geq 0$ for $s\in [\tau,t]$ and $\kappa_\tau=0$. Thus for $s\in [\tau,t]$, \eqref{e:dkappa} gives that
\begin{align}\label{e:dkappa25}
\del_s \kappa_s
\leq 0,
\end{align}
since $E_s+\kappa_s\geq E_s=\max\supp \varrho_s$, and $\kappa_s\geq 0$. By integrating \eqref{e:dkappa25} from $\tau$ to $t$, we conclude that $\kappa_\tau\geq \kappa_t\geq 0$, which contradicts with our choice of $\tau$. Thus we conclude 
\begin{align}\label{e:kappas>0}
\kappa_s\geq 0  \quad \text{and $\kappa_s$ is decreasing for}\quad 0\leq s\leq t,
\end{align}
where the second statement follows from $\del_s\kappa_s<0$ (from \eqref{e:dkappa}). 

We can also lower bound the second integral in \eqref{e:dkappa} by 
\begin{align}\label{e:dkappa2}
\int_\bR \frac{\varrho_s(x)\rd x}{(E_s-x)(E_s+\kappa_s-x)}
\geq \int_\bR \frac{\varrho_s(x)\rd x}{|z_s-x|^2}=\frac{1}{\eta_s}\Im\int_\bR \frac{\varrho_s(x)\rd x}{x-z_s}=\frac{\Im[m_s(z_s)]}{\eta_s},
\end{align}
where in the first statement we used $z_s=E_s+\kappa_s+\ri\eta_s$ with $\kappa_s\geq 0$ from \eqref{e:kappas>0} and $\max\supp \varrho_s= E_s$, thus $|z_s-x|\geq E_s+\kappa_s-x\geq E_s-x$; the last two statements are from the definition of Stieltjes transform \eqref{mz0}. Combining \eqref{e:dkappa} and \eqref{e:dkappa2}, we get
\begin{align}\label{e:delks}
\del_s \kappa_s\leq
-\frac{\kappa_s}{\eta_s}\Im[m_s(z_s)].
\end{align}
Recall from the definition of the characteristic flow \eqref{e:flow} and \eqref{e:etas} that the derivative for the imaginary part $\del_s \eta_s$ satisfies
$
\del_s \eta_s=-\Im[m_s(z_s)].
$
By plugging this relation to \eqref{e:delks}, we conclude 
\begin{align*}
\del_s \kappa_s\leq
\frac{\kappa_s}{\eta_s}\del_s \eta_s, \qquad \text{so that} \qquad \frac{\del_s \eta_s}{\eta_s}\geq \frac{\del_s\kappa_s}{\kappa_s},
\end{align*}

\noindent and hence for $0 \le s \le t$ we obtain 
\begin{align}\label{e:ksktrelation2}
\kappa_s\geq \frac{\eta_s}{\eta_t}\kappa_t. 
\end{align}
Recall from \eqref{e:etas} that $(\eta_s-\eta_t)/\eta_t=(t-s)\Im[m_t(z_t)]/\eta_t$. Hence, the claim \eqref{e:ksktrelation} follows from \eqref{e:ksktrelation2}:
\begin{align}
\kappa_s-\kappa_t\geq \left(\frac{\eta_s}{\eta_t}-1\right)\kappa_t=(t-s)\Im[m_t(z_t)]\frac{\kappa_t}{\eta_t}. 
\end{align}

If we moreover have that $s\geq \fC \sqrt{\eta_*}$, then we can also  bound the integral in \eqref{e:dkappa} by
\begin{align}\begin{split}
\label{e:dkappa3}
 -\kappa_s\int_\bR\frac{\varrho_s(x)\rd x}{(E_s-x)(E_s+\kappa_s-x)}
\leq -\kappa_s\int_{E_s-\kappa_s}^{E_s} \frac{\varrho_s(x)\rd x}{2\kappa_s^2} =-\frac{\mu([E_s-\kappa_s, E_s])}{2\kappa_s}\leq-\frac{1}{2}\fc\kappa_s^{1/2},
\end{split}\end{align}
where we in the first inequality we restricted the integral on $[E_s-\kappa_s, E_s]$ and used $\kappa_s \geq E_s-x$ and $2\kappa_s\geq E_s+\kappa_s-x$ on this interval; in the last two inequalities, we used $\kappa_s\leq \sfT^2$ and $\mu_s([E_s-\kappa_s, E_s])\geq \fc\kappa_s^{3/2}$ from \Cref{p:densitybound}. Plugging \eqref{e:dkappa3} into \eqref{e:dkappa} we get
\begin{align}
\del_s\kappa_s \leq -\frac{\fc}{2}\kappa_s^{1/2},\quad \fC \sqrt{\eta_*}\leq s\leq t.
\end{align}
Solving it we obtain 
\begin{align}\label{e:gap}
\kappa_s^{1/2}-\kappa_t^{1/2}\geq \frac{\fc(t-s)}{4},\quad \fC \sqrt{\eta_*}\leq s\leq t. 
\end{align}
This finishes the proof of \eqref{e:ktksdiff}
\end{proof}

\begin{proof}[Proof of \Cref{c:domainic}]
For $t\geq s\geq \fC \sqrt{\eta_*}$, the estimates \eqref{e:ksktrelation} and \eqref{e:ktksdiff} imply
\begin{align}\label{e:kskttwodiff}
\kappa_s-\kappa_t\geq (t-s)\Im[m_t(z_t)]\frac{\kappa_{t}}{\eta_{ t}},\quad \kappa_s-\kappa_t\geq \frac{\fc(t-s)}{4}(\sqrt{\kappa_s}+\sqrt{\kappa_t}),
\end{align}

From our construction of $f(t)$, using \eqref{e:ftfsdiff} and \eqref{e:ktksdiff}, and recalling that $\kappa_t\geq f(t)$, we have
\begin{align}\label{e:ksbound0}
\sqrt{f(s)}\leq \frac{\fc(t-s)}{8}+\sqrt{f(t)}\leq  \frac{\fc(t-s)}{8}+\sqrt{\kappa_t}\leq \sqrt{\kappa_s}, \quad \fC \sqrt{\eta_*}\leq s\leq t,
\end{align}

\noindent Rearranging \eqref{e:ftfsdiff}, together with \eqref{e:ksbound0}, gives $\kappa_s\geq f(s)\geq f(t)$. Thus, 
\begin{align}\label{e:ksbound}
f(s)-f(t)\leq \frac{\fc}{8}(t-s)(\sqrt{f(s)}+\sqrt{f(t)})\leq \frac{\fc(t-s)}{8}(\sqrt{\kappa_s}+\sqrt{\kappa_t}),\quad \fC \sqrt{\eta_*}\leq s\leq t,
\end{align}
Put everything together, we get
\begin{align*}
\phantom{{}={}}\kappa_s-f(s)
& \geq \frac{1}{2}(\kappa_s-\kappa_t)+\frac{1}{2}(\kappa_s-\kappa_t)-(f(s)-f(t))\\
&\geq \frac{(t-s)}{2}\Im[m_t(z_t)]\frac{\kappa_{t}}{\eta_{ t}}
+\frac{\fc(t-s)}{8}(\sqrt{\kappa_s}+\sqrt{\kappa_t})
-\frac{\fc(t-s)}{8}(\sqrt{\kappa_s}+\sqrt{\kappa_t}) \\
& =  \frac{(t-s)}{2}\Im[m_t(z_t)]\frac{\kappa_{t}}{\eta_{ t}},
\end{align*}
where in the first inequality we used our assumption $\kappa_t\geq f(t)$ in \Cref{c:domainic}; in the second inequality we used  \eqref{e:kskttwodiff} to bound the first two terms and \eqref{e:ksbound} to bound the last term. 
This gives \eqref{e:ksfsdiff1}.

Next we prove \eqref{e:ksfsdiff2}, by considering two cases.
If $t=\widetilde t$, then $f(\widetilde t)=f(t)\leq \kappa_t=\kappa_{\widetilde t}$. If $t>\widetilde t$, then $\widetilde t=\fC \sqrt{\eta_*}$ and \eqref{e:ksfsdiff1} (applied with the $s$ there equal to $\widetilde{t}$ here) implies $f(\widetilde t)\leq \kappa_{\widetilde t}$. In both cases we have $f(\widetilde t)\leq \kappa_{\widetilde t}$. 
For $s\leq \widetilde t$ we have $f(s)=f(\widetilde t)\leq \kappa_{\widetilde t}$, and \eqref{e:ksktrelation} (applied with the $t$ there equal to $\widetilde{t}$ here, observing that $\kappa_{\tilde{t}} \ge \kappa_t \ge f(t) \ge 0$ holds) implies
\begin{align}
\kappa_s-f(s) \ge \kappa_s-\kappa_{\widetilde t} \ge  ({\widetilde t}-s)\Im[m_s(z_s)]\frac{\kappa_{\widetilde t}}{\eta_{\widetilde{s}}} =  ({\widetilde t}-s)\Im[m_t(z_t)]\frac{\kappa_{\widetilde t}}{\eta_{\widetilde t}}.
\end{align}

\noindent This finishes the proof of \eqref{e:ksfsdiff2}. 
\end{proof}

\begin{figure}
\centering
\includegraphics[width=1\textwidth, trim=2cm 1cm 2cm 0cm, clip]{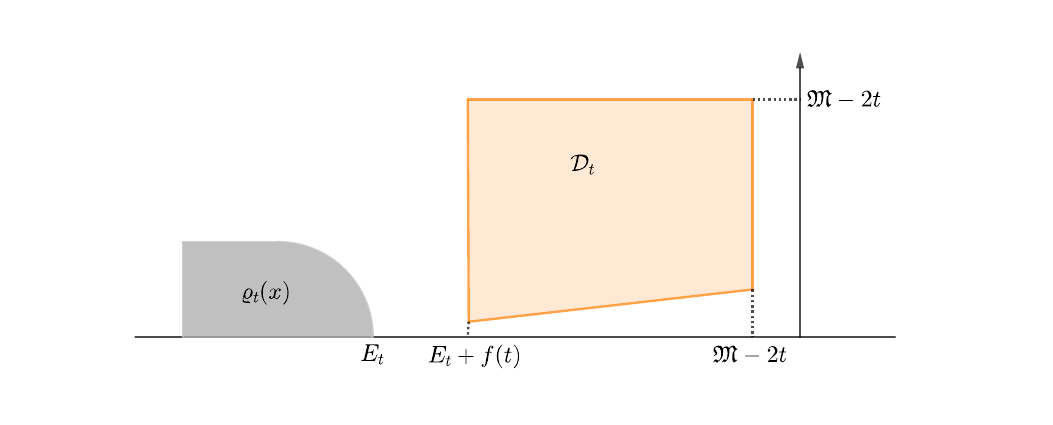}
\caption{The spectral domain $\cD_t$ constructed in \eqref{e:defDt}. }
\label{f:Spectral_domain}
\end{figure}

Fix the real number $\fM=6(\sfT+\fC^2\eta_*)$, and recall $f(t)$ from \eqref{e:deff}. For any real number $t \in [0, \mathsf{T}]$, we define the spectral domain $\cD_t=\cD_t(\mu_t)$ (see \Cref{f:Spectral_domain}) by 
\begin{align}\label{e:defDt}
\cD_t\deq\left\{z\in \bH:  E_t+f(t)\leq \Re[z]\leq \fM-2t, \frac{1}{(\log n)n\Im[m_t(z)]}\leq \Im[z]\leq \fM-2t  \right\}
\end{align}
 Roughly speaking, the information of the Stieltjes transform $m_t(z)$ on $\cD_t$ can be used to detect particles of distance $f(t)$ from the spectral edge $E_t$. The following lemma collects some properties of the domain $\cD_t$. 
  
\begin{lem}\label{c:Dtproperty}
Adopt \Cref{a:densitylowbound}. Recall the domain $\cD_t$ from \eqref{e:defDt}. For  $n$ large enough and any time $0\leq t\leq \sfT$, the following hold.
\begin{enumerate}
\item \label{i:Im}$\cD_t$ is non-empty. For any $z=E_t+\kappa+\ri\eta\in \cD_t$, we have
\begin{align}\label{e:Immz1}
\fM\geq \kappa\geq f(t)\geq \frac{1}{n^{2/3}}, \quad \fM\geq \eta\geq \frac{1}{n^2}, \quad \kappa+\eta\leq \sfT^2, \quad  \Im[m_t(z)]\geq \frac{\fc\eta}{8\sqrt{\kappa+\eta}},
\end{align}
where $\fc$ is from \Cref{l:densityt}, and 
\begin{align}\label{e:Immz2}
\frac{(\log n)^{7/2}}{n\sqrt{(\kappa+\eta)\eta}}\leq \frac{\Im[m_t(z)]}{\log n}.
\end{align}
\item \label{i:zs}If $z_t(u)=E_t+\kappa_t+\ri\eta_t\in \cD_t$, then $z_s(u)\in \cD_s$ for any $0\leq s\leq t$.
\end{enumerate}
\end{lem}

\noindent By the second part of \Cref{c:Dtproperty}, if for any $u = z_0 (u) \in \cD_0$ we define
\begin{align}\label{e:deftu}
\ft(u):=\sfT\wedge \sup\{t\geq 0: z_t(u)\in \cD_t\},
\end{align}

\noindent then $z_s(u)\in \cD_s$ for any $0\leq s\leq \ft(u)$.

\begin{proof}[Proof of \Cref{c:Dtproperty}]
	For Item \ref{i:Im},  $\fM=6(\sfT+\fC^2 \eta_*)\geq 2\sfT +2\sqrt{\sfT}+\fC^2\eta_*\geq 2t+E_t+f(t)$, where we used that $t\leq \sfT$, $E_t\leq 2\sqrt t\leq 2\sqrt{\sfT}$ from the first statement in \Cref{p:densitybound}, and $f(t)\leq f(0)=\fC^2 \eta_*$ from \eqref{e:deff}. Moreover for any $z\in \bH$ with $E_t+f(t)\leq \Re[z]\leq \fM-2t$ and $\Im[z]\leq \fM-2t$ we have 
	\begin{align}
	\Im[m_t(z)]=\int_\bR\frac{\rd \mu_t(x)}{|x-z|^2}\geq \int^{E_t}_{E_t-1}\frac{\rd \mu_t(x)}{(2\fM)^2}\geq \frac{\fc}{4\fM^2},
	\end{align}
	where the first inequality is from $|z-x|\leq 2\fM$ for $x\in [E_t, E_t-1]$; and the last inequality follows from the second statement of \Cref{p:densitybound} (with $x=1$). 
	Thus for $n$ large enough, $\fM-2t> 1>1/(n \Im[m_t(z)]\log n)$, and $\cD_t$ is not empty. 
	
	If $E_t+\kappa+\ri\eta\in \cD_t$, then we have $\eta, \kappa\leq \fM$ from the construction of $\cD_t$ in \eqref{e:defDt}; $\kappa\geq f(t)\geq 1/n^{2/3}$ from the definition \eqref{e:deff};
and $\eta\geq 1/((\log n) n\Im[m_t(z)])\geq1/((\log n) n(1/\kappa))\geq 1/n^2$ (where the first statement is from the definition of $\cD_t$ \eqref{e:defDt}; the second statement is from $\max (\supp \mu_t)=E_t$ and \Cref{l:STproperty} that $|m_t(z)|\leq 1/\dist(z, \supp \mu_t)\leq 1/\kappa$; the last statement follows from $\kappa\geq n^{-2/3}$) provided $n$ is large enough. This gives the first two statement in \eqref{e:Immz1}. For the third statement, we recall from our assumption from \eqref{p:edgerigidity} that $\sfT\geq (2^{60}/\sfb^6)\sqrt\eta_*\geq 2^{10}\fC\sqrt{\eta_*}$ (we recall that $\fC=2^{21}/\sfb^2$ from \eqref{p:densitybound} and  $\sfb,\eta_*<1$ from \Cref{a:densitylowbound}). Then $\kappa+\eta\leq 2\fM=12(\sfT+\fC^2\eta_*)\leq \sfT^2$ (where we also used $\sfT\geq 100$ from \Cref{a:densitylowbound}.)
The last relation in \eqref{e:Immz1} follows from \eqref{e:Immbound} and $\kappa+\eta\leq \sfT^2$. 

For  \eqref{e:Immz2}, by the definition of $\cD_t$, we have
\begin{align}
\eta\Im[m_t(E_t+\kappa+\ri\eta)]\geq \frac{1}{(\log n)n}.
\end{align}
There are two cases; we either have $\eta\geq (\log n)^4n^{-2/3}$ or $\eta\leq (\log n)^4n^{-2/3}$. If $\eta\geq (\log n)^4n^{-2/3}$, then
\begin{align}
\sqrt{(\kappa+\eta)\eta}\Im[m_t(E_t+\kappa+\ri\eta)]
\geq \sqrt{(\kappa+\eta)\eta}\frac{\fc\eta}{8\sqrt{\kappa+\eta}}=\frac{\fc}{8}\eta^{3/2}\geq \frac{(\log n)^{9/2}}{n},
\end{align}
provided $n$ is large enough.
If $\eta\leq (\log n)^4n^{-2/3}$, then $\kappa\geq f(t)\geq (\log n)^{15}n^{-2/3}\geq (\log n)^{11}\eta$, and
\begin{align}
\sqrt{(\kappa+\eta)\eta}\Im[m_t(E_t+\kappa+\ri\eta)] \ge \sqrt{\frac{\kappa}{\eta}} \cdot \eta\Im[m_t(E_t+\kappa+\ri\eta)]\geq \sqrt{\frac{\kappa}{\eta}} \cdot \frac{1}{n\log n}\geq \frac{(\log n)^{9/2}}{n}.
\end{align}
This finishes the proof of  \eqref{e:Immz2}.

For Item \ref{i:zs}, for observe that \Cref{c:domainic} gives if $\kappa_t\geq f(t)$ then $\kappa_s\geq f(s)$ for $s\leq t$. Since $\eta_s\geq \eta_t$ from \eqref{e:etas}, if $\eta_t\geq 1/((\log n)n\Im[m_t(z_t)])$, then $\eta_s\geq \eta_t\geq 1/((\log n)n\Im[m_t(z_t)])=1/((\log n)n\Im[m_s(z_s)])$. Thus, $\Real z_s (u)$ and $\Imaginary z_s (u)$ satisfy the lower bounds required for $z_s (u)$ to be in $\mathcal{D}_s$. 

Next we prove by contradiction that $\Re[z_s(u)]\leq \fM-2s$ and $\Im[z_s(u)]\leq \fM-2s$.
Otherwise, assume $\max\{\Re[z_s(u)], \Im[z_s(u)]\}>\fM-2s$. Recall our choice $\fM=6(\sfT+\fC^2\eta_*)\geq 6\sfT$ and $E_s\leq 2\sqrt s\leq 2\sqrt \sfT$ from the first statement of \Cref{p:densitybound}.  Then $\fM-2\sfT-E_s\geq \fM-3\sfT\geq \fM/2$ (we used that $\sfT\geq 100$ from \Cref{a:densitylowbound} and $\fM-2\sfT\geq \fM/2$). Thus $\dist(z_s(u), (-\infty, E_s])\geq \fM/2$; so, $|m_s(z_s)|\leq 1/(\dist(z_s(u), (-\infty, E_t])\leq 2/\fM$ (from \Cref{l:STproperty}), and $|z_s-z_t|=|t-s||m_s(z_s)|\leq 2(t-s)/\fM$ (here we used the definition of characteristic flow \eqref{e:etas} and \eqref{e:m0=ms}). Then $\Re[z_t(u)]\geq \Re[z_s(u)]-2(t-s)/\fM$ and $\Im[z_t(u)]\geq \Im[z_s(u)]-2(t-s)/\fM$. Therefore, $\max\{\Re[z_t(u)], \Im[z_t(u)]\}\geq \max\{\Re[z_s(u)], \Im[z_s(u)]\}-2(t-s)/\fM> \fM-2s-2(t-s)/\fM>\fM-2t$ (we used our hypothesis $\max\{\Re[z_s(u)], \Im[z_s(u)]\}>\fM-2s$ and $\fM\geq 1$), which contradicts the fact that $z_t(u)\in \cD_t$. It follows that $\Real z_s (u)$ and $\Imaginary z_s (u)$ satisfy the upper bounds required for $z_s (u)$ to be in $\mathcal{D}_s$, and $z_s(u)\in \cD_s$.
\end{proof}

We next define the lattice on the upper half-plane $\bH$ given by
\begin{align}\label{def:L}
\mathcal L=\left\{E+\ri \eta\in \dom_0: E\in \bZ/ n^{8}, \eta\in \bZ/n^{8}\right\}.
\end{align}

\begin{lem}\label{c:Lapproximate}
Adopt \Cref{a:densitylowbound}.
For any $t\in [0,\sfT]$ and $w\in \dom_t$, there exists some lattice point $u\in \mathcal L\cap z_t^{-1}(\dom_t)$, such that
$
|z_t(u)-w|\leq n^{-5},
$
provided $n$ is large enough. 
\end{lem}
\begin{proof}[Proof of \Cref{c:Lapproximate}]
It follows from  \Cref{c:Dtproperty} that if $z_t(u)\in \dom_t$, then $u\in \cD_0$, and  $z_t(u)=u-tm_0(u)$ with $u\in \dom_0$. In particular $|u-E_0|\geq f(0) = \mathfrak{C}^2 \eta_* \geq1/n^{2/3}$ (by \eqref{etan}), and 
$\del_u z_t(u)=1-tm_0'(u)\leq 1+t/|u-E_0|^2\leq 1+\sfT n^{4/3}$ (using \Cref{l:STproperty}). Thus
$z_t(u)$ is Lipschitz in $u$ with Lipschitz constant bounded by $(1+\sfT n^{4/3})$. Thus for any $w\in \dom_t$, there exists some lattice point $u\in \mathcal L\cap z_t^{-1}(\dom_t)$, such that
$
|z_t(u)-w|\leq (1+\sfT n^{4/3})/n^8\leq n^{-5},
$
provided $n$ is large enough. 
\end{proof}

\subsection{Proof of \Cref{p:edgerigidity}}\label{s:prove31}
To prove \Cref{p:edgerigidity}, we denote the Stieltjes transform of the empirical particle density by
\begin{align}\label{e:defmt}
\widetilde m_t(z)=\frac{1}{n}\sum_{i=1}^n \frac{1}{\lambda_i(t)-z},\quad z\in \bH.
\end{align}
By It\^{o}'s formula, $\widetilde m_s(z)$ satisfies a Burgers type stochastic differential equation, 
\begin{align}\begin{split}\label{eq:dm}
\rd\widetilde  m_s(z)
&=-\frac{1}{n}\sum_{i=1}^n \frac{\rd \la_i(s)}{(\la_i(s)-z)^2}
+\frac{1}{n}\sum_{i=1}^n \frac{\langle \rd \la_i(s), \rd \la_i(s)\rangle}{(\la_i(s)-z)^3}\\
&=-\sqrt{\frac{2}{\beta n^3}}\sum_{i=1}^n \frac{\rd B_i}{(\la_i(s)-z)^2}-\frac{1}{n^2}\sum_{i=1}^n \frac{1}{(\la_i(s)-z)^2}\sum_{j:j\neq i}\frac{1}{\lambda_i(s)-\lambda_j(s)}
+\frac{2}{\beta n^2}\sum_{i=1}^n \frac{\rd s}{(\la_i(s)-z)^3}\\
&=-\sqrt{\frac{2}{\beta n^3}}\sum_{i=1}^n \frac{{\rm d} B_i(s)}{(\la_i(s)-z)^2}+ \widetilde m_s(z)\del_z\widetilde  m_s(z)\rd s +\frac{2-\beta}{\beta n^2}\sum_{i=1}^{n}\frac{\rd s}{(\lambda_i(s)-z)^3},
\end{split}\end{align}
where the last line follows from the fact that
\begin{align*}
\frac{1}{n^2}& \sum_{i=1}^n \frac{1}{(\la_i(s)-z)^2}\sum_{j:j\neq i}\frac{1}{\lambda_i(s)-\lambda_j(s)} \\
& =\frac{1}{2n^2}\sum_{i\neq j}^n \bigg( \frac{1}{(\la_i(s)-z)^2}\frac{1}{\lambda_i(s)-\lambda_j(s)}+\frac{1}{(\la_j(s)-z)^2}\frac{1}{\lambda_j(s)-\lambda_i(s)} \bigg) \\
&= -\frac{1}{2n^2}\sum_{i\neq j}\frac{\la_i(s)-z+ \la_j(s)-z}{(\la_i(s)-z)^2(\lambda_j(s)-z)^2} \\
& =-\frac{1}{n^2}\sum_{i\neq j}\frac{1}{(\la_i(s)-z)(\lambda_j(s)-z)^2}
=\widetilde m_s(z)\del_z \widetilde m_s(z)+\frac{1}{n^2}\sum_{i=1}^n\frac{1}{(\la_i(s)-z)^3}.
\end{align*}
 Plugging the characteristic flow \eqref{e:flow} into \eqref{eq:dm}, and using the chain rule,  we have 
\begin{align}\label{e:tdmzt}\begin{split}
\rd \widetilde m_s(z_s) = \partial_s \widetilde{m}_s (z_s) + \partial_z \widetilde{m}_s (z_s) \partial_s z_s =-&\sqrt{\frac{2}{\beta n^3}}\sum_{i=1}^n \frac{{\rm d} B_i(s)}{(\la_i(s)-z_s)^2}\\
&+\del_z\widetilde m_s(z_s)\left(\widetilde m_s(z_s)-m_s(z_s)\right)\rd s+\frac{2-\beta}{\beta n^2}\sum_{i=1}^{n}\frac{\rd s}{(\la_i(s)-z_s)^3}.
\end{split}
\end{align}

Let $\mu_t = \mu_0 \boxplus \mu_{\semci}^{(t)}$ denote the free convolution between $\mu_0$ with the rescaled semicircle distribution, and let $m_t$ denote its Stieltjes transform. Then, we can integrate both sides of \eqref{e:tdmzt} from $0$ to $t$ to obtain
\begin{align}\label{eq:mzt}
\widetilde m_t(z_t)- m_t(z_t)= \widetilde{m}_t (z_t) - m_0 (z_0) = \widetilde{m}_t (z_0) - \widetilde{m}_0 (z_0) =  \displaystyle\int_0^t \mathrm{d} \widetilde{m}_s (z_s) ds = \int_0^t \left( \cE_1(s)\rd s+\rd \cE_2(s) \right) ,
\end{align}
where the error terms are
\begin{align}
\label{defcE1}\cE_1(s)=&\left(\widetilde m_s(z_s)-m_s(z_s)\right)\del_z \widetilde m_s(z_s)
\\
\label{defcE2}\rd \cE_2(s)=&\frac{2-\beta}{\beta n^2}\sum_{i=1}^n\frac{\rd s}{(\lambda_i(s)-z_s)^3}-\sqrt{\frac{2}{\beta n^3}}\sum_{i=1}^n \frac{{\rd} B_i(s)}{(\la_i(s)-z_s)^2}.
\end{align}
We remark that $\mathcal E_1$ and $\mathcal E_2$ implicitly depend on $u$, the initial value of the flow $z_0=u$.

We define the stopping time 
\begin{align}\begin{split}\label{stoptime}
\sigma\deq
\sfT \wedge
\inf_{s\geq0}\left\{\exists w\in \dom_s: \left|\widetilde m_s(w)-m_s(w)\right|\geq \frac{  (\log n)^{7/2} }{n\sqrt{\Im[w](\Re[w-E_s]+\Im[w])}}\right\}.
\end{split}\end{align}

For $u\in \cD_0$ and $s\leq \ft(u)\wedge \sigma$, by the definition of the stopping time \eqref{stoptime}, $\Im[\widetilde m_s(w)]$ and $\Im[m_s(w)]$ are close. A more precise estimate is given in the following lemma. 
\begin{lem}\label{c:Im-Im}
Adopt the notation and assumptions of \Cref{p:edgerigidity}, and assume $n$ is large enough. Given any $u \in  \cD_0$, recall $\ft(u)$ from \eqref{e:deftu}. For any $0\leq s\leq \ft(u)$, abbreviate $z_s=z_s(u)$; we have
\begin{align}\begin{split}\label{e:Im-Im}
    &\left|\Im[\widetilde m_{s\wedge \sigma}(z_{s\wedge \sigma})]-\Im[m_{s\wedge \sigma}(z_{s\wedge \sigma})]\right|\leq \frac{\Im[m_{s\wedge \sigma}(z_{s\wedge \sigma})]}{\log n},\\
     &\Im[\widetilde m_{s\wedge \sigma}(z_{s\wedge \sigma})]\leq 2\Im[m_{s\wedge \sigma}(z_{s\wedge \sigma})].
\end{split}\end{align}
\end{lem}
\begin{proof}[Proof of \Cref{c:Im-Im}]
From the definition of $\ft(u)$ as in \eqref{e:deftu}, $z_s=z_s(u):=E_s+\kappa_s+\ri\eta_s\in \cD_s$ for $0\leq s\leq t(u)$. From the definition \eqref{stoptime} of $\sigma$, we have that 
\begin{align*}
|\widetilde m_{s\wedge \sigma}(z_{s\wedge \sigma}) - {m}_{s\wedge \sigma}(z_{s\wedge \sigma})|
\leq 
\frac{(\log n)^{7/2}}{n\sqrt{(\kappa_{s\wedge \sigma}+\eta_{s\wedge \sigma})\eta_{s\wedge \sigma}}}
\leq \frac{\text{Im}[{m}_{s\wedge \sigma}(z_{s\wedge \sigma})]}{\log n},
\end{align*}
where we used \eqref{e:Immz2} in the last inequality. The last statement in \eqref{e:Im-Im} also follows for $n$ large enough. 
\end{proof}

The following lemma states that before the stopping time $\sigma$, the largest particle is bounded by $E_t+f(t)$. 
\begin{lem}\label{l:beforestop}
Adopt the notation and assumptions in \Cref{p:edgerigidity}, and assume $n$ is large enough. For any $t\leq \sigma$ (from \eqref{stoptime}), we have 
\begin{align}
\lambda_1(t)\leq E_t+f(t).
\end{align} 
\end{lem}

\begin{proof}[Proof of \Cref{l:beforestop}]
Fix a time $t\geq 0$ and recall $\cD_t$ from \eqref{e:defDt}. We notice that
\begin{align}\label{e:monotone}
\eta \Im[m_t(E_t+f(t)+\eta)]=\int_\bR \frac{\eta^2 \varrho_t(x)\rd x}{(x-E_t-f(t))^2+\eta^2},
\end{align}

\noindent is increasing in $\eta \ge 0$. When $\eta\rightarrow +\infty$, \eqref{e:monotone} goes to one; when $\eta\rightarrow 0$, \eqref{e:monotone} goes to zero. Therefore, there exists a (unique) $\eta_t>0$ satisfying
\begin{align}\label{e:defw}
 \frac{1}{(\log n)n}=\eta_t\Im[m_t(E_t+f(t)+\ri\eta_t)].
 \end{align}
 Then let $w_t=E_t+f(t)+\ri\eta_t$. Recalling the definition of $\cD_t$ from \eqref{e:defDt}, \eqref{e:defw} and the monotonicity \eqref{e:monotone} imply that $w_t$ is on the south-west corner of $\cD_t$.

Next we claim that the bound
\begin{align}\label{e:twmestimate}
|\widetilde m_t(w_t)- m_t(w_t)|\leq \frac{(\log n)^{7/2}}{n\sqrt{\Im[w_t](\Re[w_t-E_t]+\Im[w_t])}},
\end{align}
implies that no $\lambda_i(t)$ exists in the interval $[E_t+f(t)-\eta_t, E_t+f(t)+\eta_t]$. Indeed, the estimate \eqref{e:twmestimate} implies 
\begin{align}
     |\widetilde m_t(w_t)- m_t(w_t)|\leq \frac{(\log n)^{7/2}}{n\sqrt{\Im[w_t](\Re[w_t-E_t]+\Im[w_t])}}\leq \Im[m_t(w_t)]= \frac{1}{(\log n)n\Im[w_t]},
\end{align}
where the second inequality is from $w_t\in \cD_t$ and \eqref{e:Immz2}; the last equality is from our construction \eqref{e:defw}.
Together with our choice of $\eta_t$ from \eqref{e:defw}, it follows that
\begin{align}\label{e:sumerror}
    \Im[\widetilde m_t(w_t)]
    \leq \Im[ m_t(w_t)]+ |\widetilde m_t(w_t)- m_t(w_t)|
    \leq \frac{2}{(\log n)n\Im[w_t]}.
\end{align}
If there exists some $\la_i(t)$ such that $\la_i(t)\in [E_t+f(t)-\eta_t, E_t+f(t)+\eta_t]$, then the left-hand side of \eqref{e:sumerror} is at least 
\begin{align*}
\Im[\widetilde m_t(w_t)]=\Im\left[\sum_{i=1}^n \frac{1}{\lambda_i(t)-w_t}\right]
=\frac{1}{n}\sum_{i=1}^n \frac{\Im[w_t]}{|\lambda_i(t)-w_t|^2}\geq \frac{\Im[w_t]}{n|w_t-\lambda_i(t)|^2}\geq \frac{1}{2n\Im[w_t]},
\end{align*}
where in the first inequality we used that $\Im[w_t]=\eta_t>0$, so each term in the summation is positive; in the last inequality we used that $|w_t-\lambda_i(t)|^2 \leq 2\eta_t^2$ since $\la_i(t)\in [E_t+f(t)-\eta_t, E_t+f(t)+\eta_t]$. This leads to a contradiction with \eqref{e:sumerror}, provided $n$ is large enough.

For $t\leq \sigma$, the bound \eqref{e:twmestimate} holds (by the definition \ref{stoptime}), and so there is no $\lambda_j(t)$ in the interval $[E_t+f(t)-\eta_t, E_t+f(t)+\eta_t]$.  At time $t=0$, we have $\la_1(0)\leq 0\leq f(0)$. Since $\la_1(t)$ is continuous, and also $f(t)$ and $E_t$ are continuous (the first follows from the definition \eqref{e:deff}, and the second follows from the facts $E_t=\xi(t)-tm_0(\xi(t))$ and $\xi(t)$ is continuous from \eqref{e:defxit}) we conclude that $\la_1(t)\leq E_t+f(t)$ for all $0\leq t\leq \sigma $.
\end{proof}

\Cref{p:edgerigidity} follows from the following statement and \Cref{l:beforestop}. Its proof will occupy the remainder of this section. 
\begin{prop}\label{p:stoptime}
Adopt the notation and assumptions in \Cref{p:edgerigidity}. There exists a constant $C=C(\sfb, \sfT)>1$, with probability $1-Ce^{-(\log n)^2}$, we have $\sigma=\sfT $.
\end{prop}

\begin{proof}[Proof of \Cref{p:edgerigidity}]

On the event $\sigma=\sfT $, \Cref{l:beforestop} implies that $\la_1(t)\leq E_t+f(t)$ for $0\leq t\leq \sfT $. Thus  \Cref{p:edgerigidity}  follows from \Cref{p:stoptime}.
\end{proof}

In the rest of this section we prove \Cref{p:stoptime} by analyzing the stochastic differential equation \eqref{eq:mzt}. The following proposition will be used to upper bound the error terms $\cE_1(s)$ and $\cE_2(s)$. 

\begin{prop}\label{p:thirdbound}

There exists a constant $C=C(\sfb, \sfT)>1$ and an event $\Omega$, measurable with respect to the Brownian paths $\{B_1(s), B_2(s), \cdots, B_n(s)\}_{0\leq s\leq \sfT }$, such that $\mathbb{P} [\Omega] \ge 1-Ce^{-(\log n)^2}$ and the following holds. On $\Omega$, for any $u\in \cL$ (recall this lattice from \eqref{def:L}), denote $z_s(u)=E_s(u)+\kappa_s(u)+\ri \eta_s(u)$. Then, for any $0\leq s\leq \ft(u)$ (recall \eqref{e:deftu}), we have
\begin{equation}\label{StochasticBnd}
\int_{0}^{s \wedge \sigma}\sqrt{\frac{2}{\beta n^3}} \sum_{i=1}^{n} \frac{\rd B_i(\tau)}{|z_\tau(u) - \lambda_i(\tau)|^2} \rd \tau \leq  \frac{C(\log n)^2}{n \sqrt{(\kappa_{s \wedge \sigma}(u)+\eta_{s \wedge \sigma}(u)) \eta_{s \wedge \sigma}(u)}}.
\end{equation}
and
 \begin{align}\label{e:ssbound1}
\frac{|2 - \beta|}{\beta} \int_{0}^{s \wedge \sigma}  \frac{1}{n^2} \sum_{i=1}^{n}\frac{1}{|\lambda_{i}(\tau) - z_{\tau}(u)|^3}\rd \tau\leq  \frac{ C\log n}{n (\kappa_{s \wedge \sigma}(u)+\eta_{s \wedge \sigma}(u))},
 \end{align}
\end{prop}
\begin{proof}[Proof of \Cref{p:thirdbound}]
For simplicity of notation, we write $\mathfrak{t} (u), z_{\tau}(u), \kappa_{\tau}(u), \eta_{\tau}(u)$ as $\mathfrak{t}, z_{\tau}, \kappa_{\tau}, \eta_{\tau}$, respectively.
By the definition of $\sigma$, for $\tau\leq \sigma$, \Cref{l:beforestop} gives that 
\begin{align}\label{e:la1bb}
\la_1(\tau)\leq E_\tau+f(\tau). 
\end{align}

To prove \eqref{StochasticBnd}, we notice by \Cref{c:Dtproperty} that $z_{\tau}(u)\in \cD_{\tau}$ for any $0\leq {\tau}\leq \ft$. We define a series of stopping times $0=t^{(0)}< t^{(1)}<t^{(2)}<\cdots<t^{(m)}=\ft$, 
as follows:
\begin{equation}\label{e:choosetik}
t^{(k)} = \ft \wedge \inf\{{\tau}> t^{(k-1)} : (\kappa_{{\tau}} +\eta_{{\tau}})\eta_{{\tau}} <( \kappa_{t^{(k-1)}}+ \eta_{t^{(k-1)}})\eta_{t^{(k-1)}}/2\},\quad k=1,2,3,\cdots,m,
\end{equation}

\noindent where $m = m(u)$ might depend on $u$. From \eqref{e:Immz1}, $\kappa_t$ and $\eta_t$ are finite and cannot be smaller than $1/n^2$; in particular, $\eta_t (\kappa_t + \eta_t) \ge 2n^{-4}$, and so any $u \in \mathcal{L}$ must satisfy $m = m(u) \leq10 \log n$.
We now apply the Burkholder--Davis--Gundy inequality to our stochastic integral. This allows us to bound the supremum of our stochastic integral over a time range by its quadratic variation with overwhelming probability. We refer to \cite[Theorem 3.2]{DFIM} (see also \cite[Theorem 11.2.1]{PTIIM}) for a detailed introduction to the Burkholder--Davis--Gundy inequality. The quadratic variation of the left side of \eqref{StochasticBnd} at $s = t^{(k)}$ is given by 
\begin{align}\begin{split}\label{e:quadraticv}
&\phantom{{}={}}\int_{0}^{t^{(k)} \wedge \sigma} \frac{2}{\beta n^3} \sum_{k=1}^{n} \frac{\rd {\tau}}{|z_{\tau} - \lambda_k({\tau})|^4}
 \leq 
  \frac{\OO(1)}{n^3} \int_{0}^{t^{(k)} \wedge \sigma} \sum_{i=1}^n\frac{1}{|z_\tau-E_\tau-f(\tau)|^2}\frac{\rd \tau}{|z_\tau-\lambda_i(\tau)|^2} \\
 &=\frac{\OO(1)}{n^2} \int_{0}^{t^{(k)} \wedge \sigma} \frac{\text{Im}[\widetilde m_{\tau}(z_{\tau})]}{\eta_{\tau} ((\kappa_{\tau} - f({\tau}))^2 + (\eta_{\tau})^2)} \rd {\tau} \leq 
\frac{\OO(1)}{n^2} \int_{0}^{t^{(k)} \wedge \sigma} \frac{\text{Im}[m_{\tau}(z_{\tau})]}{\eta_{\tau} ((\kappa_{\tau} - f({\tau}))^2 + (\eta_{\tau})^2)} \rd {\tau},
\end{split}\end{align}
where in the first inequality we used that for $\tau\leq \sigma$ we have $\lambda_1(\tau)\leq E_\tau+f(\tau) \le \Real z_{\tau}$; in the second inequality we used the definition of $\wt m_t$ from \eqref{e:defmt}; and in the last inequality, we used  \eqref{e:Im-Im} indicating that, for ${\tau}\leq \sigma$, we have $\Im[\widetilde m_{\tau}(z_{\tau})]\leq 2\Im[ m_{\tau}(z_{\tau})]$.

Let $s=t^{(k)}\wedge\sigma$ and $\widetilde s=\min\{s, \mathfrak{C}\sqrt{\eta_*}\}$. We can rewrite the integral in \eqref{e:quadraticv} as
\begin{align}\begin{split}\label{e:intcov}
\int_{0}^{s} \frac{\text{Im}[m_{\tau}(z_{\tau})]}{\eta_{\tau} ((\kappa_{\tau} - f({\tau}))^2 + (\eta_{\tau})^2)} \rd {\tau} & \leq \int_{0}^{\widetilde s} \frac{\text{Im}[m_{\tau}(z_{\tau})]}{\eta_{\tau} (((\widetilde s-{\tau})\Im[m_s(z_s)]\kappa_{\widetilde s}/\eta_{\widetilde s})^2 + (\eta_{\tau})^2)} \rd {\tau}\\
& \qquad +\int_{\widetilde s}^{s} \frac{\text{Im}[m_{\tau}(z_{\tau})]}{\eta_{\tau} (((s-{\tau})\Im[m_s(z_s)]\kappa_{ s}/2\eta_{ s})^2 + (\eta_{\tau})^2)} \rd {\tau},
\end{split}\end{align}
where we used  \eqref{e:ksfsdiff2} to get the first term on the righthand of \eqref{e:intcov}, and \eqref{e:ksfsdiff1} to get the second term. The two terms on the righthand side of \eqref{e:intcov} can be bounded in the same way, so we will only study the second one.
 Let $A:=\Im[m_{\tau}(z_{\tau})]=\Im[m_s(z_s)]$; then $\eta_{\tau}=\eta_s+(s - \tau)A$ (from \eqref{e:etas}). Applying a change of variables replacing $s-\tau$ by $\tau$, and then one replacing $A \tau$ by $\tau$, we find 
 \begin{align*}
\int_{\widetilde s}^{s} \frac{\text{Im}[m_{\tau}(z_{\tau})]\rd {\tau}}{\eta_{\tau} (((s-{\tau})\Im[m_s(z_s)]\kappa_{ s}/2\eta_{ s})^2 + (\eta_{\tau})^2)} \rd {\tau}
& =
\int_{0}^\infty \frac{A\rd {\tau}}{(\eta_s+{\tau}A) (({\tau}A\kappa_{ s}/2\eta_{ s})^2 + (\eta_s+{\tau}A)^2)} \\
&=\int_{0}^\infty \frac{\rd {\tau} }{(\eta_s+{\tau}) (({\tau}\kappa_{ s}/2\eta_{ s})^2 + (\eta_s+{\tau})^2)}.
 \end{align*}

\noindent If $\eta_s\geq \kappa_s$, then 
\begin{align}
\int_{0}^\infty \frac{\rd {\tau} }{(\eta_s+{\tau}) (({\tau}\kappa_{ s}/2\eta_{ s})^2 + (\eta_s+{\tau})^2)}
\leq \int_{0}^\infty \frac{\rd {\tau} }{(\eta_s+{\tau})^3 }\leq \frac{1}{2\eta_s^2}\leq \frac{1}{(\kappa_s+\eta_s)\eta_s}.
\end{align}
If $\kappa_s\geq \eta_s$, we need to decompose the integral into three parts:
\begin{align}\begin{split}\label{e:ffv}
&\phantom{{}={}}\int_{0}^\infty \frac{\rd {\tau} }{(\eta_s+{\tau}) (({\tau}\kappa_{ s}/2\eta_{ s})^2 + (\eta_s+{\tau})^2)}
=\int_0^{\eta^2_s/\kappa_s}(\cdots)
+\int_{\eta^2_s/\kappa_s}^{\eta_s}(\cdots)
+\int_{\eta_s}^{\infty}(\cdots)\\
&\leq 
\int_0^{\eta^2_s/\kappa_s}\frac{\rd {\tau} }{\eta_s^3}
+\int_{\eta^2_s/\kappa_s}^{\eta_s}\frac{\rd {\tau} }{\eta_s ({\tau}\kappa_{ s}/2\eta_{ s})^2 }
+\int_{\eta_s}^{\infty}\frac{\rd {\tau} }{{\tau}({\tau}\kappa_{ s}/2\eta_{ s})^2}
\leq \frac{1}{\kappa_s\eta_s}+\frac{4}{\kappa_s\eta_s}+\frac{2}{\kappa_s^2}\leq \frac{16}{(\kappa_s+\eta_s)\eta_s}.
\end{split}\end{align}
By the same argument, the first term on the righthand side of \eqref{e:intcov} can be bounded by $16/((\kappa_{\widetilde s}+\eta_{\widetilde s})\eta_{\widetilde s})\leq 16/((\kappa_{s}+\eta_{s})\eta_s)$ (where the inequality is from the facts that $\widetilde s\leq s$ and that $\kappa_\tau$ and $\eta_\tau$ are both decreasing function in the time $\tau$ from \eqref{e:etas} and \eqref{e:ksktrelation}). 

	 Thus, for  $s=t^{(k)}\wedge \sigma$ and any $p > 1$, the Burkholder--Davis--Gundy inequality \cite[Theorem 3.2]{DFIM} (see also \cite[Theorem 11.2.1]{PTIIM}), together with the upper bound of the quadratic variance \eqref{e:quadraticv}, \eqref{e:intcov} and \eqref{e:ffv}, yields 
	that there exists a constant $C_0>1$ such that
	\begin{flalign}
		\label{integral} 
		\bE
		\left[ \displaystyle\sup_{\tau \in [0, t^{(k)}]} 
		\left| \int_{0}^{{\tau}\wedge \sigma}\sqrt{\frac{2}{\beta n^3}} \sum_{i=1}^{n} \frac{\rd B_i({\tau})}{|z_{\tau} - \lambda_i({\tau})|^2} \right|^{2p} \right] 
		\le (72p)^{2p} \cdot \mathbb{E} \Bigg[ \left(\frac{C_0}{n^2(\kappa_{t^{(k)} \wedge \sigma}+\eta_{t^{(k)} \wedge \sigma}) \eta_{t^{(k)} \wedge \sigma}}\right)^{p} \Bigg].
	\end{flalign}

By taking $p=(\log n)^2$ in \eqref{integral}, there exists some $C_1>1$, such that with  probability $1-e^{-2(\log n)^2}$ we have 
\begin{equation}\label{localineq}
\sup_{0\leq {\tau} \leq t^{(k)} } \left|\int_{0}^{{\tau}\wedge \sigma}\sqrt{\frac{2}{\beta n^3}} \sum_{i=1}^{n} \frac{\rd B_i({\tau})}{|z_{\tau} - \lambda_i({\tau})|^2}  \right|\leq \frac{C_1(\log n)^2}{n \sqrt{(\kappa_{t^{(k)} \wedge \sigma}+\eta_{t^{(k)} \wedge \sigma}) \eta_{t^{(k)} \wedge \sigma}}}.
\end{equation}
We define $\Omega$ to be the event on which \eqref{localineq} holds for all $0\leq k\leq m$ and all $u\in\cL$. Since $m|\cL|\leq |\mathcal{L}| \cdot 10 \log n \le n^{20}$, it follows from the discussion above that  $\Omega$ holds with probability $1-n^{20}e^{-2(\log n)^2}\geq 1-e^{-(\log n)^2}$, provided $n$ is large enough. Therefore, for any $s\in[t^{(k-1)},t^{(k)}]$, the bounds \eqref{localineq} and our choice of $t^{(k)}$ \eqref{e:choosetik} (with the fact that $\eta_s$ and $\kappa_s$ are non-increasing in $s$) yield on $\Omega$ that, for any $0\leq s\leq \ft(u)$, we have
\begin{align*}\begin{split}
\left|\int_{0}^{s\wedge \sigma}\sqrt{\frac{2}{\beta n^3}} \sum_{i=1}^{n} \frac{\rd B_i(\tau)}{|z_\tau - \lambda_i(\tau)|^2} \right|
 \leq \frac{C_1(\log n)^2}{n \sqrt{(\kappa_{t^{(k)} \wedge \sigma}+\eta_{t^{(k)} \wedge \sigma}) \eta_{t^{(k)} \wedge \sigma}}}\leq \frac{2C_1(\log n)^2}{n \sqrt{(\kappa_{s \wedge \sigma}+\eta_{s \wedge \sigma}) \eta_{s \wedge \sigma}}}.
\end{split}
\end{align*}
This finishes the proof of  \eqref{StochasticBnd}.

The left side of \eqref{e:ssbound1} is, following \eqref{e:quadraticv}, bounded by 
 \begin{align}\begin{split}\label{e:cubintegral}
 &\phantom{{}={}}\int_{0}^{s \wedge \sigma} \frac{1}{n^2} \sum_{i=1}^{n}\frac{1}{|\lambda_{i}(\tau) - z_{\tau}|^3} \rd \tau
 \leq  \int_{0}^{s \wedge \sigma} \frac{1}{n^2} \sum_{i=1}^{n}\frac{1}{|\lambda_{i}(\tau) - z_{\tau}|^2|\kappa_\tau-f(\tau)+\ri \eta_\tau|} \rd \tau\\
 &=\frac{1}{n}\int_{0}^{s \wedge \sigma} \frac{ \text{Im}[\widetilde m_{\tau}(z_{\tau})]}{|\kappa_\tau -f({\tau})+\ri\eta_{\tau}| \eta_\tau} \rd {\tau} \\
 &\leq  \frac{\OO(1)}{n}\int_{0}^{s \wedge \sigma} \frac{\text{Im} [{m}_{\tau}(z_{\tau})]}{|\kappa_\tau -f({\tau})+\ri\eta_{\tau}|\eta_\tau} \rd {\tau}
 \leq  \frac{\OO(1)}{n} \int_{0}^{s \wedge \sigma} \frac{\Im[m_\tau(z_\tau)]}{(\kappa_{\tau} -f({\tau})+\eta_{\tau}) \eta_{{\tau}}}\rd {\tau}.
\end{split}\end{align}
where we used $\lambda_k(\tau)\leq \la_1(\tau)\leq E_\tau+f(\tau)$ from \eqref{e:la1bb}, thus $|\lambda_{k}(\tau) - z_{\tau}|\geq |\kappa_\tau+E_\tau-f(\tau)+\ri \eta_\tau|$ in the first inequality; the definition of $\wt m_t(z)$ for the second inequality; \eqref{e:Im-Im} in the third inequality; and $|\kappa_\tau-f(\tau)+\ri\eta_\tau|\geq |\kappa_\tau-f(\tau)+\eta_\tau|/2$ in the last inequality. Let $\widetilde s=\min\{s\wedge \sigma, \fC\sqrt{\eta_*}\}$; we can rewrite the integral in \eqref{e:cubintegral} as
\begin{align}\label{e:intcov0}
\int_{0}^{s\wedge \sigma} \frac{\Im[m_\tau(z_\tau)]}{(\kappa_{\tau} -f({\tau})+\eta_{\tau}) \eta_{{\tau}}}\rd {\tau}
=\int_{0}^{\widetilde s} \frac{\Im[m_\tau(z_\tau)]}{(\kappa_{\tau} -f({\tau})+\eta_{\tau}) \eta_{{\tau}}}\rd {\tau}
+\int_{\widetilde s}^{s\wedge \sigma} \frac{\Im[m_\tau(z_\tau)]}{(\kappa_{\tau} -f({\tau})+\eta_{\tau}) \eta_{{\tau}}}\rd {\tau}.
\end{align}
Using \Cref{c:domainic}, the two terms on the righthand side of \eqref{e:intcov0} can be bounded in the same way as for \eqref{e:intcov}; we will only study the second one.
Again let $A:=\Im[m_{\tau}(z_{\tau})]=\Im[m_{s\wedge\sigma}(z_{s\wedge\sigma})]$, so that $\eta_{\tau}=\eta_{s\wedge\sigma}+(s\wedge\sigma -\tau)A$. We can bound the second term on the righthand side of \eqref{e:intcov0} as
 \begin{align*}
&\phantom{{}={}}\int_{\widetilde s}^{{s\wedge\sigma}} \frac{\text{Im}[m_{\tau}(z_{\tau})]\rd {\tau}}{\eta_{\tau} ((s\wedge \sigma-{\tau})\Im[m_{s\wedge\sigma}(z_{s\wedge\sigma})]\kappa_{ {s\wedge\sigma}}/2\eta_{ {s\wedge\sigma}} + \eta_{\tau})} \rd {\tau}\\
&=
\int_{0}^\infty \frac{A\rd {\tau}}{(\eta_{s\wedge\sigma}+{\tau}A) (({\tau}A\kappa_{ {s\wedge\sigma}}/2\eta_{ {s\wedge\sigma}}) + (\eta_{s\wedge\sigma}+{\tau}A))} \\
&=\int_{0}^\infty \frac{\rd {\tau} }{(\eta_{s\wedge\sigma}+{\tau}) (({\tau}\kappa_{ {s\wedge\sigma}}/2\eta_{ {s\wedge\sigma}}) + (\eta_{s\wedge\sigma}+{\tau}))}.
 \end{align*}
where we used \eqref{e:ksfsdiff1} to bound the second term on the righthand side of \eqref{e:intcov0} by the first line; in the second line we used the definition of $A$ and applied a change of variables replacing $s \wedge \sigma -\tau$ by $\tau$; in the last line we replaced $A \tau$ with $\tau$.

 If $\eta_{s\wedge \sigma}\geq \kappa_{s\wedge \sigma}$, then 
\begin{align}
\int_{0}^\infty \frac{\rd {\tau} }{(\eta_{s\wedge\sigma}+{\tau}) (({\tau}\kappa_{ {s\wedge\sigma}}/2\eta_{ {s\wedge\sigma}}) + (\eta_{s\wedge\sigma}+{\tau}))}
\leq \int_{0}^\infty \frac{\rd {\tau} }{(\eta_{s\wedge \sigma}+{\tau})^2 }\leq \frac{1}{\eta_{s\wedge \sigma}}\leq \frac{2}{\kappa_{s\wedge \sigma}+\eta_{s\wedge \sigma}}.
\end{align}
If $\kappa_{s\wedge \sigma}\geq \eta_{s\wedge \sigma}$, we need to decompose the integral into two parts:
\begin{align*}
&\phantom{{}={}}\int_{0}^\infty \frac{\rd {\tau} }{(\eta_{s\wedge\sigma}+{\tau}) (({\tau}\kappa_{ {s\wedge\sigma}}/2\eta_{ {s\wedge\sigma}}) + (\eta_{s\wedge\sigma}+{\tau}))}
=
\int_{0}^{\eta_{s\wedge \sigma}}(\cdots)
+\int_{\eta_{s\wedge \sigma}}^{\infty}(\cdots)\\
&\leq 
\int_0^{\eta_{s\wedge\sigma}}\frac{\rd {\tau} }{\eta_{s\wedge \sigma}(\tau \kappa_{s\wedge\sigma}/2\eta_{s\wedge\sigma}+\eta_{s\wedge\sigma})}
+\int_{\eta_{s\wedge \sigma}}^{\infty}\frac{\rd {\tau} }{{\tau}({\tau}\kappa_{ s\wedge\sigma}/2\eta_{ s\wedge\sigma})}
\\
&\leq \frac{2 \log((\kappa_{s\wedge\sigma}+\eta_{s\wedge\sigma})/\eta_{s\wedge\sigma})}{\kappa_{s\wedge \sigma}}+\frac{2}{\kappa_{s\wedge \sigma}}\leq \frac{\OO(1)(\log n)}{\kappa_{s\wedge \sigma}+\eta_{s\wedge \sigma}},
\end{align*}
 where in the last inequality we used $(\kappa_{s\wedge\sigma}+\eta_{s\wedge\sigma})/\eta_{s\wedge\sigma}\leq n^3$ (from \eqref{e:Immz1}). The first term on the righthand side \eqref{e:intcov0} can similarly be bounded by $\OO(1)\log n/(\kappa_{\widetilde s}+\eta_{\widetilde s})\leq \OO(1)\log n/(\kappa_{s\wedge\sigma}+\eta_{s\wedge\sigma})$ (where the inequality is from the facts that $\widetilde s\leq s\wedge\sigma$, and that $\kappa_\tau$ and $\eta_\tau$ are both decreasing function in time $\tau$ from \eqref{e:etas} and \eqref{e:ksktrelation}). This finishes the proof of \eqref{e:ssbound1}.

\end{proof}

\begin{proof}[Proof of \Cref{p:stoptime}]
Let us return to equation \eqref{eq:mzt} for the difference between $\widetilde{m}_t(z)$ and $m_t(z)$. 
We plug \eqref{StochasticBnd} and \eqref{e:ssbound1} into \eqref{eq:mzt}, on the event $\Omega$ as defined in \Cref{p:thirdbound}, to obtain 
\begin{align}\begin{split}\label{e:globalest}
|\widetilde m_{t\wedge\sigma}(z_{t\wedge\sigma})- m_{t\wedge \sigma}(z_{t\wedge \sigma})|
&=\int_0^{t\wedge\sigma}\left|\widetilde{m}_{\tau}(z_{\tau})-m_{\tau}(z_{\tau})\right|\left|\partial_z  \widetilde m_{\tau}(z_{\tau})\right|\rd \tau\\
&+\OO\left(\frac{(\log n)^2}{n\sqrt{(\kappa_{t\wedge \sigma}+\eta_{t\wedge \sigma})\eta_{t\wedge \sigma}}}\right).
\end{split}\end{align}
For the first term on the righthand side of \eqref{e:globalest}, using  \Cref{l:STproperty} and \eqref{e:Im-Im}, we have
\begin{align}\label{e:aterm2}
\left|\partial_z \widetilde m_{\tau}(z_{\tau})\right|
\leq \frac{\text{Im}[ \widetilde{m}_{\tau}(z_{\tau})]}{\eta_{\tau}}
\leq \left(1+\frac{1}{\log n}\right)\frac{\text{Im}[ {m}_{\tau}(z_{\tau})]}{\eta_{\tau}}.
\end{align}

We denote the quantity,
\begin{align} \label{eqn:betabd}
\beta(s):= \left(1+\frac{1}{\log n}\right)\frac{\text{Im}[ {m}_s(z_s)]}{\eta_s}\leq \frac{2\text{Im}[ {m}_s(z_s)]}{\eta_s},
\end{align}
and rewrite \eqref{e:globalest} as
\begin{align*}\begin{split}
\left|\widetilde{m}_{ t\wedge\sigma}(z_{ t\wedge\sigma})-  m_{ t\wedge\sigma}(z_{ t\wedge\sigma})\right|
\leq& \int_0^{ t\wedge\sigma}\beta(s)\left|\widetilde{m}_s(z_s)- m_s(z_s)\right|\rd s+\frac{\OO(1)(\log n)^2}{n\sqrt{(\kappa_{ t \wedge \sigma}+\eta_{ t \wedge \sigma})\eta_{ t \wedge \sigma}}}.
\end{split}
\end{align*}
By the Gr\"{o}nwall inequality, this implies for any $ 0\leq t\leq \ft(u)$ that, on $\Omega$, 
\begin{align}\begin{split}\label{e:midgronwall}
\left|\widetilde{m}_{t\wedge\sigma}(z_{t\wedge\sigma})-m_{t\wedge\sigma}(z_{t\wedge\sigma})\right|
& \leq 
\frac{\OO(1)(\log n)^2}{n\sqrt{ (\kappa_{t \wedge \sigma}+\eta_{t\wedge \sigma})\eta_{t\wedge \sigma}}}\\
& \qquad + \int_0^{t\wedge\sigma}\beta(s)\left(\frac{\OO(1)(\log n)^2}{n\sqrt{(\kappa_s+\eta_{s})\eta_{s}}}+\right)e^{\int_s^{t\wedge\sigma} \beta(\tau)\rd \tau} \rd s.
\end{split}
\end{align}
For the function $\beta(\tau)$, using  \eqref{eqn:betabd}, we have the following estimates
\begin{align*}\begin{split}
\int_s^{t\wedge\sigma} \beta(\tau)\rd \tau
&\leq  \left(1+\frac{1}{\log n}\right)\int_s^{t\wedge\sigma} \frac{\Im[ m_\tau(z_\tau)]}{\eta_\tau}\rd \tau= \left(1+\frac{1}{\log n}\right)\log \left(\frac{\eta_s}{\eta_{t\wedge\sigma}}\right),
\end{split}\end{align*} 
where in the last equality we used $\del_\tau \eta_\tau=-\Im[m_\tau(z_\tau)]$, by taking the time derivative of \eqref{e:etas}. Thus,
\begin{align}\label{e:beta}
e^{\int_s^{t\wedge\sigma} \beta(\tau)d \tau}
\leq  e^{\left(1+\frac{1}{\log n}\right)\log \left(\frac{\eta_{s}}{\eta_{t\wedge\sigma}}\right)}
\leq \OO(1)  \frac{\eta_s}{\eta_{t\wedge\sigma}},
\end{align}
where in the last inequality, we used \eqref{e:Immz1}, which gives $\log(\eta_{s}/\eta_{t\wedge\sigma}) \leq 3\log n$, provided $n$ is large enough. 
Combining \eqref{e:beta} with \eqref{e:midgronwall} and \eqref{eqn:betabd}, we can bound the last term in \eqref{e:midgronwall} by
\begin{align}\begin{split}\label{e:term2}
&\phantom{{}={}} \mathcal{O} (1) \int_0^{t\wedge\sigma}\frac{\text{Im}[m_{s}(z_s)]}{\eta_{t\wedge \sigma}}\left(\frac{(\log n)^2}{n\sqrt{(\kappa_{s}+\eta_{s})\eta_{s}}}\right) \rd s.
\end{split}\end{align}
Thanks to \eqref{e:ksktrelation} in \Cref{c:kappa}, we have
\begin{align}\label{e:kslowb}
(\kappa_{s}+\eta_{s})\eta_{s}\geq \left(\frac{\kappa_{t\wedge \sigma}}{\eta_{t\wedge \sigma}}\eta_s+\eta_s\right)\eta_s=\eta_s^2\frac{\kappa_{t\wedge \sigma}+\eta_{t\wedge \sigma}}{\eta_{t\wedge \sigma}}.
\end{align}
By plugging \eqref{e:kslowb} into \eqref{e:term2}, we obtain 
\begin{align}\begin{split}\label{e:dier}
\int_0^{t\wedge\sigma}\frac{\text{Im}[m_{s}(z_s)]}{\eta_{t\wedge \sigma}}\left(\frac{(\log n)^2}{n\sqrt{(\kappa_{s}+\eta_{s})\eta_{s}}}\right) \rd s
\leq
\frac{(\log n)^2}{n\sqrt{\eta_{t\wedge\sigma}(\kappa_{t\wedge\sigma}+\eta_{t\wedge\sigma})}}\int_0^{t\wedge\sigma}\frac{\text{Im}[m_{s}(z_s)]\rd s}{\eta_s}\\
\leq 
\frac{(\log n)^2 \log(\eta_0/\eta_{t\wedge \sigma})}{n\sqrt{\eta_{t\wedge\sigma}(\kappa_{t\wedge\sigma}+\eta_{t\wedge\sigma})}}\leq \frac{3(\log n)^3 }{n\sqrt{\eta_{t\wedge\sigma}(\kappa_{t\wedge\sigma}+\eta_{t\wedge\sigma})}}.
\end{split}\end{align}

\noindent where in the first bound we used \eqref{e:kslowb}; in the second inequality we used $\del_\tau \eta_\tau=-\Im[m_\tau(z_\tau)]$ by taking the time derivative of \eqref{e:etas}; and in the last inequality, we used  \eqref{e:Immz1}, which implies $\log(\eta_{0}/\eta_{t\wedge\sigma}) \leq 3\log n$.

Combining the above estimates \eqref{e:midgronwall} and \eqref{e:dier}, we deduce on the event $\Omega$ that, for any $u\in \cL$ and $0\leq t\leq \mathfrak{t}(u)$, we have
\begin{align*}
 \left|\widetilde{m}_{t\wedge \sigma}(z_{t\wedge \sigma}) - m_{t\wedge \sigma}(z_{t\wedge \sigma}) \right|
 &\leq \frac{\OO(1)(\log n)^3}{n \sqrt{(\kappa_{t\wedge \sigma}+\eta_{t\wedge \sigma}) \eta_{t\wedge \sigma}}}.
\end{align*}
For $w\in \cD_{t\wedge\sigma}$, from \eqref{e:Immz1}, $\Im[w]\geq 1/n^2$. Thus on $\cD_{t\wedge\sigma}$, 
both $\widetilde{m}_{t\wedge \sigma}(w)$ and $m_{t\wedge \sigma}(w)$ are Lipschitz with Lipschitz constant bounded by $n^{4}$ (using \Cref{l:STproperty}).
From \Cref{c:Lapproximate}, for any $w\in \cD_{t\wedge\sigma}$ there exists some $u\in \cL$ with $|z_{t\wedge\sigma}(u)-w|\leq n^{-5}$. Thus we have 
$|\widetilde{m}_{t\wedge \sigma}(w) -\wt m_{t\wedge \sigma}(z_{t\wedge\sigma}(u))|\leq 1/n$ and  $|{m}_{t\wedge \sigma}(w) - m_{t\wedge \sigma}(z_{t\wedge\sigma}(u))|\leq 1/n$. Hence, for $n$ large enough,
\begin{align}\begin{split}\label{e:outside}
 \left|\widetilde{m}_{t\wedge \sigma}(w) - m_{t\wedge \sigma}(w) \right|& \leq \left|\widetilde{m}_{t\wedge \sigma}(z_{t\wedge \sigma}(u)) - m_{t\wedge \sigma}(z_{t\wedge \sigma}(u)) \right| +  \left|\widetilde{m}_{t\wedge \sigma}(w) - m_{t\wedge \sigma}(z_{t\wedge\sigma}(u)) \right| \\
 & \qquad  + \left|{m}_{t\wedge \sigma}(w) - m_{t\wedge \sigma}(z_{t\wedge\sigma}(u)) \right|\\
 &\leq \frac{\OO(1)(\log n)^3}{n \sqrt{(\kappa_{t\wedge \sigma}+\eta_{t\wedge \sigma}) \eta_{t\wedge \sigma}}}+\frac{2}{n}\leq \frac{1}{10}\frac{(\log n)^{7/2}}{n \sqrt{(\Re[w-E_{t\wedge\sigma}]+\Im[w]) \Im[w]}}.
\end{split}\end{align}
uniformly for any $w\in \cD_{t\wedge\sigma}$. Comparing \eqref{e:outside} with the definition of the stopping time $\sigma$ from \eqref{stoptime}, we conclude that on $\Omega$ we have $\sigma=\sfT $.  Since $\mathbb{P} [\Omega] \ge 1 - \mathcal{O} (e^{-(\log n)^2})$, this finishes the proof of \Cref{p:stoptime}.
\end{proof}

\subsection{Edge Universality}

Given what we have already done, our universality result \Cref{t:universality} will be a quick consequence of the following coupling from \cite[Theorem 3.1]{landon2017edge}. Before recalling that result, we need to introduce some notation.
Take $0<\omega_1<\omega_0/100$ and fix two time scales
\begin{align}
t_0=\frac{n^{\omega_0}}{n^{1/3}},\quad t_1=\frac{n^{\omega_1}}{n^{1/3}}.
\end{align}
We consider a pair of coupled solutions of the Dyson Brownian motion \eqref{e:DBM}  $\{\wt\lambda_i(s)\}_{1\leq i\leq n}$ and $\{\mu_i(s)\}_{1\leq i\leq n}$ using the same Brownian motions as in \Cref{l:coupling}.  

The Dyson Brownian motion  $\{\mu_i(s)\}_{1\leq i\leq n}$ starts from the eigenvalues of a Gaussian $\beta$ ensemble. It is known that the law of $\{\mu_i(s)\}_{1\leq i\leq n}$ is given by the $\beta$-ensemble  and the empirical eigenvalue $(1/n)\sum_{1\leq i\leq n}\delta_{\mu_i(s)}$ concentrates around the rescaled semicircle distribution (\cite[Theorem 2.4]{MR3253704} and also \cite{MR3192527,MR2905803, bourgade2022optimal}),
\begin{align}
\mu_{\rm sc}^{(1+s)}=\frac{\sqrt{4(1+s)-x^2}}{2(1+s)\pi}\rd x,
\end{align}
We denote its right edge by $E_\mu(s)=2\sqrt{1+s}$.

We also assume there exists a small constant $c>0$, so that the empirical measure of the initial data of $\{\wt\lambda_i(s)\}_{1\leq i\leq n}$ concentrates around a density $\wt \varrho_{0}(x)$ with $E_{\la}(0)=\max\supp  \wt \varrho_{0}(x)$, such that the following statements hold.
\begin{enumerate}
\item  $\wt \varrho_{0}(x)$ has square root behavior close to $E_{\la}(0)$:
\begin{align}\label{e:trho0bb}
c\sqrt{E_{\la}(0)-x} \leq\wt \varrho_{0}(x)\leq c^{-1}\sqrt{E_{\la}(0)-x}, \quad \text{for}\quad  E_\la(0)-c\leq x\leq E_\la(0).
\end{align}
and
\begin{align}\label{e:trho0}
\wt \varrho_{0}(x)=\frac{\sqrt{E_{\la}(0)-x}}{\pi}\left(1+\OO\left(\frac{|E_{\la}(0)-x|}{t_0^2}\right)\right),\quad \text{for}\quad E_\la(0)-ct_0^2\leq x\leq E_\la(0).
\end{align}
\item We denote $\widetilde \varrho_{s}$ the free convolution of $\widetilde \varrho_0$ with the rescaled semicircle distribution $\mu_{\semci}^{(s)}$, its right edge by $E_{\la}(s)=\max\supp \wt\varrho_s$, and its classical eigenvalue locations by $\wt \gamma_i(s)=\gamma_{i;n}^{\widetilde \varrho_{s}}$ as in \eqref{e:classical_loc}. Then with probability $1-c^{-1}e^{-(\log n)^2}$, we have for all $0\leq s\leq t_1$ that
\begin{align}\label{e:rigg}
|\wt\lambda_i(s)-\wt\gamma_i(s)|\leq \frac{(\log n)^{1/c}}{i^{1/3}n^{2/3}}, \quad \text{for}\quad  1\leq i\leq cn.
\end{align}
\end{enumerate}
The assumptions \eqref{e:trho0bb} and \eqref{e:trho0} verify the conditions \cite[Equation (3.19) and (3.21)]{landon2017edge}, and \cite[Theorem 3.1]{landon2017edge} states that with high probability, the difference $\wt\lambda_i(s)-\mu_i(s)$ is a constant up to a very small error of size $\OO(n^{-2/3-\delta})$.\footnote{We remark the result in \cite[Theorem 3.1]{landon2017edge} is stated for $\beta=1$, but the coupling directly generalizes for $\beta\geq 1$. A detailed discussion in the setting of Dyson Brownian motion with $\beta\geq 1$ (and general potential) can be found in  \cite[Section 6]{adhikari2020dyson}.}

\begin{prop}[{\cite[Theorem 3.1]{landon2017edge}}]\label{p:diffsmall}

There exists a small constant $\delta>0$ such that the following holds for any real number $D > 1$. We can couple the two Dyson Brownian motions $\{\wt\la_i(s)\}_{1\leq i\leq n}$ with $\{\mu_i(s)\}_{1\leq i\leq n}$, such that for $n$ large enough, with probability $1-n^{-D}$, we have 
\begin{align}
|(\wt\lambda_i(t_1)-E_\la(t_1))-(\mu_i(t_1)-E_\mu(t_1))|\leq n^{-2/3-\delta},
\end{align}
for any finite $i\geq 1$ (that is uniformly bounded in $n$). 
\end{prop}

Next we prove \Cref{t:universality} using \Cref{p:diffsmall}. We recall the notation from the beginning of \Cref{s:Cf}.
We also recall that $t\geq \ft=\max\{2\fB\sqrt{\eta_*}, n^{-1/3+\fa}\}$ (where $\ft$ and $\fa$ are from \Cref{t:universality}). Fix a small real number $0<\omega<\fa/100$ and let $t_0=t-n^{-1/3+\omega}$. Then 
\begin{align}\label{e:t0low}
t_0\geq t/2\geq n^{-1/3+\fa}/2+\fB\sqrt{\eta_*}.
\end{align}
Then  \eqref{e:rhotdensity} in \Cref{l:densityt} gives for $x\in [E_{t_0}, E_{t_0}-\fc(t_0)]$ that
\begin{align}\label{e:rhot0a}
 \varrho_{t_0}(x)= (1+\cE(E_{t_0}-x))\frac{\sqrt{\cA({t_0})(E_{t_0}-x)}}{\pi},\qquad \text{where $\mathcal{E}$ satisfies $|\cE(x)|\leq  \frac{|x|}{\fc({t_0})}$},
\end{align}
and $\fc({t_0})$ is from \eqref{e:defAs0}, and \eqref{e:defAs} gives $\fc({t_0})\geq 2^{-20}(\sfb {t_0})^2$.
In the following we denote $\cA:=\cA(t_0)$.  
Our assumption in  \Cref{t:universality} gives that $\cA\in [\Theta^{-1}, \Theta]$, and for $t_0\leq t\leq \sfT$ that
\begin{align}\label{e:Et0}
 \Theta^{-1}\sqrt{E_{t}-x}\leq \varrho_{t}(x)\leq \Theta\sqrt{E_{t}-x},\quad \text{for}\quad \quad E_{t}-\Theta^{-1}\leq x\leq E_{t}.
\end{align}

The next lemma states the difference $|\cA-\cA(t)|$ is small (recall $\cA=\cA(t_0)$ and $t_0=t-n^{-1/3+\omega}$)
\begin{lem}\label{l:cAdiff}
Adopt the assumptions in \Cref{t:universality}, we have
\begin{align}
|\cA-\cA(t)|\leq n^{-\omega},
\end{align}
provided $n$ is large enough.
\end{lem}

\begin{proof}
Recall $\xi(s)$ from \eqref{e:defxit}, which satisfies the relation
\begin{align}\label{e:xitexp}
 m_0' (\xi(s))=\int_\bR \frac{\rd \mu_0(x)}{(x-\xi(s))^2}=\frac{1}{s}. 
\end{align}
Recall from the first statement in \eqref{e:defAs} that $\xi(s)\geq (\sfb s/20)^2$, and
\begin{align}\label{e:third}
|m_0^{(3)} (\xi(s))|=6\int_\bR \frac{\rd \mu_0(x)}{(x-\xi(s))^4} 
\leq \frac{6}{\xi(s)^2}\int_\bR \frac{\rd \mu_0(x)}{(x-\xi(s))^2} =\frac{6}{s\xi(s)^2}\leq \frac{2400}{\sfb^4 s^5},
\end{align}

\noindent where in the second statement we used the fact that $\xi(s) > \max (\supp \mu_0)$ (by \eqref{e:defxit}). By taking the derivative with respect to $s$ on both sides of \eqref{e:xitexp}, we obtain
\begin{align}\label{e:dxit}
\xi'(s)=\frac{1}{-s^2 m_0^{(2)} (\xi(s))}. 
\end{align}
The time derivative of $g(s):=-s^3 m_0^{(2)} (\xi(s))/2=1/\cA(s)$ is bounded by 
\begin{align}\begin{split}\label{e:dgt}
|\del_s g(s)|&=\frac{1}{2}|\del_s (s^3 m_0^{(2)} (\xi(s)))|\\
&=\frac{1}{2}|3s^2 m_0^{(2)} (\xi(s))+s^3 m_0^{(3)} (\xi(s))\xi'(s)|
\leq \frac{3|g(s)|}{s}+\frac{1200}{\sfb^4s |g(s)|},
\end{split}\end{align}
where in the last inequality we used \eqref{e:third} and \eqref{e:dxit}

By our assumption in \Cref{t:universality}, we have $g(s)\in [\Theta^{-1}, \Theta]$ for $s\in [t_0, t]\subset [\ft/2, \sfT]$. Then by \eqref{e:dgt}, for $s\in[t_0, t]$ we have
\begin{align}\label{e:dsg}
|\del_s g_s|\leq \frac{3|g(s)|}{2s}+\frac{1200}{\sfb^4 s |g(s)|}\leq \frac{3\Theta}{t_0}+\frac{4800 \Theta}{\sfb^4 t_0}, \quad s\in [t_0, t].
\end{align}
We conclude that 
\begin{align*}
|\cA(t)-\cA(t_0)|
&=g_{t}^{-1} g_{t_0}^{-1}
|g_{t}-g_{t_0}|\leq \Theta^2 |t-t_0|\left(\frac{3\Theta}{t_0}+\frac{4800  \Theta}{\sfb^4 t_0}\right)\\
&\leq \frac{10^4 \Theta^3}{\sfb^4}\frac{|t-t_0|}{t_0}\leq \frac{10^4 \Theta^3}{\sfb^4}\frac{n^{-1/3+\omega}}{n^{-1/3+\fa}/2}
\leq \frac{2\cdot 10^4 \Theta^3}{\sfb^4} n^{\omega-\fa}\leq n^{-\omega},
\end{align*}
provided $n$ is large enough. Here, in the first statement we used $\cA(s)=1/g(s)$; in the second statement we used $|g(s)|^{-1} \leq \Theta$ and integrated \eqref{e:dsg} from $t_0$ to $t$; the third statement is from basic algebra; the fourth statement uses that $t-t_0=n^{-1/3+\omega}$ and $t_0\geq t/2\geq n^{-1/3+\fa}/2$ from \eqref{e:t0low}; and the last two inequalities follow from $\omega\leq \fa/100$ and the fact that $n$ is large enough.
This finishes the lemma.
\end{proof}

\begin{proof}[Proof of \Cref{t:universality}]

For the Dyson Brownian motion $\{\lambda_i(s)\}_{1\leq i\leq n}$, we can rescale time by $\cA^{2/3}$ and space by $\cA^{1/3}$. Specifically, define: $\widetilde \lambda_i(s)=\cA^{1/3}\lambda_i(t_0+\cA^{-2/3}s)$. Then $\{\wt \lambda_i(s)\}_{1\leq i\leq n}$ also satisfies Dyson's Brownian motion:
\begin{align}\label{e:ttla}
	\rd\wt\lambda_i (s) = \frac{1}{n}\displaystyle\sum_{\substack{1 \le j \le n \\ j \ne i}}	\displaystyle\frac{\rd s}{\wt\lambda_i (s) - \wt\lambda_j (s)} + \Big( \displaystyle\frac{2}{\beta n} \Big)^{1/2} \rd B_i (s).
\end{align}
Next we check that rescaled Dyson Brownian motion $\{\wt \lambda_i(s)\}_{1\leq i\leq n}$ satisfies \eqref{e:trho0bb}, \eqref{e:trho0} and \eqref{e:rigg}.

We introduce the rescaled density $\wt \varrho_{0}(x)$, by setting (using \eqref{e:rhot0a})
\begin{align}\label{e:trho01}
\wt \varrho_{0}(x)= \frac{\varrho_{t_0}(\cA^{-1/3}x)}{\cA^{1/3}}=\frac{\sqrt{(E_\la(0)-x)}}{\pi}\left(1+\OO\left(\frac{|E_{\la}(0)-x|}{t_0^2}\right)\right),\quad E_\la(0):=\cA^{1/3}E_{t_0},
\end{align}
for $x\in [E_\la(0)-\cA^{1/3}\fc(t_0), E_\la(0)]$, where $\fc({t_0})\geq 2^{-20}(\sfb {t_0})^2$.
We denote by $\widetilde \varrho_{s}$ the free convolution of $\widetilde \varrho_0$ with the rescaled semicircle distribution $\mu_{\semci}^{(s)}$. Then, by \Cref{r:convpmeasure},
\begin{align}
\widetilde \varrho_{s}(x)=\frac{1}{\cA^{1/3}}\varrho_{t_0+\cA^{-2/3}s}(\cA^{-1/3}x),
\end{align}
and the right edge of $\widetilde \varrho_{s}(x)$ is given by $\wt E_s= \cA^{1/3}E_{t_0+\cA^{-2/3}s}$. 
Moreover, \eqref{e:Et0} implies that for $0\leq s\leq 1$ (from $t_0+\cA^{-2/3}\leq t_0+\Theta\leq \sfT$) we have 
\begin{align}\label{e:Et02}
 \Theta^{-1}\sqrt{x/\cA}\leq \wt\varrho_{s}(E_{\la}(s)-x)=\frac{\varrho_{t_0+\cA^{-2/3}s}(\cA^{-1/3}(E_{\la}(s)-x))}{\cA^{1/3}}\leq \Theta\sqrt{x/\cA}, 
 \end{align}
 for $E_{\la}(s)-\cA^{1/3}\Theta^{-1}\leq x\leq E_{\la}(s)$.
 The relations \eqref{e:trho01} and \eqref{e:Et02} verify the conditions \eqref{e:trho0bb} and \eqref{e:trho0} provided $c\leq \min\{2^{-20}\sfb^2\cA^{1/3},  \cA^{1/3}\Theta^{-1}\}$ is small enough.

We denote the classical locations (recall from \eqref{e:classical_loc}) of $\wt\varrho_s(x)$ by $\wt \gamma_i(s)=\gamma_{i;n}^{\wt\varrho_s}$.
The optimal edge rigidity estimate (from \Cref{t:main}) and bulk one (from \Cref{t:bulkrigidity}) give that there exists a constant $C=C(\sfb, \sfT)>1$ such that, with probability $1-Ce^{-(\log n)^2}$,
\begin{align}\label{e:rigiditycopy}
\wt \la_1(s)\leq E_\la(s)+\frac{(\log n)^{15}}{n^{2/3}},\quad 
\wt\gamma_{i+\lfloor(\log n)^5\rfloor}(s)-n^{-1}\leq \wt\lambda_i(s)\leq \wt\gamma_{i-\lfloor(\log n)^5\rfloor}(s)+n^{-1}.
\end{align}

\noindent for all $0\leq s\leq 1$ (as $\fB\sqrt{\eta_*}\leq t_0\leq t_0+\cA^{-2/3}\leq t_0+\Theta\leq \sfT$).

Since $\wt \varrho_s$ has square root behavior as in \eqref{e:Et02}, there exist constants $C_0=C_0(\Theta) > 1$ and $c > 0$ such that the classical locations satisfy
\begin{align}\label{e:inbound}
C_0^{-1}\frac{i^{2/3}}{n^{2/3}}\leq E_\la(s)-\wt\gamma_i(s)\leq C_0 \frac{i^{2/3}}{n^{2/3}},\quad 0\leq s\leq 1,\quad  1\leq i\leq cn.
\end{align}

\noindent Thus, we have for any $1\leq i\leq j\leq cn$ that
\begin{align*}
\frac{j-i}{n}
&=\int_{\wt \gamma_j(s)}^{\wt\gamma_i(s)}\wt\varrho_s(x)\rd x
\geq \frac{1}{\cA^{1/2}\Theta}\int_{\wt \gamma_j(s)}^{\wt\gamma_i(s)}\sqrt{E_\la(s)-x}\rd x
=\frac{2(E_\la(s)-\wt \gamma_j(s))^{3/2}-(E_\la(s)-\wt \gamma_i(s))^{3/2}}{3\cA^{1/2}\Theta}\\
&\geq \frac{2(\wt\gamma_i(s)-\wt \gamma_j(s))}{3\cA^{1/2}\Theta}\frac{(E_\la(s)-\wt \gamma_j(s))^{1/2}+(E_\la(s)-\wt \gamma_i(s))^{1/2}}{2}
\geq \frac{(\wt\gamma_i(s)-\wt \gamma_j(s))}{3C_0\cA^{1/2}\Theta }\frac{j^{1/3}}{n^{1/3}}
\end{align*}
where the first statement is from the definition of classical locations \eqref{e:classical_loc}; the second statement is from \eqref{e:Et02}; the third statement is from performing integration; the fourth statement is from the bound $a^{3/2}-b^{3/2}\geq (a-b)(\sqrt a+\sqrt b)/2$; in the last statement is from \eqref{e:inbound}. Thus we get 
\begin{align}\begin{split}\label{e:tgammaest}
\wt\gamma_i(t)-\wt\gamma_j(t)
&\leq \frac{3C_0 \cA^{1/2}\Theta (j-i)}{j^{1/3}n^{2/3}}, \quad 0\leq s\leq 1, \quad 1\leq i\leq j\leq cn.
\end{split}\end{align}

\noindent Then, the second statement in \eqref{e:rigiditycopy} and \eqref{e:tgammaest} together imply that with probability $1-Ce^{-(\log n)^2}$ we have for each $2(\log n)^5 \leq i \leq cn$ that 
\begin{align}\begin{split}\label{e:rigidityh1}
&\phantom{{}={}}|\wt\lambda_i(t)-\wt\gamma_{i}(t)|\leq |\wt\gamma_{i-\lfloor(\log n)^5\rfloor}(t)- \wt\gamma_{i+\lfloor(\log n)^5\rfloor}(t)|+2n^{-1}\\
&\leq \frac{3C_0 \cA^{1/2}\Theta \cdot 2(\log n)^5}{(i+(\log n)^5)^{1/3}n^{2/3}}+2n^{-1}\leq \frac{(\log n)^6}{i^{1/3}n^{2/3}},
\end{split}\end{align}
provided $n$ is large enough.

To address the remaining $1 \le i \le 2(\log n)^5$, observe from \eqref{e:rigiditycopy} that, with probability $1-Ce^{-(\log n)^2}$, we have
\begin{align}
\wt\gamma_{i+\lfloor(\log n)^5\rfloor}(s)-n^{-1}\leq \wt\lambda_i(s) \le \widetilde{\lambda}_1 (s) \le E_\la(s) + \frac{(\log n)^{15}}{n^{2/3}}, 
\end{align}
so
\begin{align}\label{e:llagammaha}
\wt\lambda_i(s), \wt\gamma_i(s)\in \left[\gamma_{3\lfloor(\log n)^5\rfloor}(s) - n^{-1}, E_\la(s)+\frac{(\log n)^{15}}{n^{2/3}}\right].
\end{align}
It follows that for each $1 \le i \le 2(\log n)^5$ we have
\begin{align}\begin{split}\label{e:rigidityh2}
&\phantom{{}={}}|\wt\lambda_i(s)-\wt\gamma_{i}(s)|\leq \frac{(\log n)^{15}}{n^{1/3}}+|E_\la(s)- \gamma_{3\lfloor(\log n)^5\rfloor}(s)| + n^{-1} \\
&\leq \frac{(\log n)^{15}}{n^{1/3}}+\frac{3C_0 \cA^{1/2}\Theta \cdot 3(\log n)^5}{((\log n)^5)^{2/3}n^{1/3}}+n^{-1}\leq \frac{(\log n)^{16}}{i^{1/3}n^{2/3}},
\end{split}\end{align}
provided $n$ is large enough.
Here the first inequality follows from \eqref{e:llagammaha}, and the second inequality follows from \eqref{e:tgammaest}.
The two estimates \eqref{e:rigidityh1} and \eqref{e:rigidityh2} together give \eqref{e:rigg}.
Thus \Cref{p:diffsmall} (with $(t_0, t_1)$ there taken to be $(t_0, \cA^{2/3}n^{-1/3+\omega})$ here) implies that we can couple the two Dyson Brownian motions $\{\wt\la_i(s)\}_{1\leq i\leq n}$ (from \eqref{e:ttla}) with $\{\mu_i(s)\}_{1\leq i\leq n}$, such that with probability $1-n^{-D}$ we have for any finite $i\geq 1$ that
\begin{align}\label{e:wtcoupling}
|(\wt\lambda_i(t_1)-E_{\la}(t_1))-(\mu_i(t_1)- 2 \sqrt{1+t_1})|\leq n^{-2/3+\delta}.
\end{align}
where $t_1=\cA^{2/3}n^{-1/3+\omega}=\cA^{2/3}(t-t_0)$, provided $n$ is large enough.
Recall the relation $\widetilde \lambda_i(s)=\cA^{1/3}\lambda_i(t_0+\cA^{-2/3}s)$,   and $\cA^{1/3}E_{\la}(s)=E_{t_0+\cA^{-2/3}s}$; \eqref{e:wtcoupling} then implies
\begin{align}\begin{split}\label{e:ccouple}
&\phantom{{}={}}|(\cA^{1/3}(\lambda_1(t)-E_{t}))-(\mu_1(t_1)-2 \sqrt{1+t_1})| \leq n^{-2/3+\delta}.
\end{split}\end{align}

It was proved in \cite{MR2813333} that as the number of particles $n$ goes to infinity, the rescaled vector,
\begin{align}\label{e:Airyb}
  n^{2/3}(2\sqrt{1+t}-\mu_1(t))\rightarrow TW_\beta, 
\end{align}
converges to the Tracy-Widom $\beta$ distribution, which is characterized by the stochastic Airy operator. A quantitative version of the convergence \eqref{e:Airyb} follows from \cite[Theorem 1.2]{landon2020edge}:  Letting $F : \bR \to \bR$ be a bounded test function with bounded derivatives, we have for $n$ is large enough that
\begin{align}\begin{split} \label{e:Airyb2}
\bE[ F (n^{2/3} (2\sqrt{1+t}-\mu_1(t))] = \bE_{TW_\beta}[ F (\Lambda) ]+\OO\left(n^{-\delta}\right),
\end{split}\end{align}
where $\Lambda$ is sampled under the Tracy-Widom $\beta$ distribution. \Cref{t:universality} then follows from applying the test function $F$ to \eqref{e:ccouple}, replacing $\cA$ by $\cA(t)$ (from \Cref{l:cAdiff}, their difference is bounded by $n^{-\omega}$),  taking expectations, and plugging in \eqref{e:Airyb2}.
\end{proof}

\section{Examples}\label{s:example}
In this section, we discuss some examples and applications of our optimal rigidity result \Cref{t:main}.

\subsection{Initial Data of Small Support}
In this section, we study the example that the initial data $\mu_0=\mu_0^{(n)}$ is supported on a small interval.
\begin{assumption}\label{a:deltadensity}

Let $c\in (0,1)$ be a real number, and assume that the initial data $\mu_0=\mu_0^{(n)}$ satisfies $\supp \mu_0 \subseteq [-c/2, c/2]$. 
\end{assumption}

We notice that $\mu_0([c/2-x,c/2])=1$ for any $x\geq c$. Then (after shifting by $c/2$) \Cref{a:densitylowbound} is satisfied with $\eta_*=c$, and $\sfb=1/\sfT^3$. For the remainder of this section, we denote the free convolution of $\mu_0$ with the rescaled semicircle distribution $\mu_{\semci}^{(t)} $ by $\mu_t=\mu_0\boxplus \mu_{\semci}^{(t)}$. We also let $E_t^- = \min (\supp \mu_t)$ and $E_t^+ = \max (\supp \mu_t)$, for each $t \ge 0$. 

We will show the following estimate on the $i$-th particle $\lambda_i (t)$ of $\bm{\lambda}(t)$,  using \Cref{t:main}.

\begin{prop}\label{c:oprigidity}
Adopt \Cref{a:deltadensity}, fix a real number $\sfT\geq 1$, and set $\fB=2^{60}\sfT^{18}$. Then, there exists a constant $C=C(c,\sfT) > 1$ such that, with probability $1-Ce^{-(\log n)^2}$, we have for  any time $\fB\sqrt c\leq t\leq \sfT$ and index $1 \le i \le n$ that
\begin{align}\label{e:rr}
|\lambda_i(t)-\gamma_i(t)|\leq \frac{(\log n)^{17}}{\min\{i,n-i+1\}^{1/3}n^{2/3}}.
\end{align}

\end{prop}

	As a consequence of \Cref{c:oprigidity}, we next state a result bounding the gaps between the first particles under the rescaled Dyson Brownian motion,
	\begin{flalign}
		\label{e:reDBM}
		\rd x_i (t) = \displaystyle\sum_{\substack{1 \le j \le n \\ j \ne i}}	\displaystyle\frac{\rd t}{x_i (t) - x_j (t)} + \Big( \displaystyle\frac{2}{\beta } \Big)^{1/2} \rd B_i (t),\quad 1\leq i\leq n,
	\end{flalign}
	 whose initial data is ``sufficiently small.'' This result will be used in the forthcoming work \cite{U}. 
	 We remark that similar result has been proven in \cite[Lemma 4.7]{lee2015edge}. They do not require that the support of $\mu_0$ is on a short interval. However, they require that the initial data $\mu_0$ is close to a deterministic profile, with certain distance bounded by $n^{-\oo(1)}$.

	\begin{cor}		
		\label{initialsmall2}
		For any real number $B > 1$, there exists a constant $c =B^{-34}2^{-120}> 0$ such that the following holds. Let $\bm{x} = (x_1, x_2, \ldots , x_n) \in \overline{\mathbb{W}}_n$ be a sequence of real numbers such that $x_1 -x_n < cn^{2/3}$. Let $\bm{x}(t) = \big( x_1 (t),x_2 (t), \ldots , x_n (t) \big) \in \overline{\mathbb{W}}_n$ denote recaled Dyson Brownian motion (from \eqref{e:reDBM}) with initial data $\bmx$, run for time $t$. 		
Then, there exists a constant $C = C(B) > 1$ such that the following holds with probability at least $1 - C e^{-(\log n)^2}$. For all $t \in [B^{-1}, B]$ and $1 \le i \le j \le \lfloor n/2 \rfloor$, we have 
\begin{flalign*}
|x_i (tn^{1/3}) - x_j (tn^{1/3})| \le (24\pi)^{2/3} t^{1/2} (j^{2/3} - i^{2/3}) + (\log n)^{20} i^{-1/3}.
\end{flalign*}
		
	\end{cor}

An ingredient for the proof of \Cref{c:oprigidity} is the following lemma, which states that for time $t$ much bigger than $c$, then  $\mu_t$ is close to a semicircle distribution.

\begin{lem}\label{l:density1}

Adopt \Cref{a:deltadensity}. Fix $M = 100$ and a real number $t > 0$ satisfying $\sqrt{t} \ge 25cM^2$. Then the following four statements hold.
\begin{enumerate}
\item \label{i:edget}The spectral edge of $\mu_t$ satisfies $|E_t^\pm \mp 2\sqrt t|\leq 2c$.
\item \label{i:bulkest}For $-2\sqrt{t}\sqrt{1-(2M)^{-2}}\leq y\leq 2\sqrt{t}\sqrt{1-(2M)^{-2}}$, we have
\begin{align}\label{e:middleregion}
\frac{2}{3}\frac{\sqrt{4t-y^2}}{2\pi t}\leq \varrho_t(y)\leq \frac{3}{2}\frac{\sqrt{4t-y^2}}{2\pi t}.
\end{align}
\item \label{i:edgeest}For $2\sqrt{t}\sqrt{1-(2M)^{-2}}\leq y\leq E_t^+$, we have
\begin{align}\label{e:edgeregion}
\frac{2}{3\pi}\sqrt{\frac{E^+_t-y}{t^{3/2}}}\leq \varrho_t(y)\leq \frac{3}{2\pi}\sqrt{\frac{E^+_t-y}{t^{3/2}}},
\end{align}
and the similar statement holds for $E_t^-\leq y \leq -\sqrt{2t}\sqrt{1-(2M)^{-2}}$,
\begin{align}\label{e:edge-region}
\frac{2}{3\pi}\sqrt{\frac{y-E^-_t}{t^{3/2}}}\leq \varrho_t(y)\leq \frac{3}{2\pi}\sqrt{\frac{y-E^-_t}{t^{3/2}}}.
\end{align}

\item \label{i:gammaest}
We denote the inverted cumulative density function of $\varrho_t$ by $\gamma_t$ (recall from \eqref{gammay}). Then, for any $0\leq y\leq y'\leq 4/7$, we have 
\begin{align}\label{e:gammaest}
\gamma_t(y)-\gamma_t(y')\leq (24 \pi)^{2/3}t^{1/2}(y'^{2/3}-y^{2/3}).
\end{align}
\end{enumerate}

\end{lem}

\begin{proof}[Proof of \Cref{i:edget} in \Cref{l:density1}]
We first prove $|E_t^\pm \mp 2\sqrt t|\leq 2c$. We will only prove the estimate for $E_t^+$, as the proof for $E_t^-$ is very similar. We recall $\xi(t)$ from \eqref{e:defxit} and that $\mu_0$ is supported on $[-c/2,c/2]$.   Hence,
\begin{align}
\frac{1}{(\xi(t)+c/2)^2}\leq m_0'(\xi(t))=\int_{-c/2}^{c/2}\frac{\rd \mu_0(x)}{(\xi(t)-x)^2}=\frac{1}{t}\leq \frac{1}{(\xi(t)-c/2)^2},
\end{align}
where the first and last inequality follows from \Cref{a:deltadensity} that $\mu_0$ is supported on $[-c/2, c/2]$; the second and third equation is the definition of $\xi(t)$ from \eqref{e:defxit}. It follows that 
\begin{align}\label{e:xitest}
 \sqrt t-c/2\leq \xi(t)\leq \sqrt t+c/2.
\end{align}
By \Cref{i:edge} in \Cref{yconvolution}, the spectral edge $E_t^+$ is given by
\begin{align}\begin{split}\label{e:Etfu}
E_t^+ =\xi(t)-tm_0(\xi(t))& =\xi(t)-t\int_\bR \frac{\rd \mu_0(x)}{x-\xi(t)}\\
&\leq \xi(t)+ t\int_\bR \frac{(\xi(t)+c/2)\rd \mu_0(x)}{(x-\xi(t))^2} = 2\xi(t)+c/2\leq 2\sqrt{t}+2c,
\end{split}\end{align}
where in the second line we first used that $\mu_0$ is supported on $[-c/2,c/2]$, and then that $\xi(t)\leq \sqrt t+c/2$ from \eqref{e:xitest}. Through very similar reasoning (which we omit), we also have the lower bound $E_t^+\geq 2\sqrt t-2c$, thereby establishing the bound $|E_t^+ - 2 \sqrt{t}| \le 2c$.
\end{proof}

Next we recall the relation between $\varrho_t(y)$ and the Stieltjes transform $m_0$ of $\mu_0$ from \Cref{yconvolution}. Recalling the region $\Lambda_t$ from \eqref{mtlambdat} and the function $v_t$ from \eqref{vte}, \Cref{i:boundary} in \Cref{yconvolution} indiciates that we can parametrize $\del \Lambda_t$ as $\{u+\ri v_t(u):u\in \bR\}$; we also have
\begin{align}\label{e:boundarye}
\frac{1}{t}\geq \int_\bR\frac{\rd \mu_0(x)}{|x-(u+\ri v_t(u))|^2}.
\end{align}

\noindent Thus, from \Cref{mz}, for any $y\in \bR$, there exists $w=w_t(u)=u+\ri v_t(u)\in \del \Lambda_t$ such that $y=w-tm_0(w)$.
 From \Cref{i:rhoy} in \Cref{yconvolution},  $y\in \bR$ satisfies $\varrho_t (y) > 0$ if and only if $w\in \del \Lambda_t\cap \bH$, i.e., $\Im[w]>0$.  Moreover, if this is the case, \Cref{i:boundary} and \Cref{i:rhoy} in \Cref{yconvolution} give 
\begin{align}\label{e:density}
\frac{1}{t}=\int_\bR\frac{\rd \mu_0(x)}{|x-(u+\ri v_t(u))|^2}, \quad \varrho_t(y)=\frac{\Im[w]}{\pi t}.
\end{align}

Since $\mu_0$ is supported on $[-c/2,c/2]$ by \Cref{a:deltadensity}, we get from \eqref{e:boundarye} that 
\begin{align}
\sqrt t\leq \max\{|u+\ri v_t(u)-c/2|, |u+\ri v_t(u)+c/2|\}
\end{align}
We conclude that $\sqrt t-c/2\leq |w|=|u+\ri v_t(u)|$, and hence for any $w\in \del \Lambda_t$ that
\begin{align*}
&\phantom{{}={}}\left|m_0(w)+\frac{1}{w}\right|=\left|\int_\bR\frac{\rd \mu_0(x)}{x-w}+\int_\bR\frac{\rd \mu_0(x)}{w}\right|\leq\int_{-c/2}^{c/2}\frac{|x|\rd \mu_0(x)}{|x-w||w|}\\
&\leq \frac{c/2}{(\sqrt{t}-c/2)(\sqrt t-c)}
= \frac{c}{2t}\frac{1}{(1-c/(2\sqrt t))(1-c/\sqrt t)}\leq \frac{c}{2t}\frac{1}{(1-1/6)(1-1/3)}
\leq \frac{c}{t}. 
\end{align*}
where the first statement is from the definition of $m_0(w)$; the second statement holds by taking absolute value of the integrand and using $\supp \mu_0\subset [-c/2, c/2]$ from \Cref{a:deltadensity}; the third statement follows from the bounds $|x|\leq c/2$ and $|w|\geq \sqrt t-c/2$; the fourth statement follows from dividing $t$ in both numerator and denominator; and in the fifth statement follows from the bound $\sqrt t\geq 3c$ imposed in the assumption of \Cref{l:density1}. In what follows, we define for any $z\in \del \Lambda_t$,
\begin{flalign} 
	\label{ewc} 
	\cE(z)=t(m_0(z)+1/z), \qquad \text{so that $|\cE(z)|\leq c$}. 
\end{flalign}

\noindent Then,  
\begin{align}\label{e:yw}
y=w-tm_0(w)=w+\frac{t}{w}-\cE(w), \qquad \text{so that} \qquad w^2-(y+\cE(w))w+t=0.
\end{align} 
By solving \eqref{e:yw}, we deduce 
\begin{align}\label{e:wyrelation}
w=\frac{(y+\cE(w))+ \sqrt{(y+\cE(w))^2-4t}}{2},
\end{align}
where the branch of the square root will be chosen such that $\Im[w]\geq 0$ (see \ref{e:diffsqrt} below).

\begin{proof}[Proof of \Cref{i:bulkest} in \Cref{l:density1}]
We must estimate $\varrho_t(y)$ for 
$-2\sqrt{t}\sqrt{1-(2M)^{-2}}\leq y\leq 2\sqrt{t}\sqrt{1-(2M)^{-2}}$.
Notice that $-2\sqrt{t}\sqrt{1-(2M)^{-2}}\leq y\leq 2\sqrt{t}\sqrt{1-(2M)^{-2}}$ is equivalent to $4t-y^2 \geq t/M^2$, which implies that 
\begin{align}\label{e:lowbb}
2\sqrt{t}-y\geq \frac{\sqrt t}{4M^2},\quad y+2\sqrt t\geq  \frac{\sqrt t}{4M^2}.
\end{align}
We have proven in \Cref{i:edget} that $|E_t^\pm \mp 2\sqrt t|\leq 2c$, which together with the assumption $\sqrt t\geq 25 cM^2$ from \Cref{l:density1} implies that $y\in [E_t^-, E_t^+]$. Recall $w\in \del \Lambda_t$ with $w-tm_0(w) = y \in \mathbb{R}$. From \eqref{e:wyrelation} and \eqref{ewc}, we have
\begin{align}\label{0}
0\leq \Im[w]=\frac{\Im[\cE(w)]+\Im[\sqrt{(y+\cE(w))^2-4t}]}{2}, \quad \text{so that} \quad \Im[\sqrt{(y+\cE(w))^2-4t}]\geq -c.
\end{align} 
Next we show that we should take the square root branch in \eqref{e:wyrelation} with positive imaginary part, and 
\begin{align}\label{e:diffsqrt}
|\sqrt{(y+\cE(w))^2-4t}-\sqrt{y^2-4t}|\leq \frac{4c\sqrt t}{\sqrt{|y^2-4t|}}.
\end{align} 
We first notice that 
\begin{align}\label{e:diffyy}
|((y+\cE(w))^2-4t)-(y^2-4t)|\leq |\cE(w)||2y+\cE(w)|\leq c(4\sqrt t-\sqrt t/(2M^2)+c)\leq 4c\sqrt t\leq t/{2M^2}.
\end{align}
where in the second inequality we used $|\cE(w)|\leq c$ from \eqref{ewc} and \eqref{e:lowbb}; in the last two inequalities we used that $c\leq \sqrt t/(8M^2)$ following from our assumption in \Cref{l:density1}. There are two choices for $\sqrt{(y+\cE(w))^2-4t}$; its imaginary part either can either be nonnegative or negative. In the latter case $\Im[\sqrt{(y+\cE(w))^2-4t}]<0$,  we have
\begin{align}\label{e:toosmall}
\Im[\sqrt{(y+\cE(w))^2-4t}]\leq -\sqrt{4t-y^2-t/(2M^2)}\leq -\sqrt{t/4M^2}=-\frac{\sqrt t}{2M}<-c,
\end{align}
where in the first inequality we used \eqref{e:diffyy}; in the second inequality we used $4t-y^2\geq t/M^2$ (see above \eqref{e:lowbb}); and in the last inequality we used $c\leq \sqrt t/(8M^2)$ following from our assumption in \Cref{l:density1}. Since \eqref{e:toosmall} contradicts with \eqref{0}, we conclude that $\Im[\sqrt{(y+\cE(w))^2-4t}]\geq 0$. Then \eqref{e:diffsqrt} follows from 
\begin{align*}
&\phantom{{}={}}|\sqrt{(y+\cE(w))^2-4t}-\sqrt{y^2-4t}|
=\left|\frac{((y+\cE(w))^2-4t)-(y^2-4t)}{\sqrt{(y+\cE(w))^2-4t}+\sqrt{y^2-4t}}\right|\\
&\leq \frac{4c\sqrt t}{|\Im[\sqrt{(y+\cE(w))^2-4t}+\sqrt{y^2-4t}]|}
\leq \frac{4c\sqrt t}{|\Im[\sqrt{y^2-4t}]|}= \frac{4c\sqrt t}{\sqrt{|y^2-4t|}}.
\end{align*}
where in the first inequality we used \eqref{e:diffyy}; in the second we used $\Im[\sqrt{(y+\cE(w))^2-4t}]\geq 0$ from the above discussion; and in the last inequality we used that $y^2-4t<0$ (see above \eqref{e:lowbb}). 

Now we can use \eqref{e:diffsqrt} to estimate $w$ in \eqref{e:wyrelation}. It follows from \eqref{e:wyrelation} that 
\begin{align}\begin{split}\label{e:diffbb}
&\phantom{{}={}}\left|w-\frac{y+\sqrt{y^2-4t}}{2}\right|
\leq \frac{c}{2}+\frac{1}{2}\left|\sqrt{(y+\cE(w))^2-4t}-\sqrt{y^2-4t}\right|
\leq \frac{c}{2}+\frac{2c\sqrt t}{\sqrt{|y^2-4t|}}\leq \frac{(4M+1)c}{2}.
\end{split}\end{align}
where we used that $|\cE(w)|\leq c$ in the first inequality; we used \eqref{e:diffsqrt} in the second inequality; and we used  $4t-y^2 \geq t/M^2$ (from above \eqref{e:lowbb}) for the last inequality.

We conclude from plugging \eqref{e:diffbb} into the second statement of \eqref{e:density} that 
\begin{align}
\left|\varrho_t(y)-\frac{\sqrt{4t-y^2}}{2\pi t}\right|\leq \frac{(4M+1)c}{2\pi t}\leq \frac{1}{3}\frac{\sqrt t/M}{2\pi t}\leq \frac{1}{3}\frac{\sqrt{4t-y^2}}{2\pi t},
\end{align}
where we used $\sqrt t\geq 25M^2 c\geq 3cM(4M+1)$ following from our assumption in \Cref{l:density1} and $4t-y^2 \geq t/M^2$ (from above \eqref{e:lowbb}). This finishes the proof of \eqref{e:middleregion}.
\end{proof}

\begin{proof}[Proof of \Cref{i:edgeest} in \Cref{l:density1}]
The proof of \eqref{e:edge-region} is exactly the same at that of \eqref{e:edgeregion}, so we only establish the latter. We have the estimate for $m_0^{(2)}(\xi(t))$,
\begin{align}\label{e:msecond}
\frac{2}{(\sqrt t+c)^3}\leq \frac{2}{(\xi(t)+c/2)^3}\leq -m_0^{(2)}(\xi(t))=\int_{-c/2}^{c/2}\frac{2\rd \mu_0(x)}{(\xi(t)-x)^3}\leq \frac{2}{(\xi(t)-c/2)^3}\leq \frac{2}{(\sqrt t-c)^3}.
\end{align}
where the first and last inequalities follow from the estimate on $\xi(t)$ from \eqref{e:xitest}; the second and fourth inequalities follow from the fact that $\supp\mu_0\subset[-c/2, c/2]$. 

Recall from \eqref{e:Etfu} that $E_t^+\leq 2\sqrt t+2c$, so $2\sqrt{t}\sqrt{1-(2M)^{-2}}\leq y < E_t^+$ implies that 
\begin{align}\label{e:yloc0}
|y^2-4t|\leq \max\{t/M^2, (2\sqrt t +2c)^2-4t\}=\max\{t/M^2, 8\sqrt t c+4c^2\}\leq t/M^2,
\end{align}
and 
\begin{align}\label{e:yloc}
-\sqrt t/(2M^2)\leq y-2\sqrt t\leq 2c\leq \sqrt t/(2M^2),
\end{align}
where in the last inequalities (of both) we used $c\leq \sqrt{t}/(25 M^2)$ from our assumption in \Cref{l:density1}.
Using the relation \eqref{e:wyrelation} \begin{align}\begin{split}\label{e:wtdiff}
|w-\sqrt t|
&\leq \frac{|y-2\sqrt t+\cE(w)|+\sqrt{|(y+\cE(w))^2-4t|}}{2}\\
&\leq \frac{(\sqrt t/2M^2+|\cE(w)|)+\sqrt{t/M^2+|\cE(w)(2y+\cE(w))|}}{2}\leq \frac{2\sqrt t}{M},
\end{split}\end{align}
where we used \eqref{e:yloc0} and \eqref{e:yloc} for the second inequality; and for the last inequality we used  $|\cE(w)|\leq c$ (from \eqref{ewc}), implying that
\begin{align*}
|\cE(w)(2y+\cE(w))|\leq c|4\sqrt t+4c+c|\leq \frac{8\sqrt t}{25M^2}\left(\sqrt t+\frac{\sqrt t}{10M^2}\right)\leq \frac{t}{M^2},
\end{align*}

\noindent which holds since $2y\leq 4\sqrt t+4c$ (by \eqref{e:yloc}), $M =  100$, and $c\leq \sqrt{t}/25M^2$ from our assumption in \Cref{l:density1}.

We can now Taylor expand $m_0(w)$ around $\xi(t)$:
\begin{align}\begin{split}\label{e:m0taylor2}
m_0(w)
&=m_0(\xi(t))+m_0'(\xi(t))(w-\xi(t))+m_0^{(2)}(\xi(t))\frac{(w-\xi(t))^2}{2}+ \widetilde \cE(w) \frac{(w-\xi(t))^3}{3!}\\
&=m_0(\xi(t))+\frac{(w-\xi(t))}{t}+m_0^{(2)}(\xi(t))\frac{(w-\xi(t))^2}{2}+\widetilde \cE(w) \frac{(w-\xi(t))^3}{3!},
\end{split}\end{align}

\noindent  where we used $m_0'(\xi(t))=1/t$ from the definition \eqref{e:defxit} of $\xi(t)$,  and for $w$ satisfying \eqref{e:wtdiff} the Taylor remainder satisfies 
\begin{align}\label{e:tcE}
|\widetilde \cE(w) |\leq \max_{|z-\sqrt t|\leq 2\sqrt t/M}|m_0^{(3)}(z)|
\leq \max_{|z-\sqrt t|\leq 2\sqrt t/M}\frac{6}{||z|-c/2|^4}
\leq \frac{6}{|\sqrt t -2\sqrt t/M-c/2|^4},
\end{align}
where we used \Cref{l:STproperty} with $(A,p)=(1,3)$ and $\dist(z, \supp \mu_0)\geq ||z|-c/2|$ for the second inequality; and $|z|\geq \sqrt t-2\sqrt t/M$ for the last inequality. 

We can now use \eqref{e:m0taylor2} to solve for the curve $\{u+\ri v_t(u):u\in \bR\}=\del \Lambda_t$ in a small neighborhood of $\xi(t)$. In particular, setting $y=E^+_t-x=w-tm_0(w)$, we obtain by rearranging \eqref{e:m0taylor2} that
\begin{align}\begin{split}\label{e:constructx2}
\bR\ni x
&=(\xi(t)-tm_0(\xi(t)))-(w-tm_0(w))
=tm_0^{(2)}(\xi(t))\frac{(w-\xi(t))^2}{2}+t\widetilde \cE(w) \frac{(w-\xi(t))^3}{3!}\\
&=\left(tm_0^{(2)}(\xi(t))+t\widetilde \cE(w)\frac{w-\xi(t)}{3}\right)\frac{(w-\xi(t))^2}{2}.
\end{split}\end{align}

\noindent Hence, 
\begin{align}\label{e:wxit}
w-\xi(t)=\ri\sqrt{\frac{2x}{-tm_0^{(2)}(\xi(t))}}\left(1+\frac{\widetilde \cE(w)(w-\xi(t))}{3m_0^{(2)}(\xi(t))}\right)^{-1/2}.
\end{align}

Using \eqref{e:msecond}, for the denominator in \eqref{e:wxit}, we have
\begin{align}\label{e:err2}
\frac{3}{4} \frac{2}{\sqrt t}\leq\frac{2t}{(\sqrt t+c)^3}\leq -t m_0^{(2)}(\xi(t))\leq \frac{2t}{(\sqrt t-c)^3}\leq \frac{2}{\sqrt t}\frac{4}{3},
\end{align}
where we used that $c\leq \sqrt t/ 25 M^2$ and $M=100$ from our assumption in \Cref{l:density1}.
Using \eqref{e:wtdiff} and \eqref{e:xitest}, we have the bound $|w-\xi(t)|\leq 2\sqrt t/M+c/2$. This, together with \eqref{e:tcE} and \eqref{e:err2}, bounds the error term in \eqref{e:wxit} by
\begin{align}\begin{split}\label{e:err1}
\left|\frac{\widetilde \cE(w)(w-\xi(t))}{3m_0^{(2)}(\xi(t))}\right|
&\leq \left(\frac{2\sqrt t}{M}+\frac{c}{2}\right)\frac{6}{|\sqrt t-2\sqrt t /M-c/2|^4}\frac{(\sqrt t+c)^3}{6}\\
&\leq 
\left(\frac{2}{M}+\frac{1}{25M^2}\right)\frac{(1+1/25 M^2)^3}{|1-2 /M-1/50 M^2|^4}
\leq \frac{1}{3}
\end{split}\end{align}
where we also used that $c\leq \sqrt t/25 M^2$ and $M=100$ from our assumption in \Cref{l:density1}. Then, plugging \eqref{e:err2} and \eqref{e:err1} into \eqref{e:wxit}, we obtain
\begin{align}
\frac{1}{2\pi}\sqrt{\frac{x}{t^{3/2}}}\leq \varrho_t(E_t-x)=\frac{\Im[w]}{t\pi}\leq \frac{2}{\pi}\sqrt{\frac{x}{t^{3/2}}}.
\end{align}
This finishes the proof of \eqref{e:edgeregion}.
\end{proof}

\begin{proof}[Proof of \Cref{i:gammaest} in \Cref{l:density1}]
To prove \eqref{e:gammaest}, we first show that 
for $-\sqrt t \leq x\leq \sqrt t$, we have 
\begin{align}\label{e:den2}
\frac{1}{4\pi}\sqrt{\frac{E^+_t-x}{t^{3/2}}}\leq \varrho_t(x)\leq \frac{4}{\pi}\sqrt{\frac{E^+_t-x}{t^{3/2}}},
\end{align}
and for $\sqrt t \leq x\leq E^+_t$ we have slightly better estimate
\begin{align}\label{e:den1}
\frac{1}{2\pi}\sqrt{\frac{E^+_t-x}{t^{3/2}}}\leq \varrho_t(x)\leq \frac{2}{\pi}\sqrt{\frac{E^+_t-x}{t^{3/2}}}.
\end{align}
The same statement holds close to $E_t^-$, i.e., for $E_t^- \leq x\leq -\sqrt t$
\begin{align}\label{e:den1-}
\frac{1}{2\pi}\sqrt{\frac{x-E^-_t}{t^{3/2}}}\leq \varrho_t(x)\leq \frac{2}{\pi}\sqrt{\frac{x-E^-_t}{t^{3/2}}}.
\end{align}

The proofs of \eqref{e:den2}, \eqref{e:den1} and \eqref{e:den1-} are very similar, so we will only prove \eqref{e:den1}. The statement \eqref{e:den1} for $2\sqrt t\sqrt{1-(2M)^{-2}}\leq x\leq E_t^+$ follows from \eqref{e:edgeregion}. In the case when $\sqrt t\leq x\leq 2\sqrt t\sqrt{1-(2M)^{-2}} \leq 2\sqrt t-\sqrt t/4M^2$, we have 
\begin{align}\label{e:low1}\begin{split}
 \sqrt{\frac{ t^{-3/2} (E^+_t-x)}{ (2t)^{-2} (4t-x^2)}}
 &\leq \sqrt{\frac{4 \sqrt t (2\sqrt t+2c-x)}{(2\sqrt t- x)(2\sqrt t+ x)}}\leq \sqrt{\frac{4 \sqrt t (\sqrt t/4M^2+2c)}{(\sqrt t/4M^2 )3\sqrt t}}\\
 &\leq \sqrt{\frac{4 \sqrt t (\sqrt t/4M^2+2\sqrt t/(25M^2))}{(\sqrt t/4M^2 )3\sqrt t}}
 =\sqrt{\frac{44}{25}}\leq \frac{4}{3},
\end{split}\end{align}
where in the first statement we used that $|E_t^+ - 2 \sqrt{t}| \le 2c$ from the first statement of the lemma; in the second statement we used that $\sqrt t\leq x\leq 2\sqrt t-\sqrt t/4M^2$ so $2\sqrt t+x\geq 3\sqrt t$, and $(2\sqrt t+2c-x)/(2\sqrt t-x)\leq (\sqrt t/4M^2+2c)/(\sqrt{t}/4M^2)$; and in the third statement we used $c\leq \sqrt t/(25M^2)$ from the assumption of \Cref{l:density1}. Moreover, 
\begin{align}\label{e:low2}\begin{split}
 \sqrt{\frac{t^{-3/2} (E^+_t-x)}{(2t)^{-2} (4t - x^2)}} 
 &= \sqrt{\displaystyle\frac{4t^{1/2} (E_t^+ - x)}{(2\sqrt{t} - x)(2\sqrt{t} + x)}}
 \geq \sqrt{\frac{ 2\sqrt t-2c-x}{2\sqrt t- x}}\\
 &\geq \sqrt{\frac{ \sqrt t/4M^2-2c}{\sqrt t/4M^2}}\geq \sqrt{\frac{ \sqrt t/4M^2-2\sqrt t/(25M^2)}{\sqrt t/4M^2}} =\sqrt{\frac{17}{25}}\geq \frac{3}{4},
\end{split}\end{align}
where the first statement follows from factoring $4t-x^2$; the second statement follows from $|E_t^+ - 2 \sqrt{t}| \le 2c$ from the first statement of the lemma and $x\leq 2\sqrt t-\sqrt t/4M^2\leq 2\sqrt{t} - 2c$; the third statement follows from $x\leq 2\sqrt t-\sqrt t/4M^2$ so $(2\sqrt t-2c-x)/(2\sqrt t-x)\geq (\sqrt t/4M^2-2c)/(\sqrt{t}/4M^2)$; and the fourth statement follows from $c\leq \sqrt t/(25M^2)$ from the assumption of \Cref{l:density1}. Then \eqref{e:den1} follows from \eqref{e:middleregion} by combining \eqref{e:low1} and \eqref{e:low2}, 
\begin{align}
\frac{1}{2\pi}\sqrt{\frac{E_t^+-x}{t^{3/2}}}\leq \frac{2}{3}\frac{\sqrt{4t-x^2}}{2\pi t}\leq \varrho_t(x)\leq \frac{3}{2}\frac{\sqrt{4t-x^2}}{2\pi t}\leq \frac{2}{\pi}\sqrt{\frac{E_t^+-x}{t^{3/2}}}.
\end{align}

 Now, recall that $\gamma_t(y)$ is the inverted cumulative density of $\varrho_t$  (as defined in \eqref{gammay}), and so $-\del_t \gamma_t(y)=1/\varrho_t(\gamma_t(y))$ (where we have used the facts by \eqref{e:den2}, \eqref{e:den1}, and \eqref{e:den1-} that $\varrho_t (y) > 0$ for each $y \in [E_t^-, E_t^+]$). Next we show that $\gamma_t(4/7)\geq -\sqrt t$. This follows from that
\begin{align*}
\int_{-\infty}^{-\sqrt t}\varrho_t(x) dx \leq \displaystyle\int_{E_t^-}^{-\sqrt{t}} \frac{2}{\pi}\sqrt{\frac{x-E_t^-}{t^{3/2}}} dx
& \leq \frac{2}{\pi}\frac{2}{3}\frac{(\sqrt t+2c)^{3/2}}{t^{3/4}} \\ 
& = \frac{4}{3\pi}\left(1+\frac{2c}{\sqrt t}\right)^{3/2}\leq \frac{4}{3\pi}\left(1+\frac{4}{25M^2}\right)^{3/2}\leq  \frac{3}{7},
\end{align*}
where in the first inequality we used \eqref{e:den1-};  in the second inequality we used $E_t^-\geq -2\sqrt t-2c$ from the first statement of the lemma; and in the last two inequalities we used that $c/\sqrt t \leq 2/25M^2$ and $M=100$ following from our assumption in \Cref{l:density1}. We conclude from combining \eqref{e:den2} and \eqref{e:den1} that, for $\gamma_t(4/7)\leq x\leq E_t^+$,
\begin{align}\label{e:rhotbb}
\frac{1}{4\pi}\sqrt{\frac{E_t^+-x}{t^{3/2}}}\leq \varrho_t(x)\leq \frac{4}{\pi}\sqrt{\frac{E_t^+-x}{t^{3/2}}}.
\end{align}
By plugging $x=\gamma_t(y)$ with $0\leq y\leq 4/7$ into \eqref{e:rhotbb}, we get
\begin{align}\label{e:dygammat}
\frac{\pi}{4}\sqrt{\frac{t^{3/2}}{E_t^+-\gamma_t(y)}}\leq -\del_y \gamma_t(y)=-\frac{1}{\varrho_t(\gamma_t(y))}\leq 4\pi\sqrt{\frac{t^{3/2}}{E_t^+-\gamma_t(y)}}.
\end{align}
We can rearrange the above expression as
\begin{align}
\frac{3\pi}{8}t^{3/4}\leq \del_y (E_t^+-\gamma_t(y))^{3/2}\leq 6\pi t^{3/4}.
\end{align}
Integrating from $0$ to $y$, we conclude that 
\begin{align}\label{e:lowbEt}
\left(\frac{3\pi y}{8}\right)^{2/3}\sqrt t\leq E_t^+-\gamma_t(y).
\end{align}

By plugging \eqref{e:lowbEt} into \eqref{e:dygammat}, we get
\begin{align}
 -\del_y \gamma_t(y)=-\frac{1}{\varrho_t(\gamma_t(y))}\leq 4\pi\sqrt{\frac{t^{3/2}}{E_t^+-\gamma_t(y)}} \leq 8\pi (3\pi y)^{-1/3}\sqrt t.
\end{align}
By integrating the above upper bound from $y$ to $y'$, we get the upper bound
\begin{align*}
	\gamma_t(y)-\gamma_t(y')\leq \int_y^{y'} 8\pi (3\pi u)^{-1/3}\sqrt t \rd u= (24 \pi)^{2/3}t^{1/2}(y'^{2/3}-y^{2/3}).
\end{align*}
This finishes the proof of \eqref{e:gammaest}.
\end{proof}

\begin{proof}[Proof of \Cref{c:oprigidity}]
Recall from \Cref{t:bulkrigidity} that, with probability $1-Ce^{-(\log n)^2}$, we have for any $0\leq t\leq \sfT$ and $1\leq i\leq n$ that 
\begin{align}\label{e:bulkrigidity}
\gamma_{i+\lfloor(\log n)^5\rfloor}(t)-n^{-1}\leq \lambda_i(t)\leq \gamma_{i-\lfloor(\log n)^5\rfloor}(t)+n^{-1}.
\end{align}

We will only prove \Cref{c:oprigidity} for $i\leq n/2$, as the case $i\geq n/2$ follows by the same argument. We notice that $t\geq \fB\sqrt c=2^{60}\sfT^{18}\sqrt c\geq 10^{18} c^2\geq (25 c\cdot 100^2)^2$, where we used $\sfT\geq 1$ and $c\leq 1$. This verifies the assumption in \Cref{l:density1}. Thus for $1\leq i< j\leq 4n/7$, using \eqref{e:gammaest}  we have
\begin{align}\begin{split}\label{e:gammaest2}
\gamma_i(t)-\gamma_j(t)
&\leq (24 \pi)^{2/3}t^{1/2}((j-1/2)^{2/3}-(i-1/2)^{2/3})n^{-2/3}\\
&\leq (24 \pi)^{2/3}t^{1/2}(j^{2/3}-i^{2/3})n^{-2/3}+\frac{\log n}{i^{1/3} n^{2/3}},
\end{split}\end{align}
\noindent provided $n$ is large enough. Then, \eqref{e:bulkrigidity} and \eqref{e:gammaest2} together imply that with probability $1-Ce^{-(\log n)^2}$ we have for each $2(\log n)^5 \leq i \leq n/2$ that
\begin{align}\begin{split}\label{e:rigidity1}
&\phantom{{}={}}|\lambda_i(t)-\gamma_{i}(t)|\leq |\gamma_{i-\lfloor(\log n)^5\rfloor}(t)- \gamma_{i+\lfloor(\log n)^5\rfloor}(t)|+2n^{-1}\\
&\leq (24 \pi)^{2/3}t^{1/2}n^{-2/3}((i+\lfloor(\log n)^5\rfloor)^{2/3}-(i-\lfloor(\log n)^5\rfloor)^{2/3})+\frac{\log n}{i^{1/3} n^{2/3}}+2n^{-1}\\
&\leq (24 \pi)^{2/3}t^{1/2}n^{-2/3}\frac{2\lfloor(\log n)^5\rfloor}{(i-\lfloor(\log n)^5\rfloor)^{1/3}}+\frac{\log n}{i^{1/3} n^{2/3}}+2n^{-1} \leq \frac{(\log n)^6}{i^{1/3}n^{2/3}},
\end{split}\end{align}
provided $n$ is large enough. 

To address the remaining $i$, observe from \Cref{t:main} and \eqref{e:bulkrigidity} that, with probability $1-Ce^{-(\log n)^2}$, we have for each $1 \le i \le 2(\log n)^5$ that 
\begin{align}\label{e:llagamma}
\gamma_{i+\lfloor(\log n)^5\rfloor}(t)\leq \lambda_i(t)\leq \lambda_1(t)\leq E_t^++\frac{(\log n)^{15}}{n^{2/3}}, \quad \text{so }\quad \lambda_i(t), \gamma_i(t)\in \left[\gamma_{3\lfloor(\log n)^5\rfloor}(t), E_t^++\frac{(\log n)^{15}}{n^{2/3}}\right].
\end{align}
It follows that for each $1 \le i \le 2(\log n)^5$, and $n$ large enough 
\begin{align}\begin{split}\label{e:rigidity2}
|\lambda_i(t)-\gamma_{i}(t)| & \leq \frac{(\log n)^{15}}{n^{1/3}}+|E_t^+- \gamma_{3\lfloor(\log n)^5\rfloor}(t)|\\
&\leq \frac{(\log n)^{15}}{n^{1/3}}+(24 \pi)^{2/3}t^{1/2}n^{-2/3}(3(\log n)^5)^{2/3}\leq \frac{(\log n)^{17}}{i^{1/3}n^{2/3}},
\end{split}\end{align}
where the first inequality follows from \eqref{e:llagamma} and the second from \eqref{e:gammaest}.
The two estimates \eqref{e:rigidity1} and \eqref{e:rigidity2} together give \Cref{c:oprigidity}.
\end{proof}

\begin{proof}[Proof of \Cref{initialsmall2}]
\Cref{initialsmall2} is a quick consequence of \Cref{c:oprigidity}. We rescale time by $n^{1/3}$ and space by $n^{2/3}$, and shift $x_1(0)$ to $c/2$, defining 
\begin{align}\label{e:lainitial}
\lambda_i(t)=\frac{x_i(tn^{1/3})-x_1(0)+cn^{2/3}/2}{n^{2/3}}, 
\end{align}
Then, it is quickly verified that $\{\lambda_i(t)\}_{1\leq i\leq n}$ is a Dyson Brownian motion satisfying \eqref{e:DBM}; we also have $\lambda_i(0)\in [-c/2, c/2]$ for each $1 \le i \le n$. Moreover, by taking $\sfT=B$, the assumptions in \Cref{c:oprigidity} are satisfied; hence, with probability $1 - \mathcal{O} (e^{-(\log n)^2})$, \eqref{e:rr} holds for each $t\in [1/B, B]$ and $1 \le i \le n$. 

This, together with \eqref{e:gammaest} gives that, with probability $1-Ce^{-(\log n)^2}$, for each $t \in [ 1 / B, B ]$ and $1\leq i<j\leq n/2$ we have
\begin{align}\begin{split}\label{e:lalabound}
|\lambda_i(t)-\lambda_j(t)|
&\leq |\gamma_i(t)-\gamma_j(t)|+|\lambda_i(t)-\gamma_i(t)|+|\lambda_j(t)-\gamma_j(t)|\\
&\leq (24\pi)^{2/3}t^{1/2}\frac{(j-1 / 2)^{2/3}-(i-1 / 2)^{2/3}}{n^{2/3}}+\frac{2(\log n)^{17}}{i^{1/3}n^{2/3}}\\
&\leq (24\pi)^{2/3}t^{1/2}\frac{j^{2/3}-i^{2/3}}{n^{2/3}}+\frac{(\log n)^{20}}{i^{1/3}n^{2/3}}.
\end{split}\end{align}
\Cref{initialsmall2} follows from substituting \eqref{e:lainitial} into \eqref{e:lalabound}.
\end{proof}

\subsection{Stronger Concentration Bounds for Short Time}
\label{s:Strongbound}

For very short time $t \ll 1$, the concentration of the largest particle on the scale $n^{-2/3+\oo(1)}$ provided by \Cref{t:main}  may not be optimal. In this section, we study this phenomenon through the example when the initial data $\mu_0^{(n)}$ is prescribed by the classical locations of the uniform measure $\textbf{1}_{[-1, 0]} \mathrm{d} x$. Throughout this section, we set $\mu_0=\textbf{1}_{[-1,0]} \mathrm{d} x$, and we denote by $\mu_t= \mu_0 \boxplus \mu_{\semci}^{(t)} (x)$ the free convolution of $\mu_0$ with the rescaled semicircle distribution $\mu_{\semci}^{(t)}$. Let $\varrho_t \in L^1 (\mathbb{R})$ denote the density of $\mu_t$ with respect to Lebesgue measure. Also let $E_t^- = \min (\supp \mu_t)$ and $E_t^+ = \max (\supp \mu_t)$. 

The following lemma approximates $E_t^+$ and $\varrho_t$ (with its derivative) around $E_t^+$.

\begin{lem}\label{l:densitycomp}

For any real number $\sfT \ge 1$, there exists a constant $C=C(\sfT) > 1$ such that, for any $0\leq t\leq \sfT$, the following statements hold. 
\begin{enumerate}
\item 
The density $\varrho_t(y)$ is symmetric around $-1/2$, and satisfies 
\begin{flalign} 
	\label{rhoy}
	\varrho_t(y)\leq 1, \qquad \text{for each $y \in \mathbb{R}$}.
\end{flalign} 
The spectral edge satisfies
\begin{align}\label{e:Et+bound}
|E_t^+-(1+\log(t^{-1}))t|\leq Ct^2,\quad |E_t^- -(-1-(1+\log(t^{-1}))t)|\leq Ct^2,
\end{align}

\item
The density is bounded:
\begin{align}\label{e:rhotsqure}
C^{-1} \min\left\{\sqrt{\frac{x}{t}},1\right\}\leq \varrho_t(E^+_t-x)\leq C \min\left\{\sqrt{\frac{x}{t}},1\right\},\qquad \text{whenever $-1/2 \le E_t^+ - x \le E_t^+$.}
\end{align}
and
\begin{align}\label{e:rhotsqure-}
C^{-1} \min\left\{\sqrt{\frac{x}{t}},1\right\}\leq \varrho_t(E^-_t+x)\leq C \min\left\{\sqrt{\frac{x}{t}},1\right\},\qquad \text{whenever $E_t^- \le E_t^- + x \le -1/2$.}
\end{align}
\item The derivative of the density satisfies: 
\begin{align}\label{e:rhotder}
|\varrho'_t(E^+_t-x)|\leq C \min\left\{\sqrt{\frac{1}{xt}},\frac{t}{x^2}\right\}, \qquad \text{whenever $-1/2 \le E_t^+ - x \le E_t^+$.}
\end{align} 
and
\begin{align}\label{e:rhotder-}
|\varrho'_t(E^-_t+x)|\leq C \min\left\{\sqrt{\frac{1}{xt}},\frac{t}{x^2}\right\}, \qquad \text{whenever $E_t^- \le E_t^- + x \le -1/2$.}
\end{align}  
\end{enumerate}
\end{lem}

\begin{proof}[Proof of \Cref{l:densitycomp}] 
 
We denote the Stieltjes transform of $\varrho_t$ by
\begin{align*}
m_t(z)=\int_\bR\frac{\varrho_t(x)d x}{x-z},\qquad \text{so that} \qquad m_0(z)=\log \frac{z}{z+1}.
\end{align*}

\noindent Recall from \eqref{mt} that $m_t$ satisfies 
\begin{align}
m_t(z_t)=m_0(z),\qquad \text{where} \quad z_t=z_t(z)=z-tm_0(z).
\end{align}

\begin{figure}
\centering
\includegraphics[width=1\textwidth, trim=2cm 1cm 2cm 1cm, clip]{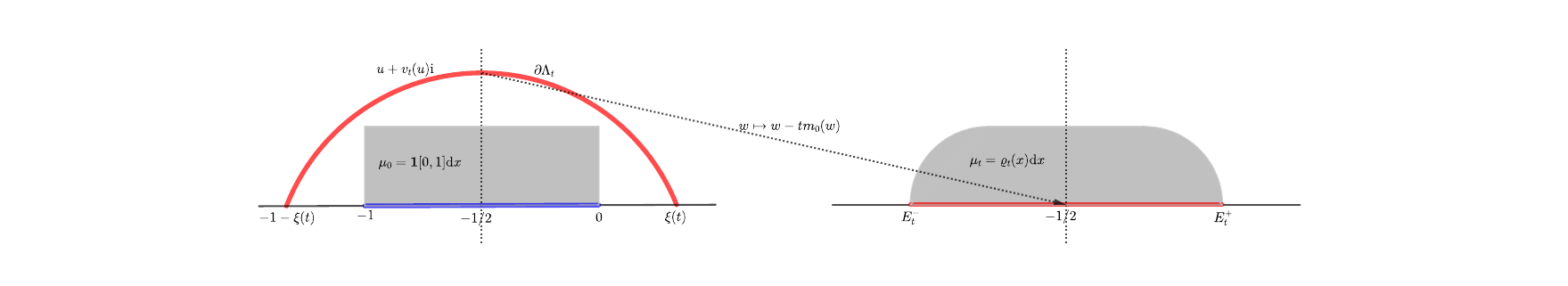}
\caption{Left panel shows $\del \Lambda_t\cap \bH$ from $-1-\xi(t)$ to $\xi(t)$, and the density $\mu_0$ (the gray region). 
Right panel is the density $\mu_t= \mu_0 \boxplus \mu_{\semci}^{(t)} (x)$ supported on $[E_t^-, E_t^+]$. Both of them are symmetric around $-1/2$. }
\label{f:density_map}
\end{figure}

Recalling the region $\Lambda_t$ from \eqref{mtlambdat} and the function $v_t$ from \eqref{vte}, \Cref{i:boundary} in \Cref{yconvolution} indicates that we can parametrize $\del \Lambda_t$ as $\{u+\ri v_t(u):u\in \bR\}$; we also have
\begin{align}\label{e:boundarye2}
\frac{1}{t}\geq \int_\bR\frac{\rd \mu_0(x)}{|x-(u+\ri v_t(u))|^2}=\int_{-1}^0 \frac{\rd x}{(x-u)^2+v_t(u)^2}.
\end{align}
From \Cref{i:edge} in \Cref{yconvolution}, there exists $\xi^+(t)$ and $\xi^-(t)$ such that $\xi^\pm(t)-tm_0(\xi^\pm(t))=E^\pm_t$ and $\xi^\pm(t)$ are characterized by 
\begin{align}\label{e:firstD}
m_0'(\xi^\pm(t))=\frac{1}{t}, \qquad \text{where we observe that} \qquad m_0'(z)=\int_{-1}^0\frac{\rd x}{(x-z)^2}=\frac{1}{z}-\frac{1}{z+1}.
\end{align}
By the symmetry for $\mu_0$ around $-1/2$, we have $\xi^-(t)=-1-\xi^+(t)$, $E_t^- = - E_t^+ - 1$, and $\varrho_t (E_t^+ - x) = \varrho_t (E_t^- + x)$; see \Cref{f:density_map}. Hence, it suffices to show all of the bounds stated in the lemma for $E_t^+$ (or for $x$ close to $E_t^+$). So, throughout this proof, we abbreviate $E_t=E_t^+$, and $\xi(t)=\xi^+(t)$.

From \Cref{mz} that, for any $y\in \bR$, there exists $w=w_t(u)=u+\ri v_t(u)\in \del \Lambda_t$ such that $y=w-tm_0(w)$. Next we show that $v_t(u)>0$ for $-1-\xi(t)<u<\xi(t)$ as in \Cref{f:density_map}. Otherwise, $v_t(u)=0$; we may assume by symmetry that $u\geq -1/2$. If $u\leq 0$, then we have from \eqref{e:boundarye2} that
\begin{align}
\frac{1}{t}\geq \int_{-1}^0 \frac{\rd x}{(x-u)^2}=\infty,
\end{align}
which is a contradiction. If $0<u<\xi(t)$, then \eqref{e:boundarye2} gives
\begin{align}
\frac{1}{t}\geq \int_{-1}^0 \frac{\rd x}{(x-u)^2}>\int_{-1}^0 \frac{\rd x}{(x-\xi(t))^2}=\frac{1}{t},
\end{align}
where in the middle inequality we used that for $x\in [-1,0]$ we have $\xi(t)-x>u-x>0$. This is again a contradiction. Thus $v_t(u)>0$ for $-1-\xi(t)<u<\xi(t)$. 
From \Cref{i:rhoy} in \Cref{yconvolution},  $y\in \bR$ satisfies $\varrho_t (y) > 0$ if and only if $w\in \del \Lambda_t\cap \bH$. Therefore, the discussion above, together with \Cref{i:boundary} and \Cref{i:rhoy} in \Cref{yconvolution} gives 
\begin{align}\label{e:defwu}
\frac{1}{t}=\int_\bR\frac{\rd \mu_0(x)}{|x-(u+\ri v_t(u))|^2}, \quad \varrho_t(y)=\frac{\Im[w]}{\pi t}=\frac{v_t(u)}{\pi t}, \quad \text{for} \quad y\in [-1-E_t, E_t].
\end{align}

Using the relation \eqref{e:firstD}, we can solve $\xi(t)$ explicitly, for each $0\leq t\leq \sfT$. In particular, we have
\begin{align}\label{e:xit}
\xi(t)^2+\xi(t)=t,\quad \text{so} \quad \frac{t}{2\sfT}\leq \xi(t)=\frac{2t }{1+\sqrt{1+4t}}\leq t,\quad \text{which implies} \quad |\xi(t)-t|\leq \big| \xi(t) \big|^2 \le t^2.
\end{align}

\noindent By \eqref{e:defwu}, we have
\begin{align}\label{e:densityupper}
\frac{1}{t}=\int_\bR\frac{\rd \mu_0(x)}{|x-w|^2}\leq \int_{-\infty}^\infty \frac{\rd x}{(x-u)^2+v_t(u)^2}=\frac{\pi}{v_t(u)},
\end{align}

\noindent and hence $\varrho_t(y)= v_t(u)/\pi t\leq 1$, verifying \eqref{rhoy}. Moreover, using the estimate of $\xi(t)$ from \eqref{e:xit}, with the relation $\xi(t)-tm_0(\xi(t))=E_t$ from \Cref{i:edge} in \Cref{yconvolution}, we have
\begin{align*}
E_t=\xi(t)-tm_0(\xi(t))=\xi(t)-t\log \frac{\xi(t)}{\xi(t)+1}
=(1+\log(t^{-1})+\OO(t ))t.
\end{align*}

\noindent This finishes the proof of \eqref{e:Et+bound}.

Next we show \eqref{e:rhotsqure}, from which \eqref{e:rhotsqure-} would follow by symmetry. From \Cref{i:rhoy} in \Cref{yconvolution},  $\varrho_t (E_t - x) > 0$ if and only if there exists some $w \in \partial \Lambda_t \cap \mathbb{H}$ such that $E_t - x = w - tm_0 (w)$; in this case, $\varrho_t (E_t - x) = (t \pi)^{-1} \Imaginary w$. Equivalently, let $\Delta =  w - \xi(t)$. We have $\varrho_t (E_t - x) > 0$ if and only if there exists some $\Delta = \Delta(x)$ satisfying $\Delta + \xi(t) \in \mathbb{H}$ such that
\begin{align}\label{e:Delta0}
\xi(t)+\Delta -t m_0(\xi(t)+\Delta )=E_t-x, \quad \text{in which case} \quad \varrho_t(E_t-x)=\Im[\Delta]/t\pi.
\end{align}

In the following we fix the constant $M=(4\sfT)^3$ and let $w=\xi(t)+\Delta$ be some complex number. If $|\Delta|\leq t/M$ then by a Taylor expansion around $\xi(t)$, using \eqref{e:firstD}, we get 
\begin{align}\label{e:taylore}
m_0(w)- m_0 ( \xi(t) ) = m'_0(\xi(t))\Delta+\frac{m_0^{(2)}(\xi(t))\Delta^2}{2}+\frac{\cE(w)\Delta^3}{6}
=\frac{\Delta}{t}+\frac{m_0^{(2)}(\xi(t))\Delta^2}{2}+\frac{\cE(w)\Delta^3}{6}.
\end{align}
where 
\begin{align}\label{e:cEbb}
|\cE(w)|\leq 2\max_{|z-\xi(t)|\leq t/M}\left|\frac{1}{z^3}- \frac{1}{(z+1)^3}\right|\leq \frac{(4\sfT)^3}{t^3},
\end{align}

\noindent where we used the facts that $m_0 (z) = \log z - \log (z + 1)$ (so that $m_0^{(3)} (z) = 2 z^{-3} - 2(z+1)^{-3}$) and \eqref{e:xit}, which implies that for $|z-\xi(t)|\leq t/M$ we have $|z|, |z+1|\geq t/(2\sfT)-t/M\geq t/(16^{1/3}\sfT)$. 

By plugging \eqref{e:taylore} into \eqref{e:Delta0} (and using the fact that $E_t = \xi (t) - t m_0 ( \xi(t) )$), we get that $\varrho_t (E_t - x) > 0$ if and only if there exists some $\Delta = \Delta(x) \in \mathbb{H}$ such that
\begin{align}\label{e:Delta}
\frac{x}{t}=\frac{m_0^{(2)}(\xi(t))\Delta^2}{2}+\frac{\cE(w)\Delta^3}{6}=\frac{m_0^{(2)}(\xi(t))\Delta^2}{2}\left(1+\frac{\cE(w)\Delta}{3m_0^{(2)}(\xi(t))}\right).
\end{align}

By explicit computation using the formula of $m_0(z) = \log z - \log (z+1)$ (so that $m_0^{(2)} (z) = (z+1)^{-2} - z^{-2}$ and $m_0^{(3)} (z) = 2z^{-3} - 2(z+1)^{-3}$) and $\xi(t)\leq t\leq \sfT$ from \eqref{e:xit}, we have 
\begin{align}\begin{split}\label{e:mder}
&\frac{1}{t^2 }\leq -m_0^{(2)}(\xi(t))=\frac{2\xi(t)+1}{\xi(t)^2(\xi(t)+1)^2}=\frac{1+2\xi(t)}{t^2}\leq \frac{3\sfT}{t^2 },\\
&|m_0^{(3)}(\xi(t))|= 2\left|\frac{3\xi(t)^2+3\xi(t)+1}{\xi(t)^3(\xi(t)+1)^3}\right|
=2\left|\frac{3t+1}{t^3}\right|
\leq \frac{8\sfT}{t^3}.
\end{split}\end{align}
 
It follows that the error term on righthand side of \eqref{e:Delta}, when $|\Delta|\leq t/M$, is bounded by
\begin{align}\label{e:Deltaexp2}
\left| \frac{\cE(w)\Delta}{3m_0^{(2)}\big(\xi(t) \big)}\right|\leq \frac{(4\sfT)^3}{t^3} \frac{t}{M}\frac{t^2}{3}\leq \frac{1}{3},
\end{align}

\noindent where we used \eqref{e:cEbb} and \eqref{e:mder} (and the fact that $M = (4 \mathsf{T})^3$). 

Next we show the righthand side of \eqref{e:Delta} as a function of $\Delta$  is a surjection from $|\Delta|\leq t/M$ onto  $\{u\in \bC: |u|\leq 1/4M^2\}$. Fix any $|u| \leq 1/4M^2$; we will use Rouch\'{e}'s theorem to compare the two holomorphic functions
\begin{align}\label{e:twof}
\frac{m_0^{(2)}(\xi(t))\Delta^2}{2}\left(1+\frac{\cE(w)\Delta}{3m_0^{(2)}(\xi(t))}\right)-u,\quad \text{and}\quad \frac{m_0^{(2)}(\xi(t))\Delta^2}{2}-u
\end{align} 
 on the domain $\{\Delta\in \bC: |\Delta|\leq t/M\}$. Using the first statement in \eqref{e:mder}, we notice that  $m_0^{(2)}(\xi(t))\Delta^2/2-u$ has a root with $|\Delta|\leq t/M$. On the boundary, for any $\Delta\in \bC$ such that $|\Delta|= t/M$, we can upper bound the difference
\begin{align}\label{e:boundarydiff2}\begin{split}
&\phantom{{}={}}\left|\left(\frac{m_0^{(2)}(\xi(t))\Delta^2}{2}\left(1+\frac{\cE(w)\Delta}{3m_0^{(2)}(\xi(t))}\right)-u\right)-\left( \frac{m_0^{(2)}(\xi(t))\Delta^2}{2}-u\right)\right|\\
&=\left|\frac{m_0^{(2)}(\xi(t))\Delta^2}{2}\frac{\cE(w)\Delta}{3m_0^{(2)}(\xi(t))}\right|\leq \frac{1}{3}\left|\frac{m_0^{(2)}(\xi(t))\Delta^2}{2}\right|\leq \left|\frac{m_0^{(2)}(\xi(t))\Delta^2}{2}-u\right|
\end{split}\end{align} 
where in the second statement we used \eqref{e:Deltaexp2}; in the last inequality we used that, when $|\Delta|=t/M$, 
\begin{align}
-\frac{2}{3} \frac{ m_0^{(2)}(\xi(t))|\Delta|^2}{2}
= -\frac{2}{3} \frac{ m_0^{(2)}(\xi(t))}{2}\frac{t^2}{M^2}
> \frac{1}{4M^2}\geq |u|,
\end{align}
where the second statement follows from $-m_0^{(2)}(\xi(t))\geq 1/t^2$ from the first statement in \eqref{e:mder}.
The relation \eqref{e:boundarydiff2} verifies the conditions of Rouch\'{e}'s theorem, and we conclude that the first function in \eqref{e:twof} (as a function of $\Delta$)  also has a root with $|\Delta|\leq t/M$. Thus we can use \eqref{e:mder} to solve for $\Delta=\Delta(x)$ for any $0\leq x\leq t/4M^2$. We collect it here for the future use:
\begin{align}\label{e:xtoDelta}
|\Delta(x)|\leq t/M, \quad \text{when},\quad 0\leq x\leq t/4M^2
\end{align}

By \eqref{e:Delta}, we can solve for $\Delta$, given by
\begin{align}\label{e:Deltaexp1}
\Delta =\ri \sqrt{\frac{2x}{-tm_0^{(2)}(\xi(t))}}\left( 1+\frac{\cE(w)\Delta}{3m_0^{(2)}(\xi (t))}\right)^{-1/2}.
\end{align}
 By plugging \eqref{e:mder} and \eqref{e:Deltaexp2} into \eqref{e:Deltaexp1}, we conclude for $0\leq x\leq t/4M^2$ that
\begin{align}\label{e:Dtbound}
\sqrt{\displaystyle\frac{tx}{2\mathsf{T}}} \leq \sqrt{\frac{2x}{(3\sfT/t)}} \left(1+\frac{1}{3}\right)^{-1/2}\leq |\Delta |, \Im[\Delta]\leq  \sqrt{\frac{2x}{(1/t)}} \left(1-\frac{1}{3}\right)^{-1/2}\leq \sqrt{3 tx}.
\end{align}
As a consequence, using \eqref{e:defwu} and \eqref{e:Dtbound}, we get 
\begin{align}\label{e:bb1}
\frac{1}{\pi}\sqrt{\displaystyle\frac{x}{2t\mathsf{T}}}\leq \varrho(E_t-x)=\frac{\Im[\Delta]}{t\pi}\leq \frac{1}{\pi}\sqrt{\frac{3x}{t}}
\end{align}
for $0\leq x\leq t/4M^2$. For $x\geq t/4M^2$, next we show
$\varrho_t(E_t-x)$ is an increasing function of $x$ for $0\leq x\leq E_t+1/2$. Then this, together with \eqref{rhoy} and the second statement in \eqref{e:Dtbound}, would imply that
\begin{align}\label{e:bb2}
1\geq \varrho_t(E_t-x)\geq \varrho_t(E_t-t/4M^2)\geq \frac{1}{2\pi M \sqrt{2\sfT}},
\quad
\text{for} 
\quad t/4M^2\leq  x\leq E_t+1/2,
\end{align}

\noindent so \eqref{e:bb1} and \eqref{e:bb2} together give \eqref{e:rhotsqure}.

Next we use \eqref{e:defwu} to show that 
$\varrho_t(E_t-x)$ is an increasing function of $x$ for $0\leq x\leq E_t+1/2$. First we notice that  on $[-1/2,\xi(t)]$, $v_t(u)$ is a decreasing function. Otherwise, there are some $-1/2\leq u<u'\leq \xi(t)$ such that $ v_t(u)\leq v_t(u')$. Then \eqref{e:defwu} leads to the contradiction,
\begin{align}
\frac{1}{t}=\int_{-1}^0\frac{\rd x}{(x-u)^2+ v_t(u)^2}\geq \int_{-1}^0\frac{\rd x}{(x-u)^2+ v_t(u')^2}>\int_{-1}^0\frac{\rd x}{(x-u')^2+ v_t(u')^2}=\frac{1}{t},
\end{align}
where the third statement follows from the following bound (by taking $b=v_t(u')$ and integrating $a$ from $u$ to $u'$): for any $a> -1/2$ and $b\geq 0$,
\begin{align}\label{e:calcu}
\del_a \int_{-1}^0\frac{\rd x}{(x-a)^2+b^2} = -\displaystyle\int_{-1}^0 \partial_x \bigg( \displaystyle\frac{1}{(x-a)^2 + b^2} \bigg) \mathrm{d} x = \displaystyle\frac{1}{(a+1)^2 + b^2} - \displaystyle\frac{1}{a^2 + b^2} <0.
\end{align}

From the discussion above \eqref{e:defwu}, and \Cref{i:bijection} and \Cref{i:rhoy} from \Cref{yconvolution}, the map $M$ (from \eqref{mtlambdat}) maps $\{u+v_t(u)\ri: -1-\xi(t)\leq u\leq \xi(t)\}$ increasingly (as a function of $u$) to $\supp \mu_t=[-1-E_t, E_t]$ (see \Cref{f:density_map}).
By symmetry around $-1/2$, 
\begin{align}\label{e:map}
M \text{ maps }\{u+v_t(u)\ri: -1/2\leq u\leq \xi(t)\} \text{ increasingly (as a function of $u$) to } [-1/2, E_t]. 
\end{align}
This together with  $\varrho_t(M(u+v_t(u)\ri)=v_t(u)/t\pi$ from  \eqref{e:defwu}, and $v_t(u)$ is decreasing on $[-1/2, \xi(t)]$, leads to that $\varrho_t(y)$ is decreasing on $[-1/2, E_t]$. Thus $\varrho_t(E_t-x)$ is an increasing function of $x$ for $0\leq x\leq E_t+1/2$, and this together with \eqref{e:bb2} finishes the claim.

We can further estimate the derivative of the density. As mentioned previously, we can view the quantity $\Delta$ in \eqref{e:Delta0} as a function of $x$. Taking its derivative with respect to $x$ yields
\begin{align}\label{e:drhot}
\del_x\Delta=-\frac{1}{1-tm_0'(\xi(t)+\Delta)}, \quad \varrho(E_t-x)=\Im[\Delta]/t\pi. 
\end{align}

Recall $M=(4\sfT)^3$. We distinguish between whether $x\leq t/(4M^2)$ or $x\geq t/(4M^2)$.
If $x\leq t/(4M^2)$, recall from \eqref{e:xtoDelta}, we have $|\Delta|\leq t/M$. In this case, the Taylor expansion of $m_0'(\xi(t)+\Delta)$ around $\xi(t)$ is 
\begin{align}
m_0'(\xi(t)+\Delta)=m_0'(\xi(t))+m_0^{(2)}(\xi(t))\Delta +\frac{\widetilde \cE(\Delta)\Delta^2}{2},\quad |\widetilde \cE(\Delta)|\leq \frac{(4\sfT)^3}{t^3},
\end{align}
where the bound of the remainder term follows from \eqref{e:cEbb}. This together with $m'_0(\xi(t))=1/t$ from \eqref{e:firstD} gives 
\begin{align}\label{e:exp2}
\left|1-tm_0'(\xi(t)+\Delta)+tm_0^{(2)}(\xi(t))\Delta\right|=\frac{|t\widetilde\cE(w)\Delta^2|}{2}
\leq \frac{2^5\sfT^3\Delta^2 }{t^2}.
\end{align}
We can reformulate \eqref{e:exp2} as
\begin{align*}
\left|1-tm_0'(\xi(t)+\Delta)\right|\geq \frac{|\Delta|}{t}-\frac{2^5\sfT^3\Delta^2 }{t^2}\geq \frac{|\Delta |}{2t}\geq \frac{1}{2\sqrt{2\sfT}}\sqrt{\frac{x}{t}},
\end{align*}
where the first equality follows from \eqref{e:exp2} and the first statement of \eqref{e:mder};
the second inequality is from $|\Delta| \le M^{-1} t$ and $M = (4 \mathsf{T})^3 $; and the last inequality follows from the first statement in \eqref{e:Dtbound} and our assumption $x\leq t/(4M^2)$. 
Plugging the above estimate into \eqref{e:drhot}, we conclude that
\begin{align}\label{e:rhotbb1}
|\varrho_t'(E_t-x)|\leq \frac{|\del_x \Delta |}{t\pi}=\frac{1}{t\pi}\frac{1}{|1-tm_0'(\xi(t)+\Delta)|}\leq \frac{\sqrt{8\sfT}}{\pi\sqrt{xt}}.
\end{align}

Now assume that $x\geq t/4M^2$. Recalling that $w = \xi(t) + \Delta=u+v_t\ri$, we must upper bound 
\begin{align}\label{2rhoderivative}
\varrho_t'(E_t-x)=\frac{\del_x\Im[\Delta]}{t\pi}=-\frac{\Im[m_0'(\xi(t)+\Delta)]}{\pi |1-tm_0'(\xi(t)+\Delta)|^2},
\end{align}
\noindent where the second statement follows from taking imaginary part of the first statement in \eqref{e:drhot}. By \eqref{e:map} and \eqref{e:densityupper}, we only need to upper bound \eqref{2rhoderivative} for 
\begin{align}\label{e:domain}
-1/2\leq \Re[w]=u\leq \xi(t),\quad \Im[w]\leq \pi t.
\end{align}
For the denominator in \eqref{2rhoderivative} observe that from the definition \eqref{mz0} that 
\begin{align}\label{e:fenmu}
	\begin{aligned} 
&\phantom{{}={}}\Re[1-tm_0'(w)]
=1-t\Re\left[\int_{-1}^0\frac{\rd x}{(x-w)^2}\right]
=1-t\int_{-1}^0\frac{\Re[(x-\overline{w})^2]\rd x}{|x-w|^4}\\
&=1-t\int_{-1}^0\frac{(|w-x|^2-2 \Im[w]^2)\rd x}{|w-x|^4}
=1-t\int_{-1}^0\frac{\rd x}{|w-x|^2}+2t\Im[w]^2\int_{-1}^0\frac{\rd x}{|w-x|^4}\\
&=2t\Im[w]^2\int_{-1}^0\frac{\rd x}{|w-x|^4}
\geq 2t \left(\frac{t}{2 M \sqrt{\mathsf{2T}}} \right)^2\int_{-1}^0 \frac{\rd x}{|u+\ri\pi t-x|^4}
=\frac{1}{4M^2\sfT} t^3\int_{-1}^0\frac{\rd x}{((u-x)^2+(\pi t)^2)^2}
\end{aligned} 
\end{align}
where for the first equality in the third line we used \eqref{e:defwu}; for the middle inequality we used that $\Im[w]=\Im[\Delta(x)]=t\pi\varrho_t(E_t-x)$ from \eqref{e:Delta0}, and $\varrho_t(E_t-x)\geq1/(2\pi M \sqrt{2\sfT})$ for $x\geq t/4M^2$ from \eqref{e:bb2},
thus
\begin{align}\label{e:Imw}
\Im[w]\geq \frac{t}{2M\sqrt{2\sfT}},\quad
\text{for} 
\quad t/4M^2\leq  x\leq E_t+1/2,
\end{align}
 and the fact that $\Im[w]\leq \pi t$ from \eqref{e:domain}. 
For the last integral in \eqref{e:fenmu} we have
\begin{align}\begin{split}\label{e:boundint}
t^3\int_{-1}^0 \frac{\rd x}{((u-x)^2+(\pi t)^2)^2}
&=\int_{-1/t}^0 \frac{\rd x}{((u/t-x)^2+\pi^2)^2}
\geq \int_{\min\{0,u/t\}-1/(2\sfT)}^{\min\{0,u/t\}} \frac{\rd x}{((u/t-x)^2+\pi^2)^2}\\
&\geq \int_{\min\{0,u/t\}-1/(2\sfT)}^{\min\{0,u/t\}} \frac{\rd x}{((1+1/2\sfT)^2+\pi^2)^2}
\geq \frac{1}{2\sfT} \frac{1}{(2^2+\pi^2)^2}
\end{split}\end{align}
where in the first statement we changed variables $x\rightarrow tx$; in the second statement we restricted the integral and used $\min\{0,u/t\}-1/(2\sfT)\geq -1/(2t)-1/(2\sfT)\geq -1/t$ (where $u\geq -1/2$ is from \eqref{e:domain}); in the third statement we used  
$u\leq \xi(t)\leq t$ (from \eqref{e:domain}), so $|u/t-x|\leq 1+1/(2\sfT)$;
and in the last inequality we used $\sfT\geq 1$. 

To bound the numerator in \eqref{2rhoderivative},
we first notice that 
\begin{align}\begin{split}\label{e:bigbang}
x&=|(\xi(t)-tm_0(\xi(t)))-(w-tm_0(w))| \le |\xi(t)-w|\left(1+t\int_{-1}^0 \frac{\rd x}{|w-x||\xi(t)-x|}\right)\\
&\leq |\xi(t)-w|\left(1+t\sqrt{\int_{-1}^0 \frac{\rd x}{|w-x|^2}\int_{-1}^0 \frac{\rd x }{|\xi(t)-x|^2}}\right)
=2 |\xi(t)-w|,
\end{split}\end{align}
where in the last equality we used \eqref{e:defwu}. Recall from\eqref{e:Imw} that $\Im[w]\geq t/(2M\sqrt{2\sfT})$ 
and from \eqref{e:xit} that $|\xi(t)|\leq t$. Thus, 
\begin{align}\label{e:ulob}
\left(1+4M \sqrt{\mathsf{T}}\right)|w|\geq  |w|+|\xi(t)|\geq  |w-\xi(t)|\geq \frac{x}{2},
\end{align}
where the last inequality is from \eqref{e:bigbang}.
By using the formula for $m_0'(z)$ from \eqref{e:firstD}, with \eqref{e:ulob}, we get
\begin{align}\label{e:fenzi}
\big| \Im[m_0'(w)] \big| = \Bigg| \Im\left[ \frac{1}{w}-\frac{1}{(w+1)} \right] \Bigg|
=\frac{\big| \Im[w(w+1)] \big|}{|w(w+1)|^2} \leq 4\frac{\big| 2\Re[w]+1 \big| \Im[w]}{|w|^2}\leq \frac{32 \pi(\sfT + 1)}{(1+4M \sqrt{\mathsf{T}})^2} \frac{t}{x^2},
\end{align}
where we used that $|w+1|\geq 1/2$ (following from \eqref{e:domain}) in the third statement and $|\Re[w]|\leq \sfT$ (following from \eqref{e:domain} and $\xi(t)\leq t\leq \sfT$ from \eqref{e:xit}) and $\Im[w]\leq \pi t$ (from \eqref{e:domain}) in the fourth. By plugging \eqref{e:fenmu}, \eqref{e:boundint} and \eqref{e:fenzi} into \eqref{2rhoderivative}, we conclude that there exists a constant $C_0 = C_0 (\mathsf{T}) > 1$ such that, for $|\Delta |\geq t/M$, we have 
\begin{align}\label{e:rhotbb2}
|\varrho_t'(E_t-x)|\leq C_0 \frac{  t}{x^2}.
\end{align}

The estimates \eqref{e:rhotbb1} and \eqref{e:rhotbb2} together give
\begin{align*}
\varrho'_t(E_t-x)\asymp \min\left\{\sqrt{\frac{1}{xt}},\frac{t}{x^2}\right\},\qquad \text{whenever $0\leq x\leq E_t+1/2$.}
\end{align*} 
This finishes the proof of \eqref{2rhoderivative}.
\end{proof}

The following proposition bounds the fluctuations for the largest (and smallest) particle under Dyson Brownian motion run under initial data prescribed by (the classical locations of) $\mu_0 = \textbf{1}_{[-1, 0]}$; it is a quick consequence of \Cref{p:detercoupling} and \Cref{l:densitycomp}. As alluded to in the beginning of this section, observe that it provides a concentration bound with error much smaller than $n^{-2/3}$, if $t \ll n^{-2/3}$. 

\begin{prop}\label{p:uniform}

For any real number $\sfT \ge 1$, there exists a constant $C = C(\mathsf{T}) > 1$ such that the following holds. Let $\lambda_1 (t) > \lambda_2 (t) > \cdots > \lambda_n (t)$ denote Dyson Brownian motion \eqref{e:DBM}, with initial data given by setting $\lambda_i (0) = (1-2i) / (2n)$ for each $1 \le i \le n$. Denote the classical locations of $\mu_t$ (recall \Cref{gammarho}) by $\gamma_j (t) = \gamma_{j;n}^{\mu_t}$. Then for $n$ large enough, we have for any $0\leq t\leq \sfT$  that
\begin{align}\label{e:la1bound}
&\bP\left(\sup_{0\leq s\leq t}|\lambda_1(s)-\gamma_1(s)|\leq C\left(t+\frac{(\log n)\sqrt t}{\sqrt n}\right)\right)\geq 1-Ce^{(-\log n)^2},\\
&\bP\left(\sup_{0\leq s\leq t,\atop 1\leq i\leq n}\lambda_i(s)-\lambda_1(0)+\frac{i-1}{n}\leq C\left(t+t|\log(t)|+\frac{(\log n)\sqrt t}{\sqrt n}\right)\right)\geq 1-Ce^{(-\log n)^2}.
\end{align}
\end{prop}

\begin{proof}[Proof of \Cref{p:uniform}]
	
	Let us first verify the assumptions in \Cref{p:detercoupling} with the $c$ there given by $1/4\fC$ here, where $\fC > 1$ is the constant $C(\mathsf{T})$ from \Cref{l:densitycomp}. The first assumption in that proposition is verified since $\sup_{x \in \mathbb{R}} \varrho_t (x) \le 1 \le (2c)^{-1}$, where the first bound follows from \eqref{rhoy} (and the second from the bound $c < 1$, as $\mathfrak{C} > 1$). To verify the second assumption, observe by \Cref{l:densitycomp}, that for each $0\leq t\leq \sfT$ and $x \in [E_t^-, E_t^+]$ we have
\begin{align}\begin{split}\label{e:derbb}
 |\varrho'_t(x)|&\leq \fC \min\left\{\sqrt{\frac{1}{ t \min\{|E_t^+-x|, |E_t^--x|\}}},\frac{t}{\min\{|E_t^+-x|, |E_t^--x|\}^2}\right\} \le  \frac{\fC}{\min\{|E_t^+-x|, |E_t^--x|\}},
\end{split} \end{align}

\noindent where the second inequality follows from the fact that $\min \{ \sqrt{(ab)^{-1}}, ab^{-2} \big\} \le b^{-1}$ for any real numbers $a, b > 0$ (applied with $(a, b) = \big( t, \min \{ |E_t^+ - x|, |E_t^- - x | \} \big)$). Furthermore, letting $\Gamma_t : [0, 1] \rightarrow \mathbb{R}$ denote the inverted cummulative density function associated with $\mu_t$ (recall \eqref{gammay}), we also have $\Gamma_t (y) - \Gamma_t(y') \ge y' - y$ for any $0 \le y \le y' \le 1$, since $\sup_{x \in \mathbb{R}} \varrho_t (x) \le 1$. Hence, $E_t^+ = \Gamma_t (0) \ge \Gamma_t ( 1 / 2n ) + 1 / 2n = \gamma_1 (t) + 1 / 2n$ and similarly $E_t^- \le \gamma_n (t) - 1 / 2n$. Thus, whenever $\dist(x, \{\gamma_1(t),\gamma_2(t),\cdots, \gamma_n(t)\})\leq c/2n$, we must have that $\dist(x, \{E_t^+, E_t^-\})\geq 1/2n-c/2n\geq 1/4n$. Together with \eqref{e:derbb}, this yields $|\varrho'_t(x)|\leq 4\fC n = c^{-1} n$, thereby verifying the second assumption in \Cref{p:detercoupling}.

Hence, that proposition applies, and so there exists a constant $C = C(\mathsf{T}) > 1$ such that for $n$ large enough and $0\leq t\leq \sfT$ the following holds. 
For each
\begin{align}\label{e:la1bound2}
\bP\left(\sup_{0\leq s\leq t}|\lambda_1(s)-\gamma_1(s)|\leq C\left(t+\frac{(\log n)\sqrt t}{\sqrt n}\right)\right)\geq 1-\OO\left(e^{-(\log n)^2}\right).
\end{align}
This gives the first statement in \eqref{e:la1bound}. Moreover, there exists a constant $\wt C = \wt C(\mathsf{T}) > 1$ such that, with probability $1 - \mathcal{O} (e^{-(\log n)^2})$ we have for each $0 \le t \le \mathsf{T}$ that
\begin{align*}
&\phantom{{}={}}\lambda_i(s)-\lambda_1(0)
\leq \gamma_i(s)-\gamma_1(0)+C\left(t+\frac{(\log n)\sqrt t}{\sqrt n}\right)\\
&= E_s^++\gamma_i(s)-\gamma_1(s)-(E_s^+-\gamma_1(s))-\gamma_1(0)+C\left(t+\frac{(\log n)\sqrt t}{\sqrt n}\right)\\
&\leq E_s^+-\frac{i-1}{n}-\frac{1}{2n}+\frac{1}{2n}+C\left(t+\frac{(\log n)\sqrt t}{\sqrt n}\right)
= E_s^++C\left(t+\frac{(\log n)\sqrt t}{\sqrt n}\right)-\frac{i-1}{n}\\
&\leq \widetilde C\left(t+t|\log t|+\frac{(\log n)\sqrt t}{\sqrt n}\right)-\frac{i-1}{n}
\end{align*}

\noindent where in the first statement we used the facts that $\gamma_1 (0) = \lambda_1 (0)$ and \Cref{p:detercoupling}; in the third statement we used the facts that $\gamma_1(0)=-1/(2n)$, that (since by \Cref{l:densitycomp} we have $\varrho_s(y)\leq 1$) $\gamma_i(s)-\gamma_1(s)\leq -(i-1)/n$, and that $E_s^+-\gamma_1(s)\geq 1/(2n)$; in the fifth statement we used the upper bound of $E_s^+\leq C'(s+s|\log s|)\leq C'(t+t|\log t|)$ from \Cref{l:densitycomp}. This establishes the second statement in \eqref{e:la1bound2}. 
\end{proof}

As a consequence of \Cref{p:uniform}, we next deduce the following result estimating the trajectories of the first (and last) particles under the rescaled Dyson Brownian motion (recall \eqref{e:reDBM}), whose initial data is given by a uniform distribution (of density governed by the parameter $B$ below) added to a delta mass (of weight governed by the parameter $M$ below). It will be used in the forthcoming work \cite{U}.

	\begin{cor}\label{p:extreme}
			
			For any real numbers $B, D > 1$, there exist constants $C_1 = C_1 (B) > 1$ and $C_2 = C_2 (B, D) > 1$ such that the following holds. Let $n \ge k \ge 2$ be integers, and let $L \in [1, k^D]$ be a real number such that $n = L^{3/2} k$. Let $\bm{x}(t) = \big(x_1 (t), x_2 (t), \ldots , x_n (s) \big) \in \overline{\mathbb{W}}_n$ denote rescaled Dyson Brownian motion (from \eqref{e:reDBM}) with initial data $\bm{x}(0)$, run for time $t$. Suppose that, for some real number $M \ge 1$, we have
			\begin{align}\label{e:xi-xj2}
				x_i(0)-x_j(0)\geq \left(\frac{j-i}{BL^{3/4}k}-M \right)k^{2/3}, \qquad \text{for each $1 \le i \le j \le n$}.
			\end{align} 
			
			\noindent Then for any real number $0\leq t\leq 1$, we have
			\begin{align}\label{e:eigbound}
				\begin{aligned}
				\mathbb{P} \Bigg[ \Big\{ x_1 (tk^{1/3})-x_1(0) \leq C_1 k^{2/3} \big( tL^{3/4} (1+|\log t|^2) + ( & Mt)^{1/2} L^{3/8}+ (tk^{-1})^{1/2} \log n \big) \Big\} \Bigg] \\
				& \qquad \qquad \qquad \ge 1 - C_2 e^{-(\log n)^2}. 
				\end{aligned} 
			\end{align}
		\noindent Similarly, for any real number $0 \le t \le 1$, we have    
	\begin{align}
		\label{xntk}
				\begin{aligned}
				\mathbb{P} \Bigg[ \Big\{ x_n (tk^{1/3})-x_n(0) \geq -C_1 k^{2/3} \big( tL^{3/4}(1+|\log t|^2) + ( & Mt)^{1/2} L^{3/8}+ (tk^{-1})^{1/2} \log n \big) \Big\} \Bigg] \\
				& \qquad \qquad \qquad \ge 1 - C_2 e^{-(\log n)^2}. 
				\end{aligned} 
			\end{align}			 
		\end{cor}

\begin{proof}[Proof of \Cref{p:extreme}]
	
	Observe that \eqref{xntk} follows from \eqref{e:eigbound} by symmetry (that is, by negating each of the $x_j$, which leaves the assumption \eqref{e:xi-xj2} satisfied). Thus, it suffices to verify \eqref{e:eigbound}. 
	
	To that end, we first rescale the $x_j$ to make them satisfy the standard Dyson Brownian motion equations \eqref{e:DBM}. More specifically, we rescale time by $k^{1/3}/B^2$ and space by $n^{2/3}/(BL^{1/4})$, defining
\begin{align}
\la_i(t)=\frac{BL^{1/4}}{n^{2/3}}x_i( B^{-2} k^{1/3} t).  
\end{align}

\noindent Since $n^{-1}  = (B L^{1/4} n^{-2/3})^2 \cdot B^{-2} k^{1/3}$ (as $n = L^{3/2} k$), it is quickly verified that $\{\la_i(t)\}_{1\leq i\leq n}$  satisfies \eqref{e:DBM}. Moreover, (since $BL^{1/4} n^{-2/3} \cdot (BL^{3/4} k^{1/3})^{-1} = n^{-1}$ and $BL^{1/4} n^{-2/3} k^{2/3} = BL^{-3/4}$, both as $n = L^{3/2} k$) the initial data satisfies
	\begin{align}
		\label{lambdaij} 
				\la_i(0)-\la_j(0)\geq \frac{j-i}{n}-\frac{BM}{L^{3/4}} , \qquad \text{for each $1 \le i \le j \le n$}.
			\end{align} 

Now set $\delta=BM L^{-3/4}$. Since $BL^{1/4} n^{-2/3} \cdot k^{-2/3} = B L^{-3/4}$, the claim \eqref{e:eigbound} is  equivalent to show that for any fixed $t\leq B^2$, we have with probability $1 - \OO( e^{-(\log n)^2})$ that 
\begin{align}\label{e:extremegao}
\la_1(t)\leq \la_1(0)+C_1' (t+t|\log t|^2+\sqrt{\delta t } + (\log n)\sqrt{t/n}),
\end{align}

\noindent for some constant $C_1' = C_1' (B) > 1$. By translating the $\lambda_i$ if necessary, we may assume that $\lambda_1 (0) = -(2n)^{-1}$. Moreover, by the comparison result given by the first statement in \Cref{l:coupling}, we may assume that the remaining $\lambda_i$ are as large as possible so that \eqref{lambdaij} is satisfied. More specifically, denoting $m=\lceil \delta n\rceil$, we may suppose 
\begin{flalign}
	\label{lambdaj0} 
	\begin{aligned} 
	& \la_1(0)=\la_2(0)=\cdots=\la_m(0)=-(2n)^{-1}=\la_{m+1}(0) = -(2n)^{-1}, \\ 
	& \la_{m+i}(0)-\la_{m+i+1}(0)= n^{-1}, \quad \text{for each $1 \le i \le n-m+1$},
	\end{aligned} 
\end{flalign} 

\noindent where the last constraint can be omitted if $m \ge n$. 

For the special case $m=0$, \eqref{e:extremegao} follows from \Cref{p:uniform}, which in fact implies the slightly stronger estimate: there exists a constant $C_0 = C_0 (\mathsf{T}) > 1$ such that, with probability $1-\OO(e^{-(\log n)^2})$, we have for each $0 \le t \le \mathsf{T}$ that 
\begin{align}
\la_1(t)\leq \la_1(0)+C_0 (t+t|\log t| + (\log n)\sqrt{t/n}).
\end{align}
  
\noindent Next assume that $m \ge 1$. We consider two cases, depending on if $\sqrt{m t/n }\leq  t |\log t|$ or $\sqrt{m t/n }\geq  t |\log t|$. 

First assume that $\sqrt{mt / n} \le t |\log t|$, so that $m/n\leq t|\log t|^2$.
We consider another Dyson Brownian motion with initial data given by
$\widetilde \la_i(0)=m/n-(i-1/2)/n$ for $1\leq i\leq n$. 
The new Dyson Brownian motion satisfies the assumptions in \Cref{p:uniform} (up to translation by $m/n$). Thus that proposition yields, with probability $1- \OO( e^{-(\log n)^2})$, that 
 \begin{align}\begin{split}\label{e:tlla}
\widetilde \la_1(t)+1/(2n) & \leq (m/n)+C_0 (t+t|\log t|+\log n\sqrt{t/n})\\
&\leq C_0 (t+t|\log t|+t|\log t|^2+\log n\sqrt{t/n})\leq 2C_0 (t+t|\log t|^2+\log n\sqrt{t/n})
\end{split}\end{align}
Moreover, by construction (and \eqref{lambdaj0}), the new Dyson Brownian motion satisfies $\widetilde \la_i(0)\geq \la_i(0)$ for $1\leq i\leq n$. By the first statement in \Cref{l:coupling}, we can couple the $\lambda_j$ with the $\widetilde{\lambda}_j$ such that $\widetilde \la_i(t)\geq \la_i(t)$ for $1\leq i\leq n$. The claim \eqref{e:extremegao} then follows \eqref{e:tlla} (with the fact that $\lambda_1 (0) = -(2n)^{-1}$).  

Now assume instead that $\sqrt{m t/n }\geq t |\log t|$, so that $m/n\geq t|\log t|^2$. We consider another Dyson Brownian motion with initial data given by 
\begin{flalign*} 
	\widehat \la_1(0)=\widehat \la_2(0)=\cdots=\widehat \la_{m+1}(0)=r, \quad \text{and} \quad \widehat{\lambda}_i (0) = \lambda_i (0), \quad \text{for $m+2 \le i \le n$},
\end{flalign*} 

\noindent  where 
\begin{flalign}
	\label{r} 
	r=3C_2 (t+\sqrt{(m+1) t/n }+\log n\sqrt{t/n}) -(1/2n).
\end{flalign} 

\noindent and $C_2$ is the maximum of the two constants $C(\mathsf{T})$ from \Cref{p:detercoupling2} and \Cref{p:uniform}.  Then the first statement in \Cref{l:coupling} yields a coupling between the $\lambda_j$ and $\widehat{\lambda}_j$ such that $\lambda_1 (t) \le \widehat{\lambda}_1 (t)$. Hence, to verify \eqref{e:extremegao} (with the fact that $\lambda_1 (0) = (-2n)^{-1}$), it suffices to show with probability at least $1 - \OO(e^{-(\log n)^2})$ that 
 \begin{align}
 	\label{lambda10} 
 \widehat{\lambda}_1 (t) + (2n)^{-1} \le C_1' (t+t|\log t|^2+\sqrt{\delta t } + (\log n)\sqrt{t/n}).
 \end{align}

To verify \eqref{lambda10}, we will first show that there is likely a gap between $\widehat{\lambda}_{m+1} (s)$ and $\widehat{\lambda}_{m+2} (s)$ for $s \in [0, t]$, which will require  both an upper on the $\widehat{\lambda}_i (s)$ for $m+2 \le i \le n$ and a lower bound on the $\widehat{\lambda}_i (s)$ for $1 \le i \le m+1$. We begin with the upper bound, which will proceed by comparing $\widehat{\lambda}_i (s)$ to Dyson Brownian motion $\widehat{y}_i (s)$ with initial data $\widehat{y}_i (0) = -(2i+1) (2n)^{-1}$ for each $1 \le i \le n$; we claim that it is possible to couple $\widehat{\lambda}_i$ with $\widehat{y}_i$ such that $\widehat{\lambda}_i (s) \le \widehat{y}_{i-m-1} (s)$ for each $s \in [0, t]$ and $m+2 \le i \le n$. 

To show this, we first compare $\{\widehat \la_i(s)\}_{1\leq i\leq n}$ with Dyson Brownian motion with initial data $ y_1(0)=y_2(0)=\cdots= y_{m+1}(0)=\infty$, and $ y_i(0)=\la_i(0) = (2m-2i+1) (2n)^{-1}$ for $m+2 \le i \le n$. Then, $y_i (0) \ge \widehat{\lambda}_i (0)$ for each $1 \le i \le n$, and so the first part of \Cref{l:coupling} yields a coupling between the $\widehat{\lambda}_i$ and $y_i$ such that, with probability $1-\OO(e^{-(\log n)^2})$, we have for each $s \in [0, t]$ and $1 \le i \le n$ that $\widehat{\lambda}_i (s) \le y_i (s)$. Let $\breve{y}_i (s)$ denote Dyson Brownian motion with initial data $\breve{y}_i (0) = y_{i+m+1} (0) = -(2i+1)(2n)^{-1}$ for $1 \le i \le n - m - 1$ and $\breve{y}_i (0) = -\infty$ for $n-m \le i \le n$. Then, the above coupling between $\widehat{\lambda}$ and $\widehat{y}$ yields $\widehat{\lambda}_i (s) \le \breve{y}_{i-m-1} (s)$ for each $m+2 \le i \le n$. Since $\breve{y}_i (0) \le \widehat{y}_i (0)$ for each $1 \le i \le n$, the first part of \Cref{l:coupling} yields a coupling between $\breve{y}_i (s)$ and $\widehat{y}_i (s)$ such that $\breve{y}_i (s) \le \widehat{y}_i (s)$ for each $s \in [0, t]$ and $1 \le i \le n$. Combining these two couplings confirms the existence of one between $\widehat{\lambda}_i (s)$ and $\widehat{y}_i (s)$ such that $\widehat{\lambda}_i (s) \le \widehat{y}_{i-m-1} (s)$ for each $s \in [0, t]$ and $m+2 \le i \le n$. Hence, with probability $1 - \OO(e^{-(\log n)^2})$, we have for each $s \in [0, t]$ and $m+2 \le i \le n$ that
\begin{align}\label{e:tlafirst}
	\begin{aligned} 
\widehat \la_i (s)\leq \widehat{y}_{i-m-1} (s)& \leq -1/(2n)+C_2 (t+t|\log t|+\log n\sqrt{t/n})-(i-m-1)/n,
\end{aligned}
\end{align}

\noindent where the last inequality is from \Cref{p:uniform} (applied with the $\lambda_i$ there given by $\widehat{y}_i$ here, translated by $mn^{-1}$) and the fact that $\widehat{y}_1 (0) = -3 (2n)^{-1}$. 

To lower bound the $\widehat{\lambda}_i (s)$, we next compare $\{\widehat \la_i(s)\}_{1\leq i\leq n}$ with Dyson Brownian motion starting from $ z_1(0)=z_2(0)=\cdots= z_{m+1}(0)=r$, and $ z_i(0)=-\infty$ for $i\geq m+2$. Since $z_i (0) \le \widehat{\lambda}_i (0)$ for each $1 \le i \le n$, the first part of \Cref{l:coupling} yields a coupling between the $z_i$ and $\widehat{\lambda}_i$ such that, with probability $1-\OO(e^{-(\log n)^2})$, we have for each $s \in [0, t]$ and $1 \le i \le m + 1$ that
\begin{align}\begin{split}\label{e:tlasecond}
\widehat \la_i(s)
\geq z_i(s) & \geq r-C_2(\sqrt{(m+1) t/n}+\log n\sqrt{t/n})-1/(2n)\\
&\geq 2C_2(t+\sqrt{(m+1)t/n}+\log n\sqrt{t/n})-1/(2n)\geq
 2C_2(t+t|\log t|+\log n\sqrt{t/n})-(1/2n),
\end{split}\end{align}

\noindent where the second inequality is from \Cref{p:detercoupling2} (translated by $r$); the third inequality uses that $r=3C_2(t+\sqrt{(m+1) t/n }+\log n\sqrt{t/n}) $; and the last inequality uses that $m/n\geq t|\log t|^2$. 

Comparing \eqref{e:tlafirst} and \eqref{e:tlasecond}, we deduce that with probability at least $1 - \OO(e^{-(\log n)^2})$ that, for each $s \in [0, t]$ and $1 \le i \le m + 1 < m+2 \le j \le n$,
\begin{align}\begin{split}\label{e:gapbb2}
\widehat{\lambda}_i(s)-\widehat{\lambda}_j(s) &\geq C_2(t+t|\log t|+\log n\sqrt{t/n})+\frac{j-m-1}{n}.
\end{split}\end{align}

\noindent In what follows, we restrict to the event $\mathscr{E}$ on which the bound \eqref{e:gapbb2} holds, observing that $\mathbb{P} [\mathscr{E}] \ge 1 - \OO(e^{-(\log n)^2})$. 

On the event $\mathscr{E}$, the stochastic differential equation \eqref{e:DBM} implies, for any $s \in [0, t]$ and $1 \le i \le m +1$, that
\begin{align}\begin{split}\label{e:tlai}
\rd \widehat \la_i(s)
&\leq \left(\frac{2}{\beta n}\right)^{1/2}\rd B_i(s)+\frac{1}{n}\sum_{j: j\neq i,\atop j\leq m+1}\frac{1}{\widehat \la_i(s)-\widehat \la_j(s)} + \displaystyle\sum_{m+2}^n \frac{1}{C_2(t+t|\log t|)n+j-m-1}\\
&\leq \left(\frac{2}{\beta n}\right)^{1/2}\rd B_i(s)+\frac{1}{n}\sum_{j: j\neq i,\atop j\leq m+1}\frac{1}{\widehat \la_i(s)-\widehat \la_j(s)}+\int_{C_2(t+t|\log t|)}^{C_2(t+t|\log t|)+1}\frac{\rd x}{x}\\
&\leq \left(\frac{2}{\beta n}\right)^{1/2}\rd B_i(s)+\frac{1}{n}\sum_{j: j\neq i,\atop j\leq m+1}\frac{1}{\widehat \la_i(s)-\widehat \la_j(s)}+\log C_2'+ C_2'|\log t|,
\end{split}\end{align} 

\noindent for some constant $C_2' = C_2' (\mathsf{T}) > 1$. Here, the first line follows from \eqref{e:DBM} and \eqref{e:gapbb2}; the second from the fact that $x^{-1}$ is decreasing; and the third from performing the integration and using the facts that $\log \big( C_2 (t + t|\log t|) + 1 \big) \le \log C_2 + \log \big( t + t |\log t| + 1 \big) \le \log C_2 + \mathsf{T} + \mathsf{T} |\log t|$ and $\big| \log \big( C_2 (t + t |\log t|) \big) \big| \le \log C_2 + |\log t| + \log \big(1 + |\log t| \big) \le \log C_2 + 2 |\log t|$. 

\eqref{e:tlai} is the same stochastic differential equation as the $(m+1)$-particle Dyson Brownian motion starting from a delta mass at $r$ with an extra drift $\log C_2+2C_2|\log t|$. Specifically, letting $w_i (s)$ denote Dyson Brownian motion \eqref{e:DBM} with initial data $w_i (0) = r$ for $1 \le i \le m + 1$, we find from \eqref{e:tlai} that on $\mathscr{E}$ we have $\widehat{\lambda}_i (s) \le w_i (s) + \big( \log C_2' + C_2' |\log t| \big) s$, for each $s \in [0, t]$ and $1 \le i \le m + 1$. Therefore, since $\mathbb{P}[\mathscr{E}] \ge 1 - \OO(e^{-(\log n)^2})$, \Cref{p:detercoupling2} and a union bound yields that, with high probability $1-\OO( e^{-(\log n)^2})$, we have for some constant $C_3 = C_3 (\mathsf{T}) > 1$ that
\begin{align*}
 \widehat \la_1(t) \le w_1 (t) + \big( \log C_2' + C_2' |\log t| \big) t &\leq r+C_2\left(\sqrt{(m+1)t/n}+\frac{(\log n)\sqrt t}{\sqrt n}\right) +(\log C_2'+2C_2|\log t|)t\\
 &\leq C_3 (t+\sqrt{(m+1) t/n }+t|\log t|+\log n\sqrt{t/n})-(1/2n) \\
 &\leq  2 C_3 (t+t|\log t|^2+\sqrt{\delta t+2t/n } + (\log n)\sqrt{t/n}) - (1/2n) \\
&\leq  6 C_3 (t+t|\log t|^2+\sqrt{\delta t } + (\log n)\sqrt{t/n}) - (1/2n).
\end{align*}

\noindent where the third bound follows from \eqref{r}; the fourth from the facts that $\delta n \ge \lceil \delta n \rceil -1 \ge m-1$ and that $t + t |\log t|^2 \ge |\log t|$; and the fifth from the facts that $\sqrt{\delta t + 2tn^{-1}} \le \sqrt{\delta t} + \sqrt{2} \cdot \sqrt{tn^{-1}}$ and that $2^{1/2} \le 3 \log 2 \le 3 \log n$. This implies \eqref{lambda10} and thus the corollary. 
\end{proof}

\bibliography{References.bib}
\bibliographystyle{abbrv}

\end{document}